\documentclass{amsart}

\usepackage{amssymb}
\usepackage{amsmath}
\usepackage{amsthm}
\usepackage{tikz}
\usepackage[linktocpage=true]{hyperref}
\usetikzlibrary{arrows}

\newtheorem{theorem}{Theorem}[section]
\newtheorem{lemma}[theorem]{Lemma}
\newtheorem{proposition}[theorem]{Proposition}
\newtheorem{corollary}[theorem]{Corollary}

\theoremstyle{definition}
\newtheorem{definition}[theorem]{Definition}
\newtheorem{situation}[theorem]{Situation}
\newtheorem{example}[theorem]{Example}
\newtheorem{question}[theorem]{Question}
\newtheorem{construction}[theorem]{Construction}
\newtheorem{claim}[theorem]{Claim}

\theoremstyle{remark}
\newtheorem{remark}[theorem]{Remark}

\newcommand{\sh}{\mathcal}
\newcommand{\modu}{\mathcal}
\newcommand{\obj}{\mathbf}
\newcommand{\cat}{\mathfrak}
\newcommand{\spec}{\mathrm{Spec}\mathop{}}
\newcommand{\ga}{\mathbb{G}_a}
\newcommand{\gm}{\mathbb{G}_m}
\newcommand{\tline}{\ell}
\newcommand{\down}{\flat}
\newcommand{\up}{\sharp}
\newcommand{\kukp}{\ddagger}
\newcommand{\zuzp}{\dagger}
\newcommand{\dsd}{\textup{\textexclamdown}}
\newcommand{\pr}{\operatorname{pr}}

\begin{document}
	
	\title[Marked nodal curves with vector fields]{Marked nodal curves with vector fields}
	
	\author[Adrian Zahariuc]{Adrian Zahariuc}
	\date{}
	\address{Department of Mathematics and Statistics, University of Windsor, 401 Sunset Ave, Windsor, ON, N9B 3P4, Canada}
	\email{adrian.zahariuc@uwindsor.ca}
	
	\maketitle
	
	\begin{abstract}
		We discuss two operations on nodal curves with (logarithmic) vector fields, which resemble the `stabilization' construction in Knudsen's proof that $\overline{\modu M}_{g,n+1}$ is the universal curve over $\overline{\modu M}_{g,n}$. We prove that both operations work in families (commute with base change). We also construct inverse operations under suitable assumptions, which allow us to prove a technical result somewhat similar to Knudsen's. 
		
		As an application, we prove that the Losev--Manin compactification of the space of configurations of $n$ points on ${\mathbb P}^1 \backslash \{0,\infty\}$ modulo scaling degenerates isotrivially to a compactification of the space of configurations of $n$ points on ${\mathbb A}^1$ modulo translation, and the natural group actions fit together globally.
	\end{abstract}
	
	\section{Introduction}\label{section: introduction}
	
	A \emph{logarithmic vector field} on a (connected but possibly reducible, reduced, projective) curve $C$ with only nodal singularities over an algebraically closed field ${\mathbb K}$ is a global section of $\omega_C^\vee$, the dual of the dualizing sheaf. Concretely, using the well-known description of $\omega_C$ in terms of meromorphic differentials and residues, we see that a logarithmic vector field on $C$ corresponds to a global section $v \in \Gamma(T_{\tilde{C}})$, where $\nu:\tilde{C} \to C$ is the normalization of $C$, such that for any node $p \in C$,
	\begin{itemize}
		\item $v$ vanishes at $q_1$ and $q_2$, where $\nu^{-1}(p) = \{q_1,q_2\}$, and
		\item $c_1+c_2 = 0$, where $c_i \in {\mathbb K}$ is the image of $v$ under the natural map
		$$ \Gamma(T_{\tilde{C}}(-q_i)) \to T_{\tilde{C}}(-q_i) \otimes {\sh O}_{\tilde{C}}/{\sh O}_{\tilde C}(-q_i) = T_{q_i,\tilde{C}} \otimes T^\vee_{q_i,\tilde{C}} = {\mathbb K}. $$
	\end{itemize}
	If the arithmetic genus of $C$ is at least $2$, the vector field must be identically $0$ on many irreducible components of $C$ (though not all in general), and there are strong restrictions on what situations can occur. The higher genus case is thus somewhat artificial, though we have decided to include it since many of the constructions below still go through. The adjective `logarithmic' will typically be omitted, and we will write just `field' or `vector field'.
	
	\subsection{Bubbling up in the presence of vector fields}\label{subsection: Bubbling up and bubbling down} One direction in Knudsen's proof that $\overline{\modu M}_{g,n+1}$ is the universal curve over $\overline{\modu M}_{g,n}$ amounts the procedure called stabilization in \cite{[Kn83]}. In the case of a single curve, stabilization works as follows. (We are discussing this procedure out of context; perhaps `prestabilization' would be a better name in our setup.) Given a projective nodal curve $C$ over an algebraically closed field ${\mathbb K}$, with several nonsingular distinct marked points $w_1,\ldots,w_m \in C({\mathbb K})$, imagine inserting a new marked point at $x \in C({\mathbb K})$. If $x$ is singular or $x=w_i$ for some $i$, then we deem this configuration `degenerate'. In this situation, stabilization inserts a new component $\Sigma \simeq {\mathbb P}^1$ at $x$, and produces non-degenerate data $C',w'_1,\ldots,w'_m,x'$, by placing $x'$ (and $w'_i$, if $x=w_i$) on $\Sigma$. (In the main body of the paper, we will instead use different notation, but for now the prime notation is adequate.) Here are two related reasons why stabilization is important:
	\begin{itemize}
		\item it \emph{works in families}: it can be generalized to families of curves in a manner which commutes with base change (and this commutativity satisfies the suitable `cocycle conditions'); and
		\item  iterating the stabilization construction in families suitably intertwined with base changes produces important moduli spaces. For instance, in this way it is possible to obtain $\overline{M}_{0,n}$ (or even $\overline{\modu M}_{g,n}$ given $\overline{\modu M}_g$), or the Fulton-MacPherson compactifications of configuration spaces of curves.
	\end{itemize} 
	
	In this paper, we will discuss a related (now two-step) procedure when a vector field on the nodal curve is given. 
		
	\subsubsection*{Knudsen stabilization with vector fields} The first step of the procedure operates in the same setup as Knudsen's stabilization, with the extra data of a vector field $\phi$ on $C$, which vanishes at $w_1,\ldots,w_m$. The meaning of degenerate is the same as in the usual Knudsen stabilization, and fixing it entails inserting a ${\mathbb P}^1$ component under the same circumstances. It can be shown that $\phi$ lifts uniquely to a vector field $\phi'$ on the curve $C'$ obtained by Knudsen stabilization, if we require $\phi'(w'_i) = 0$ for $i=1,\ldots,m$. (This boils down to the fact that, given $r \in {\mathbb K}^\times$, there exists a unique meromorphic $1$-form on ${\mathbb P}^1$ with poles with residues $r$ at $0$ and $-r$ at $\infty$.)
	
	\subsubsection*{Inflating at zero vector} The second operation is applied after the first, so it presumes milder degeneracy to begin with: the point $x$ is nonsingular and distinct from $w_1,\ldots,w_m$. We consider such a situation `degenerate' if $\phi(x) = 0$. Then we insert a ${\mathbb P}^1$ component at $x$, and fix the degeneracy by placing $x'$ on ${\mathbb P}^1$ and lifting $\phi$ to a field $\phi'$ on $C'$ such that $\phi'(w'_i) = 0$, but  $\phi'(x') \neq 0$.
	
	\begin{remark}\label{remark: first remark}
		Here are a few remarks regarding this construction.
		\begin{enumerate}
			\item Indeed, $x$ behaves differently from $w_1,\ldots,w_m$. In applications, we will get around this by thinking of $x$ as part of a different set of markings $x_1,\ldots,x_n$.
			\item The properties above don't determine $\phi'$ uniquely, although they do determine it up to automorphisms of $C'$ that fix all the rest of the data and the map $C' \to C$ which contracts the new component, if there is a new component. Still, a canonical choice exists in a certain sense. A good analogy is with the following question: in Knudsen stabilization, where do we place $x'$ on the new ${\mathbb P}^1$ component? It seems that all points on ${\mathbb P}^1$ except two are equally good, yet, in \cite[\S2]{[Kn83]}, there is a definite formula for $x'$.
			\item The new field $\phi'$ depends nonlinearly on the old field $\phi$. For instance, $\phi = 0$ everywhere is possible (or, more frequently in applications, $\phi = 0$ on the component containing $x$), so $\phi'(x') \neq 0$ rules out linearity.
		\end{enumerate}
	\end{remark}
	
	The main result concerning these operations, which forms the technical core of the paper and will be stated more precisely and proved in \S\ref{section: Bubbling up and the inductive construction}, is the following.
	
	\begin{theorem}[{Theorems \ref{theorem: summary of first bubbling up} and \ref{theorem: summary of second bubbling up}}]\label{theorem: main theorem abstract}
		`Knudsen stabilization with vector fields' and `Inflating at zero vector' work in families: they can be generalized to families in a manner which commutes with base change.
	\end{theorem}
	
	Knudsen's original stabilization admits an inverse operation called contraction. Similarly, we will show that under suitable circumstances, our two operations admit inverse operations (Theorem \ref{theorem: the main universal curve theorem}), which will allow us to prove a technical result (Theorem \ref{theorem: universal curve theorem simplified statement}) reminiscent of Knudsen's theorem that $\overline{\modu M}_{g,n+1}$ is the universal curve over $\overline{\modu M}_{g,n}$. Very roughly, instead of $n$-marked curves, we will consider curves with a vector field and $m+n$ markings, the first $m$ of which are vanishing points of the vector field and the remaining $n$ give stability; then, the `universal curve' in this situation is isomorphic to the analogous space with $m$ markings of the first kind, and $n+1$ markings of the second kind, by Theorem \ref{theorem: universal curve theorem simplified statement}. 
	
	Due to obvious numerical constraints, the range of meaningful geometric applications of these technicalities is surely limited, though it does include a few very interesting examples. In this paper, we will focus on the applications given in Theorems \ref{theorem: Pn bar as moduli space} and \ref{theorem: isotrivial degeneration theorem} below, and to a much lesser extent the application mentioned in \S \ref{example: examples with initial data: a genus 1 example}, though a few other interesting setups are possible. These applications require and illustrate the technicalities above essentially in full generality.
		
	\subsection{Application: configurations on a line modulo scaling \emph{or} translation}\label{subsection: Configurations on a line modulo scaling or translation} Consider the following open ended (and seemingly unrelated) problems: 
	\begin{enumerate}
		\item compactifying the space of configurations of $n$ not necessarily distinct points on a punctured line modulo scaling; and 
		\item compactifying the space of configurations of $n$ not necessarily distinct points on a line modulo translation. 
	\end{enumerate}
	
	Below are two arguably optimal answers to these problems. We first discuss the constructions concretely over an algebraically closed field ${\mathbb K}$, and then over ${\mathbb Z}$.
	In what follows, ${\mathbb G}_m$ and ${\mathbb G}_a$ will denote the multiplicative, respectively additive group scheme over a ground ring $R \in \{{\mathbb K},{\mathbb Z}\}$ that will be clear from context. Recall that their formal definitions using functors of points are ${\mathbb G}_m(T) = (\Gamma(T,{\sh O}_T^\times), \times)$ and ${\mathbb G}_a(T) = (\Gamma(T,{\sh O}_T), +)$ respectively, and thus their respective underlying schemes are $\mathrm{Spec} \mathop{}R[x,y]/(xy-1)$ and ${\mathbb A}^1_R$.
	
	\subsubsection*{The Losev--Manin space} An $n$-marked Losev-Manin string is a genus $0$ nodal curve $C$ over ${\mathbb K}$ whose dual graph is a chain, with $n+2$ smooth points $p_0,p_\infty,x_1,\ldots,x_n \in C({\mathbb K})$, such that $p_0$ and $p_\infty$ live on components at opposite ends of the chain, $p_0 \neq p_\infty$ (if there is just one component), $x_i \neq p_0,p_\infty$ for all $i$, and all components contain at least one of the points $x_1,\ldots,x_n$. (However, $x_i=x_j$ is allowed. The concentric circles in Figure \ref{figure: Losev-Manin string} represent overlapping markings.)
	\begin{figure}[h]
		\begin{center}
			\begin{tikzpicture}[scale = 0.7]
				\draw [thick] (-1.2,-1) -- (2,1); 
				\draw [thick] (0,1) -- (3,-1);
				\draw [thick] (1.2,-1) -- (4,1);
				\draw [thick] (2.4,1) -- (5,-1);
				\draw [thick] (3.6,-1) -- (6,1);
				\draw [thick] (4.8,1) -- (7,-1);
				\filldraw[fill=white] (-0.4,-0.5) circle (3pt) node[anchor=north] {$0$};
				\filldraw[fill=white] (6.45,-0.5) circle (3pt) node[anchor=west] {$\infty$};
				\fill[black] (5.68,0.2) circle (3pt);
				\fill[black] (4.48,-0.27) circle (3pt);
				\fill[black] (4,-0.68) circle (3pt);
				\fill[black] (3.6,0.71) circle (3pt);
				\draw[black] (3.6,0.71) circle (5pt);
				\draw[black] (3.6,0.71) circle (7pt);
				\fill[black] (2.72,0.75) circle (3pt);
				\fill[black] (1.5,0) circle (3pt);
				\fill[black] (0.08,-0.2) circle (3pt);
				\fill[black] (0.71,0.2) circle (3pt);
				\draw[black] (0.08,-0.2) circle (5pt);
				\draw[black] (0.71,0.2) circle (5pt);
			\end{tikzpicture}
			\caption{A $12$-marked Losev-Manin string.}
			\label{figure: Losev-Manin string}
		\end{center}
	\end{figure}
	Note that there is a unique ${\mathbb G}_m$-action on $C$ which fixes $p_0$ and $p_\infty$, and acts with weight $1$ on each irreducible component.
	
	The Losev-Manin space $\overline{L}_{n,{\mathbb K}}$ is a variety which parametrizes the $n$-marked Losev-Manin strings. It is shown in \cite{[LM00]} that it is a smooth irreducible toric projective variety of dimension $n-1$. There is an action of ${\mathbb G}_m^n$ on $\overline{L}_{n,{\mathbb K}}$ such that, on ${\mathbb K}$-points,
	$$ (c_1,\ldots,c_n) \cdot (C,p_0,p_\infty,x_1,\ldots,x_n) =  (C,p_0,p_\infty,c_1 \cdot x_1,\ldots,c_n \cdot x_n), $$
	where $\cdot$ on the right hand side is the ${\mathbb G}_m$-action on $C$ above. The action restricts to a trivial one on the diagonal ${\mathbb G}_m \hookrightarrow {\mathbb G}_m^n$, and thus restricts to the toric ${\mathbb G}_m^{n-1}$-action on $\{1\} \times {\mathbb G}_m^{n-1} \subset {\mathbb G}_m^n$. Moreover, the open stratum $ L_{n,{\mathbb K}} \simeq {\mathbb G}_m^n/{\mathbb G}_m \hookrightarrow \overline{L}_{n,{\mathbb K}} $ where the string has a single component can be thought of as the space of configurations of $n$ points (not necessarily distinct) on ${\mathbb P}^1 \backslash \{0,\infty\}$ modulo scaling.
	
	The construction of the Losev-Manin space in \cite[\S1.3 and \S2.1]{[LM00]} clearly goes through over $\spec {\mathbb Z}$. From now on, we will typically denote this finite type scheme over ${\mathbb Z}$ by $\overline{L}_n$, and the projective variety above is then $\overline{L}_{n, {\mathbb K}} = \spec {\mathbb K} \times_{\spec {\mathbb Z}} \overline{L}_n$.
	
	\subsubsection*{A modular equivariant compactification of ${\mathbb G}_a^n/{\mathbb G}_a$} A remarkable compactification of the space of configurations of $n$ \emph{distinct} points on ${\mathbb A}^1$ modulo translation is the moduli space $\smash{ \overline{Q}_n }$ of `stable scaled marked curves' constructed as a projective variety by Ma'u and Woodward in \cite{[MW10]}, after Ziltener introduced it in a symplectic setting \cite{[Zi06], [Zi14]}. The moduli space $\smash{ \overline{Q}_n }$ plays an important role in the context of gauged stable maps \cite{[Wo15], [GSW17], [GSW18]}. 
	
	We will construct a related space $\overline{P}_n$ which compactifies the space of configurations of $n$ not necessarily distinct points on ${\mathbb A}^1$ modulo translation. This space will turn out to be an equivariant compactification of ${\mathbb G}_a^{n-1}$ in the sense of \cite{[HT99]}. Besides this additional feature and simplicity, Problem \ref{problem: degeneration of primitive linear system} (discussed at the end of the paper) suggests that the version in which the points may coincide is the suitable object to consider in certain applications.
	
	An \emph{$n$-marked ${\mathbb G}_a$-rational tree} is a connected projective curve $C$ over ${\mathbb K}$  of arithmetic genus $0$ with at worst nodal singularities, with a ${\mathbb G}_a$-action which operates trivially or `by translation' on each irreducible component of $C$ (i.e. $a \cdot [X:Y] = [X+aY:Y]$ in suitable coordinates), and $n+1$ nonsingular points $p_\infty,x_1,\ldots,x_n \in C({\mathbb K})$, such that $p_\infty$ is fixed by the ${\mathbb G}_a$-action on $C$, but $x_1,\ldots,x_n$ are not. (The condition that ${\mathbb G}_a$ acts trivially or by translation on each component is actually automatic in characteristic $0$.) The $n$-marked ${\mathbb G}_a$-rational tree is \emph{stable} if any irreducible component of $C$ which doesn't contain any of the points $x_1,\ldots,x_n$, either intersects at leasts $3$ other irreducible components of $C$, or contains $p_\infty$ and intersects at least $2$ other irreducible components of $C$. 
	\begin{figure}[h]
		\begin{center}
			\begin{tikzpicture}[scale = 0.7]
				\draw [thick] (-2,0) -- (7,0); 
				
				\filldraw[fill=white] (6.5,0) circle (3pt) node[anchor=south] {$\infty$};
				\fill[black] (3.6,1) circle (3pt);
				\draw[black] (3.6,1) circle (5pt);
				\draw[black] (3.6,1) circle (7pt);
				
				\draw [thick] (3.6,0.3) -- (3.6,2.3);
				\draw [thick] (4.3,0.3) -- (4.3,2.3);
				
				\fill[black] (4.3,0.6) circle (3pt);
				\fill[black] (4.3,1.4) circle (3pt);
				
				\draw[thick] (3.2,2) -- (6.7,2);
				\draw [thick] (5.3,-0.3) -- (5.3,2.3);
				\draw [thick] (5,1) -- (6.7,1);
				\fill[black] (6.2,1) circle (3pt);
				
				\draw [thick] (0.4,1) -- (2.8,1);
				\draw[thick] (0.4,2) -- (2.8,2);
				\draw [thick] (1.3,-0.3) -- (1.3,2.3);
				
				\fill[black] (2,2) circle (3pt);
				\fill[black] (2,1) circle (3pt);
				\draw[black] (2,1) circle (5pt);
				
				\draw [thick] (0,-0.3) -- (0,2.3);
				\fill[black] (0,1.4) circle (3pt);
				
				\draw [thick] (-1,-0.3) -- (-1,2.3);
				\fill[black] (-1,1) circle (3pt);
				\fill[black] (-1,2) circle (3pt);
				\draw[black] (-1,2) circle (5pt);
				
			\end{tikzpicture}
			\caption{A stable $13$-marked ${\mathbb G}_a$-rational tree.}
			\label{figure: stable n-marked Ga-rational tree}
		\end{center}
	\end{figure}
	Note that all marked points $x_1,\ldots,x_n$ live on tail components of $C$ (leaves of the dual tree), and that the action is trivial on all irreducible components of $C$ except the tails if $C$ is reducible.
	
	There exists an (irreducible) projective ${\mathbb G}_a^{n-1}$-variety $\overline{P}_{n,{\mathbb K}}$ which parametrizes stable $n$-marked ${\mathbb G}_a$-rational trees, cf. Remark \ref{remark: Pn bar as a moduli space low tech}. We use $G$-variety in the sense of \cite[Definition 2.1]{[HT99]}. The ${\mathbb G}_a^n$-action on $\overline{P}_{n,{\mathbb K}}$ comes from the ${\mathbb G}_a$-action on the stable $n$-marked ${\mathbb G}_a$-rational trees in precisely the same way the ${\mathbb G}_m^n$-action on $\overline{L}_{n,{\mathbb K}}$ comes from the ${\mathbb G}_m$-action on $n$-marked Losev-Manin strings. Again, we may restrict to $\{0\} \times {\mathbb G}_a^{n-1} \subset {\mathbb G}_a^n$ to mimic modding out the trivial diagonal action. Similarly to $\overline{Q}_n$, $\overline{P}_{n,{\mathbb K}}$ is mildly singular for $n \geq 4$. The open stratum 
	$ P_{n,{\mathbb K}} \simeq {\mathbb G}_a^n/{\mathbb G}_a \hookrightarrow \overline{P}_{n,{\mathbb K}} $ 
	where the rational trees have a single component is the space of configurations of $n$ points on ${\mathbb A}^1$ modulo translation.
	
	A similar space can be defined over $\spec {\mathbb Z}$, and one of our main goals is to show that this space is a fine moduli space (as opposed to accomplishing just some naive parametrization). In this context, the additional structure on the curves will be given by a vector field rather than a ${\mathbb G}_a$-action (though please see Remark \ref{remark: Pn bar as a moduli space low tech}).
	
	\begin{theorem}\label{theorem: Pn bar as moduli space}
		Let $F$ be the functor which associates to each noetherian scheme $S$ the set of all collections of data as follows, modulo isomorphism: 
		\begin{itemize}
			\item a genus $0$ prestable curve $\pi:C \to S$;
			\item smooth sections $x_1,\ldots,x_n,p_\infty:S \to C$ of $\pi$ (possibly not disjoint); and 
			\item an ${\sh O}_C$-module homomorphism $\phi:\omega_{C/S} \to {\sh O}_C$,
		\end{itemize}  
		such that: 
		\begin{enumerate}
			\item $\phi$ factors through the inclusion ${\sh O}_C(-2p_\infty(S)) \to {\sh O}_C$;
			\item $x_i^*\phi:x_i^*\omega_{C/S} \to {\sh O}_S$ is an isomorphism for $i=1,\ldots,n$; and
			\item the natural stability condition holds: for any geometric point $\overline{s} \to S$, 
			\begin{enumerate}
				\item with the possible exception of the component which contains $p_{\infty,\overline{s}}$, no irreducible component of $C_{\overline{s}}$ intersects exactly two other components;
				\item any irreducible component of $C_{\overline{s}}$ which intersects exactly one other irreducible component contains at least one of the points $x_{1,\overline{s}},\ldots,x_{n,\overline{s}}$ but not the point $p_{\infty,\overline{s}}$. 
			\end{enumerate}
		\end{enumerate}  
		Then $F$ is represented by a projective local complete intersection flat geometrically integral scheme (over $\spec {\mathbb Z}$).
	\end{theorem}
	
	By geometrically integral, we simply mean that all geometric fibers are integral (equivalently, all fibers are geometrically integral in the sense of \cite[\href{https://stacks.math.columbia.edu/tag/020H}{Tag 020H}]{[stacks]}). 
	
	From now on, $\overline{P}_n$ will denote the moduli space in Theorem \ref{theorem: Pn bar as moduli space}, and the projective variety above will turn out to be $\overline{P}_{n,{\mathbb K}} = \spec {\mathbb K} \times_{\spec {\mathbb Z}} \overline{P}_n$.
	
	\subsubsection*{Relation between $\overline{L}_n$ and $\overline{P}_n$} Our main result is that the two compactification problems discussed above are in fact related (or at least, the proposed answers are): we will show that $\overline{L}_n$ degenerates to $\overline{P}_n$, and the actions fit together globally. 
	
	First, let's review the fact that ${\mathbb G}_m$ degenerates isotrivially to ${\mathbb G}_a$, e.g. \cite[3.1]{[KM78]}. Consider the commutative cocommutative Hopf ${\mathbb Z}[t]$-algebra $H = {\mathbb Z}[t,x]_{1+tx}$ with the structure described below, where the maps are always uniquely determined by the property on the right and the requirement that they are ${\mathbb Z}[t]$-algebra homomorphisms, and $x$ is always the element $x \in H$.
	\begin{center}
		\begin{tabular}{rll}
			(multiplication) & $H \otimes_{{\mathbb Z}[t]}H \to H$  & $x \otimes 1\mapsto x$ and $1 \otimes x \mapsto x$ \\
			(comultiplication) & $H \to H \otimes_{{\mathbb Z}[t]}H$ & $x \mapsto x \otimes 1 + 1 \otimes x + t x \otimes x$ \\
			(unit) & ${\mathbb Z}[t] \to H$ &  \\
			(counit) & $H \to {\mathbb Z}[t]$ & $x \mapsto 0$ \\
			(antipode) & $H \to H$ & $x \mapsto \displaystyle -\frac{x}{1+tx}$ \\
		\end{tabular} 
	\end{center}
	The verification of the axioms is a tedious exercise left to the patient reader. 
	
	Recall that for any ring $R$, $R[y,y^{-1}]$ has a natural $R$-Hopf algebra structure with comultiplication characterized by $y \mapsto y \otimes y$. Note that
	\begin{equation}\label{equation: Hopf algebra isomorphism}
		{\mathbb Z}[t,t^{-1}] \otimes_{{\mathbb Z}[t]} H = {\mathbb Z}[t,t^{-1},x]_{1+tx} \simeq {\mathbb Z}[t,t^{-1},y,y^{-1}]
	\end{equation}
	as ${\mathbb Z}[t,t^{-1}]$-Hopf algebras. Indeed, the isomorphism is given by $y \mapsto 1+tx$. 
	
	Then $G = \spec H$ with projection $\gamma:G \to \spec {\mathbb Z}[t]$ via the unit in $H$ is a flat group scheme over $\spec {\mathbb Z}[t]$. Note that $G_{{\mathbb Z}[t,t^{-1}]} \simeq {\mathbb G}_{m,{\mathbb Z}[t,t^{-1}]}$ by \eqref{equation: Hopf algebra isomorphism}, and that
	\begin{equation*}
		G_{\overline{z}} \simeq 
		\begin{cases}
			{\mathbb G}_{m,{\overline{z}}} & \text{if $t \notin z$,} \\
			{\mathbb G}_{a,{\overline{z}}} & \text{if $t \in z$,}
		\end{cases}
	\end{equation*}
	for any geometric point $\overline{z} \to \spec {\mathbb Z}[t]$ with the corresponding usual point denoted by $z \in \spec {\mathbb Z}[t]$. We write $G^k_{{\mathbb Z}[t]} = \underbrace{G \times_{{\mathbb Z}[t]} \cdots \times_{{\mathbb Z}[t]} G}_{\text{$k$ copies of $G$}}$  for any integer $k \geq 0$.
	
	\begin{theorem}\label{theorem: isotrivial degeneration theorem} For any positive integer $n$, there exists a flat projective geometrically integral local complete intersection morphism $\xi:X \to \spec {\mathbb Z}[t]$, and an action of $G^n_{\mathbb Z[t]}$ on $X$ over $\spec {\mathbb Z}[t]$ such that for any geometric point $\overline{z} \to \spec {\mathbb Z}[t]$ (with the corresponding point denoted by $z \in \spec {\mathbb Z}[t]$),
		\begin{itemize} 
			\item if $t \notin z$, then $X_{\overline{z}}$ is isomorphic to $\overline{L}_{n,\overline{z}}$, and the induced action of $G_{\overline{z}}^n$ on $X_{\overline{z}}$ is isomorphic to the canonical action of ${\mathbb G}_{m,{\overline{z}}}^n$ on $\overline{L}_{n,{\overline{z}}}$;
			\item if $t \in z$, then $X_{\overline{z}}$ is isomorphic to $\overline{P}_{n,\overline{z}}$, and the induced action of $G_{\overline{z}}^n$ on $X_{\overline{z}}$ is isomorphic to the canonical action of ${\mathbb G}_{a,{\overline{z}}}^n$ on $\overline{P}_{n,{\overline{z}}}$.
		\end{itemize}
	\end{theorem}
	
	\begin{example}\label{example: the main degeneration for n=3}
		Recall that $\overline{L}_{3,{\mathbb C}}$ is the blowup of ${\mathbb P}^2$ at $3$ non-collinear points. It can be checked that $\overline{P}_{3,{\mathbb C}}$ is the blowup of ${\mathbb P}^2$ at $3$ collinear points. 
		\begin{figure}[h]
			\begin{center}
				\begin{tikzpicture}[scale=1]
					\draw (0,0) circle (1.6); 
					\node at (0,-2) {$\overline{L}_{3,{\mathbb C}}$};
					
					\draw (-1,1) to node[below] {$E_1$} (1,1);
					\draw (-1,-1) to node[above] {$F_1$}(1,-1);
					\draw (-0.3,1.3) to node[right] {$F_2$} (-1.3,-0.3);
					\draw (-0.3,-1.3) to node[right] {$E_3$} (-1.3,0.3);
					\draw (0.3,1.3) to node[left] {$F_3$} (1.3,-0.3);
					\draw (0.3,-1.3) to node[left] {$E_2$} (1.3,0.3);
					
					\def\lxa{-0.8}; \def\lya{1}; \fill (\lxa,\lya) circle (1pt);
					\def\lxb{0.8}; \def\lyb{-1}; \fill (\lxb,\lyb) circle (1pt);
					\def\lxc{-1.12}; \def\lyc{0}; \fill (\lxc,\lyc) circle (1pt);
					\def\lxd{1}; \def\lyd{0.5}; \fill (\lxd,\lyd) circle (1pt);

					\draw (6,0) circle (1.6);
					\node at (6,-2) {$\overline{P}_{3,{\mathbb C}}$};
					
					\draw (4.7,0) to (7.3,0) node[below] {$F$};
					
					\draw (5,-1) to (5,1) node[right] {$E_1$};
					\draw (6,-1) to (6,1) node[right] {$E_2$};
					\draw (7,1) to (7,-1) node[left] {$E_3$};
					
					\def\pxa{6.5}; \def\pya{0}; \fill (\pxa,\pya) circle (1pt);
					\def\pxb{7.3}; \def\pyb{0.5}; \fill (\pxb,\pyb) circle (1pt);
					\def\pxc{5}; \def\pyc{0.5}; \fill (\pxc,\pyc) circle (1pt);
					\def\pxd{5}; \def\pyd{0}; \fill (\pxd,\pyd) circle (1pt);
					
					\def\aa{1.5}; \def\ee{0.5}
					\draw (\aa,2.5+\ee) -- (\aa, 1+\ee);
					\filldraw[fill=white] (\aa,2.3+\ee) circle (2pt) node[anchor=east] {$0$};	
					\fill[black] (\aa,2.05+\ee) circle (2pt) node[anchor=west] {$1$};
					\fill[black] (\aa,1.75+\ee) circle (2pt) node[anchor=west] {$2$};
					\fill[black] (\aa,1.45+\ee) circle (2pt) node[anchor=west] {$3$};
					\filldraw[fill=white] (\aa,1.2+\ee) circle (2pt) node[anchor=east] {$\infty$};
					
					\draw [dashed] (\aa,\ee +0.9) to (\lxd,\lyd);
					
					\def\aa{7.5}; \def\ee{3};
					\draw (\aa,-1.5+\ee) -- (\aa, 0+\ee);
					\fill[black] (\aa,-0.2+\ee) circle (2pt) node[anchor=west] {$1$};
					\fill[black] (\aa,-0.6+\ee) circle (2pt) node[anchor=west] {$2$};
					\fill[black] (\aa,-1.0+\ee) circle (2pt) node[anchor=west] {$3$};
					\filldraw[fill=white] (\aa,-1.3+\ee) circle (2pt) node[anchor=east] {$\infty$};
					
					\draw [dashed] (\aa,\ee -1.6) to (\pxb,\pyb);
					
					\def\aa{-3}; \def\ee{0.5};
					\draw (\aa,1+\ee) -- (\aa+1,2.5+\ee);
					\draw (\aa+0.5,2.5+\ee) -- (\aa+1.5,1+\ee); 
					\filldraw[fill=white] (\aa + 0.2,1.3+\ee) circle (2pt) node[anchor=east] {$0$};
					\filldraw[fill=white] (\aa + 1.3,1.3+\ee) circle (2pt) node[anchor=west] {$\infty$};
					\fill[black] (\aa+0.4,1.6+\ee) circle (2pt) node[anchor=east] {$2$};
					\fill[black] (\aa+0.6,1.9+\ee) circle (2pt) node[anchor=east] {$3$};
					\fill[black] (\aa+1.03,1.7+\ee) circle (2pt) node[anchor=west] {$1$};
					
					\draw [dashed] (\aa+0.75,\ee +0.9) to (\lxa,\lya);
					
					\def\bb{0.3}; \def\aa{3.4}; \def\ee{3};
					\draw (\aa+\bb,-1.5+\ee) -- (\aa+\bb, 0+\ee);
					\filldraw[fill=white] (\aa+\bb,-1.3+\ee) circle (2pt) node[anchor=east] {$\infty$};
					\draw (\aa+\bb - 0.3, -1+\ee) -- (\aa+\bb + 1.2, -1+\ee);
					\draw (\aa+\bb - 0.3, -0.3+\ee) -- (\aa+\bb + 1.2, -0.3+\ee);
					\fill[black] (\aa+\bb+0.3,-0.3+\ee) circle (2pt) node[anchor=north] {$2$};
					\fill[black] (\aa+\bb+0.8,-0.3+\ee) circle (2pt) node[anchor=north] {$3$};
					\fill[black] (\aa+\bb+0.5,-1+\ee) circle (2pt) node[anchor=north] {$1$};
					
					\draw [dashed] (\aa+1,\ee -1.5) to (\pxc,\pyc);
					
					\def\aa{0.5}; \def\ee{-4.3};
					\draw (\aa,1+\ee) -- (\aa+1,2.5+\ee);
					\draw (\aa+0.5,2.5+\ee) -- (\aa+1.5,1+\ee); 
					\filldraw[fill=white] (\aa + 0.2,1.3+\ee) circle (2pt) node[anchor=east] {$0$};
					\filldraw[fill=white] (\aa + 1.3,1.3+\ee) circle (2pt) node[anchor=west] {$\infty$};
					\fill[black] (\aa+0.47,1.7+\ee) circle (2pt) node[anchor=east] {$1$};
					\fill[black] (\aa+0.9,1.9+\ee) circle (2pt) node[anchor=west] {$2$};
					\fill[black] (\aa+1.1,1.6+\ee) circle (2pt) node[anchor=west] {$3$};
					
					\draw [dashed] (\aa+0.75,\ee +2.5) to (\lxb,\lyb);
					
					\def\bb{1.4}; \def\aa{5.5}; \def\ee{-2};
					\draw (\aa+\bb,-1.5+\ee) -- (\aa+\bb, 0+\ee);
					\filldraw[fill=white] (\aa+\bb,-1.3+\ee) circle (2pt) node[anchor=east] {$\infty$};
					\draw (\aa+\bb - 0.3, -1+\ee) -- (\aa+\bb + 1.2, -1+\ee);
					\draw (\aa+\bb - 0.3, -0.6+\ee) -- (\aa+\bb + 1.2, -0.6+\ee);
					\draw (\aa+\bb - 0.3, -0.2+\ee) -- (\aa+\bb + 1.2, -0.2+\ee);
					\fill[black] (\aa+\bb+0.5,-0.2+\ee) circle (2pt) node[anchor=south] {$1$};
					\fill[black] (\aa+\bb+0.5,-0.6+\ee) circle (2pt) node[anchor=west] {$2$};
					\fill[black] (\aa+\bb+0.5,-1+\ee) circle (2pt) node[anchor=north] {$3$};
					\draw [dashed] (\aa+2,\ee +0.1) to (\pxa,\pya);
					
					\def\aa{-3}; \def\ee{-4};
					\draw (\aa,1+\ee) -- (\aa+1,2.5+\ee);
					\draw (\aa+0.5,2.5+\ee) -- (\aa+1.5,1+\ee); 
					\draw (\aa+1,1+\ee) -- (\aa+2,2.5+\ee);
					\filldraw[fill=white] (\aa + 0.2,1.3+\ee) circle (2pt) node[anchor=east] {$0$};
					\filldraw[fill=white] (\aa + 1.8,2.2+\ee) circle (2pt) node[anchor=west] {$\infty$};
					\fill[black] (\aa+0.47,1.7+\ee) circle (2pt) node[anchor=east] {$2$};
					\fill[black] (\aa+1.03,1.7+\ee) circle (2pt) node[anchor=west] {$1$};
					\fill[black] (\aa+1.47,1.7+\ee) circle (2pt) node[anchor=west] {$3$};
					
					\draw [dashed] (\aa+1.25,\ee +2.5) to (\lxc,\lyc);
					
					\def\bb{0.5}; \def\aa{3.4}; \def\ee{-2};
					\draw (\aa+\bb,-1.5+\ee) -- (\aa+\bb, 0+\ee);
					\filldraw[fill=white] (\aa+\bb,-1.3+\ee) circle (2pt) node[anchor=east] {$\infty$};
					\draw (\aa+\bb - 0.3, -1+\ee) -- (\aa+\bb + 1.2, -1+\ee);
					\draw (\aa+\bb - 0.3, -0.3+\ee) -- (\aa+\bb + 1.2, -0.3+\ee);
					\draw (\aa+\bb+0.3, -0.8+\ee) -- (\aa+\bb+0.3, -0.1+\ee);
					\fill[black] (\aa+\bb+0.3,-0.6+\ee) circle (2pt) node[anchor=east] {$2$};
					\draw (\aa+\bb+0.8, -0.8+\ee) -- (\aa+\bb+0.8, -0.1+\ee);
					\fill[black] (\aa+\bb+0.8,-0.6+\ee) circle (2pt) node[anchor=west] {$3$};
					\fill[black] (\aa+\bb+0.5,-1+\ee) circle (2pt) node[anchor=north] {$1$};	
					
					\draw [dashed] (\aa+1,\ee +0.1) to (\pxd,\pyd);
					
					\node at (3,0) {degenerates to};
					
				\end{tikzpicture}
				\caption{The degeneration of $\overline{L}_{3,{\mathbb C}}$ to $\overline{P}_{3,{\mathbb C}}$.}
				\label{figure: case n=3 example}
			\end{center}
		\end{figure}
		There exist $p_1(t),p_2(t),p_3(t) \in {\mathbb P}^2$ ($t \in {\mathbb C}$), which are collinear if and only if $t = 0$, and such that the blowup of ${\mathbb A}^1 \times {\mathbb P}^2$ at the union of the (images of the) $3$ sections $t \mapsto (t,p_i(t))$ can be identified with $X_{\mathbb C}$, with $X$ as in Theorem \ref{theorem: isotrivial degeneration theorem} for $n=3$. Let $E_1,E_2,E_3$ be the exceptional curves when ${\mathbb P}^2$ is blown up at $p_1(t),p_2(t),p_3(t)$, and let $F_i$ be the proper transform of the line through $p_j(t)$ and $p_k(t)$, for any permutation $(i,j,k)$ of $(1,2,3)$. (It would probably be more correct to write $E_{i,t}$ and $F_{i,t}$ instead of $E_i$ and $F_i$, but we hope this abuse of notation will not lead to confusion.) When $t=0$, $F_1=F_2=F_3$, and we denote this $(-2)$-curve by $F$. The table below specifies the flat limits (in $\overline{P}_{3,{\mathbb C}}$) of the strata of the blowups of ${\mathbb P}^2$ isomorphic to $\overline{L}_{3,{\mathbb C}}$, as $t \to 0$.
		\begin{center}
			\begin{tabular}{c|c|c|c}
				stratum for $t \neq 0$ ($\overline{L}_{3,{\mathbb C}}$) &  $E_i$ & $F_i$ & $E_i \cap F_j$ ($i \neq j$) \\ \hline
				union of strata for $t=0$ ($\overline{P}_{3,{\mathbb C}}$) & $E_i$ & $F \cup E_i$ & $E_i \cap F$
			\end{tabular}
		\end{center}
		On the other hand, Figure \ref{figure: case n=3 example} specifies the modular behaviour on strata  (we may `consistently' permute the indices of the markings and strata), so, together with the table above, we obtain a complete picture of the degeneration when $n=3$.
		
		The generalization of this example to any $n$, and some other combinatorial issues regarding $\overline{P}_n$ and its deformation to $\overline{L}_n$ will be discussed in a future note.
	\end{example}
	
		\begin{remark}
			Here are a few further remarks regarding Theorem \ref{theorem: isotrivial degeneration theorem}.
		\begin{enumerate}
			\item In fact, $X \backslash X_{(t)} \simeq \spec {\mathbb Z}[t,t^{-1}] \times \overline{L}_n $ over $\spec {\mathbb Z}[t,t^{-1}]$, and the restriction of the $G^n_{\mathbb Z[t]}$-action to $X \backslash X_{(t)} $ is the pullback of the action on $\overline{L}_n$ along $ X \backslash X_{(t)} \simeq \spec {\mathbb Z}[t,t^{-1}] \times \overline{L}_n \to \overline{L}_n$.
			\item Although $\overline{L}_{n,{\mathbb C}}$ and $\overline{P}_{n,{\mathbb C}}$ are homeomorphic only for $n \leq 3$ since $\overline{P}_{n,{\mathbb C}}$ is singular for $n \geq 4$, Theorem \ref{theorem: isotrivial degeneration theorem} still shows that $\overline{L}_n$ and $\overline{P}_n$ are related topologically.
			\item The Losev-Manin space $\overline{L}_{n}$ is the moduli space of weighted pointed stable curves \cite{[Ha03]} for weight $(1,1,\epsilon,\ldots,\epsilon)$. However, $\overline{P}_{n}$ doesn't resemble any of Hassett's moduli spaces. I would like to thank Valery Alexeev for raising this issue by email, which prompted me to include it here. 
			\item It is very tempting to imagine the $p_0$ and $p_\infty$ markings for $\overline{L}_n$ coalescing into the $p_\infty$ marking for $\overline{P}_n$, and then hope to understand the picture with no reference to vector fields. However, this doesn't work. In fact, the $p_0$ marking doesn't even extend to a section over the total space $X$ if $n \geq 2$.
		\end{enumerate}
	\end{remark}

	We conclude with some general comments on the proofs of Theorems \ref{theorem: Pn bar as moduli space} and \ref{theorem: isotrivial degeneration theorem}. From a `modular' perspective, the key point is that the degeneration $X$ from Theorem \ref{theorem: isotrivial degeneration theorem} admits a modular interpretation quite similar to that of $\overline{P}_n$ in Theorem \ref{theorem: Pn bar as moduli space} if the double vanishing of the vector field at $p_\infty$ is relaxed to a simple vanishing requirement (Theorem \ref{theorem: the main degeneration as a moduli space theorem} and Corollary \ref{corollary: explicit statement that Ln degenerates to Pn}), and the map to ${\mathbb A}^1$ comes from the natural trivialization of $\omega_{C/S}^\vee(-p_\infty)/\omega_{C/S}^\vee(-2p_\infty)$ (\S\ref{subsection: coresidues}), up to an unimportant sign. From a `constructive' perspective, the constructions of $\overline{P}_n$ and $X$ follow the same inductive pattern reminiscent of the inductive construction of $\overline{M}_{0,n}$, with the exception that eliminating degeneracy is done in two steps in our setup. Starting with the trivial case $n=1$, we repeatedly (1) make a base change to the universal curve at the last step; then (2) apply `Knudsen stabilization with vector fields' (Construction \ref{construction: first bubbling up construction}) and (3) `inflate at zero vector' (Construction \ref{construction: second bubbling up construction}) to eliminate degeneracy. To connect the explicit constructions with the `modular' point of view, we require the inverse operations alluded to in \S\ref{subsection: Bubbling up and bubbling down} (\S\ref{section: Bubbling down and representability}). Finally, to deal with the group actions in Theorem \ref{theorem: isotrivial degeneration theorem}, we first construct a relative action on the family of curves (which requires some careful geometry), then use the modular interpretation of $X$ to `transfer' the group action from the curves to the base by moving around the markings on the curves, as explained above for $\overline{L}_n$ and $\overline{P}_n$ individually.
	
	Another degeneration in the style of that in Theorem \ref{theorem: isotrivial degeneration theorem} is sketched in \S\ref{example: examples with initial data: a genus 1 example}, without a completely formal statement or proof.
	
	\subsection*{Conventions} All schemes in this paper, including those in the definitions of moduli functors or fibered categories, are assumed noetherian. In fact, this assumption can probably be removed with more work, but we will not attempt to do so. `${\mathbb K}$' is always an algebraically closed field. We write ${\mathbb P}{\sh F} = \operatorname{Proj}_X \operatorname{Sym} {\sh F}$, for any coherent ${\sh O}_X$-module ${\sh F}$. A prestable curve is a proper flat morphism (of schemes), whose geometric fibers are connected curves with at worst nodal singularities, e.g. \cite[Definition 2.1]{[BM96]}. We say that it is of genus $g$ if all geometric fibers have arithmetic genus $g$. If $\pi:C \to S$ is a prestable curve and $x:S \to C$ is a section, we will often abuse notation by writing $x$ instead of $x(S)$ for the scheme-theoretic image of $x$ (sections of separated morphisms are closed immersions \cite[5.4.6]{[EGAI]}).
	
	\subsection*{Acknowledgments} This project was initiated a few years ago, during my visit at the MSRI. I would like to thank Professor David Eisenbud for interesting discussions and encouragement. I am also grateful to the anonymous referee for many excellent suggestions, which greatly improved the quality of the paper.
	
	This material is based upon work supported by the National Science Foundation under Grant No. 1440140, while the author was in residence at the Mathematical Sciences Research Institute in Berkeley, California, during the Spring 2020 semester. 
	
	\includegraphics[scale=0.3]{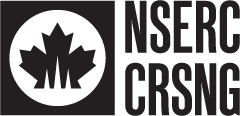} We acknowledge the support of the Natural Sciences and Engineering Research Council of Canada (NSERC), RGPIN-2020-05497. Cette recherche a \'{e}t\'{e} financ\'{e}e par le Conseil de recherches en sciences naturelles et en g\'{e}nie du Canada (CRSNG), RGPIN-2020-05497.
	
	\section{Universal curves}\label{section: universal curves}

	We start by introducing some notation.
	
	\begin{definition}\label{definition: categories of curves}
		Let $g,m,n \geq 0$ integers. An object of ${\cat V}^+_{g,m,n}$ consists of 
		\begin{itemize}
			\item a (noetherian, according to our conventions) scheme $S$;
			\item a prestable curve $\pi:C \to S$ of genus $g$;
			\item sections $w_1,\ldots,w_m:S \to C$ and $x_1,\ldots,x_n:S \to C$ of $\pi$; and
			\item an ${\sh O}_C$-module homomorphism $\phi:\omega_{C/S} \to {\sh O}_C$,
		\end{itemize} 
		satisfying the following conditions
		\begin{enumerate}
			\item\label{item: item 1 in definition: categories of curves} $\pi$ is smooth at $w_i(s)$ and $x_j(s)$, for all $s \in S$;
			\item $w_i(s) \neq w_j(s)$ if $i \neq j$, and $w_i(s) \neq x_j(s)$, for all $s \in S$;
			\item\label{item: item 3 in definition: categories of curves} $w_i^*\phi = 0$ as homomorphism $w_i^*\omega_{C/S} \to {\sh O}_S$, for all $i$;
			\item  $x_j^*\phi:x_j^*\omega_{C/S} \to {\sh O}_S$ is an isomorphism for all $j$.
		\end{enumerate}
		Note that we allow $x_i \cap x_j \neq \emptyset$. We will sometimes write ${\overline x}$ and ${\overline w}$ instead of $(x_1,\ldots,x_n)$ and $(w_1,\ldots,w_m)$. Arrows in ${\cat V}^+_{g,m,n}$ are pullbacks: an arrow 
		\[ (S',C',\pi',\overline{w}',\overline{x}',\phi') \to (S,C,\pi,\overline{w},\overline{x},\phi) \] 
		is a pair of morphisms $(h:S' \to S,r:C' \to C)$ such that the following diagrams
		\begin{center}
			\begin{tikzpicture}
				\node (nw) at (0,1) {$C'$};
				\node (ne) at (1,1) {$C$};
				\node (sw) at (0,0) {$S'$};
				\node (se) at (1,0) {$S$};
				
				\draw [->] (nw) -- node [midway, below] {$r$} (ne);
				\draw [->] (sw) -- node [midway, above] {$h$} (se);
				\draw [->] (nw) -- node [midway, left] {$\pi'$} (sw);
				\draw [->] (ne) -- node [midway, right] {$\pi$} (se);
				
				\node (nw) at (3,1) {$C'$};
				\node (ne) at (4,1) {$C$};
				\node (sw) at (3,0) {$S'$};
				\node (se) at (4,0) {$S$};
				
				\draw [->] (nw) -- node [midway, below] {$r$} (ne);
				\draw [->] (sw) -- node [midway, above] {$h$} (se);
				\draw [<-] (nw) -- node [midway, left] {$x_j'$} (sw);
				\draw [<-] (ne) -- node [midway, right] {$x_j$} (se);
				
				\node (nw) at (6,1) {$C'$};
				\node (ne) at (7,1) {$C$};
				\node (sw) at (6,0) {$S'$};
				\node (se) at (7,0) {$S$};
				
				\draw [->] (nw) -- node [midway, below] {$r$} (ne);
				\draw [->] (sw) -- node [midway, above] {$h$} (se);
				\draw [<-] (nw) -- node [midway, left] {$w_i'$} (sw);
				\draw [<-] (ne) -- node [midway, right] {$w_i$} (se);
			\end{tikzpicture}
		\end{center}
		are commutative and the first one is cartesian, and $\phi'$ corresponds to the pullback of $\phi$ under the isomorphism $C' \to C_{S'}$ induced by the first diagram.
		
		An object of ${\cat C}^+_{g,m,n}$ is defined verbatim the same as an object of ${\cat V}^+_{g,m,n}$, with the sole exception that, as part of the data, we consider an additional section $x:S \to C$ of $\pi$. No conditions involve this section in any way. Arrows are pullbacks again (as above, and another commutative diagram saying $r x' = x h$).
		
		Let ${\cat V}_{g,m,n}$ (respectively ${\cat C}_{g,m,n}$) be the full subcategory of ${\cat V}^+_{g,m,n}$ (respectively ${\cat C}^+_{g,m,n}$) consisting of objects for which 
		$ \omega_{C/S}\left( w_1+\cdots+w_m + 2x_1+\cdots+2x_n \right) $ 
		is $\pi$-ample (by condition \ref{item: item 1 in definition: categories of curves}, $x_j$ and $w_i$ are Cartier divisors on $C$).
	\end{definition}
	
	Instead of imposing condition \ref{item: item 3 in definition: categories of curves}, we could have equally well defined $\phi$ directly as a global section of $\omega_C^\vee(-w_1-\cdots-w_m)$. Depending on context, we will sometimes regard $\phi$ as a section of $\omega_C^\vee$, and sometimes as a section of $\omega_C^\vee(-w_1-\cdots-w_m)$.
	
	Unless $2g+m \leq 2$, a curve in ${\cat V}_{g,m,n}({\mathbb K})$ with a vector field that is not everywhere $0$ must be reducible, which is quite artificial, though shouldn't be discarded altogether, as illustrated by \S\ref{example: examples with initial data: positive genus example}. Examples \ref{example: some kind of relation with Mgn bar} and \ref{example: exceptions to previous example} below aim to clarify what types of curves we can expect to see in ${\cat V}_{g,m,n}({\mathbb K})$. They are slightly informal, since they will not be used in any proofs.
	
	\begin{figure}[h]
	\begin{center}
		\begin{tikzpicture}[scale = 0.95]
			
			\def\aa{0}; \def\bb{0};
			
			\draw (\aa,\bb) -- (\aa+3,\bb);
			\draw (\aa+0.5,\bb-0.5) -- (\aa+0.5,\bb+2);
			\draw (\aa,\bb+1.5) -- (\aa+1.5,\bb+1.5);
			\draw (\aa+2.5,\bb-0.5) -- (\aa+2.5,\bb+1.5);
			
			\filldraw[fill=white] (\aa+1,\bb+1.5) circle (2pt) node[above] {$y_1$};
			\filldraw[fill=white] (\aa+0.5,\bb+1) circle (2pt) node[right] {$y_2$};
			\filldraw[fill=white] (\aa+0.5,\bb+0.5) circle (2pt) node[right] {$y_3$};
			\filldraw[fill=white] (\aa+2.5,\bb+1) circle (2pt) node[right] {$y_4$};
			
			\node at (\aa+1.5,\bb-1) {in $\overline{\modu M}_{g,4}$};
			
			\def\aa{4}; \def\bb{0};
			
			\draw (\aa,\bb) -- (\aa+3,\bb);
			\draw (\aa+0.5,\bb-0.5) -- (\aa+0.5,\bb+2);
			\draw (\aa,\bb+1.5) -- (\aa+1.5,\bb+1.5);
			\draw (\aa+2.5,\bb-0.5) -- (\aa+2.5,\bb+1.5);
			
			\filldraw[fill=white] (\aa+1,\bb+1.5) circle (2pt) node[above] {$w_1$};
			\filldraw[fill=white] (\aa+0.5,\bb+1) circle (2pt) node[right] {$w_2$};
			
			\draw [very thick] (\aa,\bb+0.5) to (\aa+1.5,\bb+0.5) node[right] {$\Sigma_3$}; 
			\fill (\aa+0.3,\bb+0.5) circle (2pt) node[below] {$x_1$}; 
			\fill (\aa+0.8,\bb+0.5) circle (2pt) node[below] {$x_2$};
			\fill (\aa+1.2,\bb+0.5) circle (2pt) node[below] {$x_4$};
			
			\draw [very thick] (\aa+1.5,\bb+1) to (\aa+3,\bb+1) node[right] {$\Sigma_4$};
			\fill (\aa+2.2,\bb+1) circle (2pt) node[below] {$x_3$};
			
			\node at (\aa+1.5,\bb-1) {in ${\cat V}_{g,2,4}$};
			
			\def\aa{9}; \def\bb{0};
			
			\draw (\aa,\bb+1.5) -- (\aa+1,\bb+0.5);
			\draw (\aa+1,\bb+1.1) -- (\aa,\bb+0.1);
			\draw (\aa,\bb+0.6) -- (\aa+1,\bb-0.4);
			
			\filldraw[fill=white] (\aa+0.2,\bb+1.3) circle (2pt) node[above] {$w_1$};
			\filldraw[fill=white] (\aa+0.8,\bb-0.2) circle (2pt) node[right] {$w_2$};
			
			\fill  (\aa+0.4,\bb+1.1) circle (2pt) node[right] {$x_3$};
			\fill  (\aa+0.8,\bb+0.7) circle (2pt) node[right] {$x_2$};
			\fill  (\aa+0.5,\bb+0.1) circle (2pt) node[right] {$x_1$};
			\fill  (\aa+0.5,\bb+0.6) circle (2pt) node[left] {$x_4$};
			
			\node at (\aa+0.5,\bb-1) {in ${\cat V}_{0,2,4}$};
			
			\def\aa{11}; \def\bb{0};
			
			\draw (\aa,\bb+1.5) -- (\aa+1,\bb+0.5);
			\draw (\aa+1,\bb+1.1) -- (\aa,\bb+0.1);
			\draw (\aa,\bb+0.6) -- (\aa+1,\bb-0.4);
			
			\fill  (\aa+0.4,\bb+1.1) circle (2pt) node[right] {$x_3$};
			\fill  (\aa+0.8,\bb+0.7) circle (2pt) node[right] {$x_2$};
			\fill  (\aa+0.5,\bb+0.1) circle (2pt) node[right] {$x_1$};
			\fill  (\aa+0.5,\bb+0.6) circle (2pt) node[left] {$x_4$};
			
			\node at (\aa+0.5,\bb-1) {in ${\cat V}_{0,0,4}$};
			
		\end{tikzpicture}
	\end{center}
	\caption{The two pictures on the left show how a curve in ${\cat V}_{g,2,4}$ can be obtained from a curve in $\overline{\modu M}_{g,4}$. The two pictures on the right show two `exceptional' situations which can't be obtained by a construction of this type. (The examples are very similar, though the two vanishing points of the field are `$w$'-markings in one example but not in the other.)}
	\label{figure: relation of Mgn bar}
\end{figure}

	\begin{example}\label{example: some kind of relation with Mgn bar} Given $g,m,n,N$ such that $m < N \leq m+n$, there is a simple way to systematically produce curves in ${\cat V}_{g,m,n}({\mathbb K})$ out of curves $(Y,y_1,\ldots,y_N)$ in $\overline{\modu M}_{g,N}({\mathbb K})$. Let $\lambda:\{1,\ldots,n\} \to \{m+1,\ldots,N\}$ be a surjective function. Let 
		\[ C = Y \cup \Sigma_{m+1} \cup \cdots \cup \Sigma_N, \quad \text{where} \quad \Sigma_{m+1} \simeq \cdots \simeq \Sigma_{N} \simeq {\mathbb P}^1,  \]
		and each $\Sigma_i$ is attached transversally to $Y$ at $y_i$, for $i = m+1,\ldots,N$, $w_j = y_j$ for $j=1,\ldots,m$, and $\phi \in \Gamma(\omega_C^\vee)$ whose restriction to $Y$ is everywhere $0$, and whose restriction to $\Sigma_i \backslash \{ y_i \} \simeq {\mathbb A}^1$ is $\frac{d}{dx}$ in suitable coordinates. Finally, let us choose $x_\alpha \in \Sigma_{\lambda(\alpha)} \backslash \{y_{\lambda(\alpha)} \}$ arbitrary, for $\alpha = 1,\ldots,n$ (with automorphisms suitably taken into account, there are $n+m-N$ `moduli' for $x_1,\ldots,x_n$). Then,
		\[ (C,w_1,\ldots,w_m,x_1,\ldots,x_n,\phi) \in {\cat V}_{g,m,n}({\mathbb K}), \]
		by $\omega_Y(y_1+\cdots+y_N)$ ample, the classical description of dualizing sheaves in terms of meromorphic forms, and Definition \ref{definition: categories of curves}. An example is shown in Figure \ref{figure: relation of Mgn bar}, on the left. (As a trivial extension of this example to the case $m=N$, $n=0$, we have a functor $\overline{\modu M}_{g,m} \to {\cat V}_{g,m,0}$ by taking the vector field to be equal to $0$.)
	\end{example}
	
	\begin{example}\label{example: exceptions to previous example} 
		Can all curves in ${\cat V}_{g,m,n}({\mathbb K})$ be obtained using the construction in Example \ref{example: some kind of relation with Mgn bar}? The answer for most values of $g,m,n$ is positive if $N$ is allowed variable (the proof of Proposition \ref{proposition: normal generation in geometric case} implicitly contains the justification of this assertion), although there are a few exceptions, such as 
		\begin{enumerate}
			\item smooth genus $0$ curves with a vector field of the form $\frac{d}{dx}$ in suitable coordinates, in ${\cat V}_{0,0,n}$ and  ${\cat V}_{0,1,n}$;
			\item chains similar to those shown on the right side of Figure \ref{figure: relation of Mgn bar} (and also of the `mixed' type, with one vanishing point a `$w$'-marking and one not, which is important in the proof of Theorem \ref{theorem: isotrivial degeneration theorem}), in ${\cat V}_{0,0,n}$, ${\cat V}_{0,1,n}$, ${\cat V}_{0,2,n}$; or 
			\item genus $1$ curves (either smooth, or cycles of projective lines) with a nowhere vanishing vector field, in ${\cat V}_{1,0,n}$.
		\end{enumerate}
	\end{example}
		
	The main technical result is the following.
	
	\begin{theorem}\label{theorem: universal curve theorem simplified statement}
		If $2g+2n+m \geq 3$, then the (fibered) categories ${\cat V}_{g,m,n+1}$ and ${\cat C}_{g,m,n}$ are equivalent.
	\end{theorem} 
	
	As we will see in a moment, the equivalence between the two categories is constructed explicitly, and, in fact, it is given by the composition of Constructions \ref{construction: first bubbling up construction} and \ref{construction: second bubbling up construction}. Indeed, unlike in \cite{[Kn83]}, the passage from the universal curve to the family of curves with an extra marking requires not one, but two steps. To prove Theorem \ref{theorem: universal curve theorem simplified statement}, we introduce a class of curves ${\cat C}_{g,m,n}^2$ which will serve as an intermediate step between ${\cat V}_{g,m,n+1}$ and ${\cat C}_{g,m,n}$.
	
	\begin{definition}\label{definition: more categories of curves}
		For $c \in \{1,2,3\}$, let ${\cat C}_{g,m,n}^{c,+}$ be the full subcategory of ${\cat C}^+_{g,m,n}$ (cf. Definition \ref{definition: categories of curves}) which consists of objects $(S,C,\pi,\overline{w},\overline{x},x,\phi)$ such that:
		\begin{enumerate}
			\item if $c \geq 2$, $\pi$ is smooth at $x(s)$ and $w_i(s) \neq x(s)$, for all $s \in S$ and $i=1,\ldots,n$;
			\item if $c=3$, $x^*\phi:x^*\omega_{C/S} \to {\sh O}_S$ is an isomorphism.
		\end{enumerate}	
		Let ${\cat C}^c_{g,m,n}$ be the full subcategory of ${\cat C}_{g,m,n}^{c,+}$ whose objects also satisfy:
		\begin{enumerate}\setcounter{enumi}{2}
			\item\label{item: item 3 in definition: more categories of curves} $\omega_{C/S} (w_1+\cdots+w_m + 2x_1+\cdots+2x_n + (c-1)x)$ is $\pi$-ample.
		\end{enumerate}
	\end{definition}
	
	\begin{remark}\label{remark: only one category is new}
		Note that ${\cat V}_{g,m,n+1}$ is isomorphic to ${\cat C}_{g,m,n}^3$ and ${\cat C}_{g,m,n}$ to ${\cat C}_{g,m,n}^1$.
	\end{remark}

	As we will see in Proposition \ref{proposition: statement that the bubbling up functors were constructed}, the constructions in \S\ref{section: Bubbling up and the inductive construction} induce functors
		\begin{equation}\label{equation: sequence of up functors stated in advance}
			{\cat C}^1_{g,m,n} \longrightarrow {\cat C}^2_{g,m,n} \longrightarrow {\cat C}^3_{g,m,n}.
		\end{equation}
	
	\begin{theorem}\label{theorem: the main universal curve theorem}
		The functors in \eqref{equation: sequence of up functors stated in advance} give equivalences 
		$ {\cat C}_{g,m,n}^1 \simeq {\cat C}_{g,m,n}^2$ if $2g+2n+m \geq 3$, respectively ${\cat C}_{g,m,n}^2 \simeq {\cat C}_{g,m,n}^3$ if $2g+2n+m \geq 2$.
	\end{theorem}
	
	In light of Remark \ref{remark: only one category is new}, Theorem \ref{theorem: universal curve theorem simplified statement} follows from Theorem \ref{theorem: the main universal curve theorem}. 
	
	Although the results are very similar and the techniques overlap to a significant extent, there does not seem to be a clear-cut mathematical relation between Theorem \ref{theorem: universal curve theorem simplified statement}  and Knudsen's theorem that $\overline{\modu M}_{g,N+1}$ is the universal curve over $\overline{\modu M}_{g,N}$.
	
	Theorem \ref{theorem: universal curve theorem simplified statement} should be taken as a purely formal result. However, especially in the range $2g+m \leq 2$, there is reasonable amount of `elbow room' for interesting geometry, and it is usually the `exceptional' situations (Example \ref{example: exceptions to previous example}) that lead to more surprising phenomena. This will be illustrated using Theorems \ref{theorem: Pn bar as moduli space} and \ref{theorem: isotrivial degeneration theorem} (with $(g,m) = (0,1)$) and \S\ref{example: examples with initial data: a genus 1 example} (with $(g,m) = (1,0)$), though there are other setups which are at least nontrivial, even if arguably less aesthetic.
	
	\section{Contracting components of prestable curves}\label{section: Rational contractions of prestable curves}
	
	\subsection{Morphisms between prestable curves} We start by collecting some basic technicalities needed to deal with contracting (rational) irreducible components of prestable curves in families. These technicalities are quite well-known (cf. \cite[\href{https://stacks.math.columbia.edu/tag/0E7B}{Tag 0E7B}]{[stacks]} for some deeper aspects than the ones discussed here).
	
	\begin{definition}\label{definition: rational contraction}
		Let $\pi:X \to S$ and $\rho:Y \to S$ be prestable curves, and 
		$ f:X \to Y$ an $S$-morphism. We say that $f$ \emph{has property R} if $\smash{ f^\#:{\sh O}_Y \to f_*{\sh O}_X }$ is an isomorphism and $R^1f_*{\sh O}_X = 0$, and these hold universally, that is, $f_{S'}^\#:{\sh O}_{Y_{S'}} \to f_{S',*}{\sh O}_{X_{S'}} $ is an isomorphism and $R^1f_{S',*}{\sh O}_{X_{S'}} = 0$ for all $S' \to S$.
	\end{definition}

	\begin{remark}\label{remark: pushforward of pullback of line bundles}
		In the situation of Definition \ref{definition: rational contraction}, the map ${\sh L} \to f_*f^* {\sh L}$ is an isomorphism, for any ${\sh L} \in \mathrm{Pic}(Y)$. This is actually true for any morphism $f$ for which $f^\#$ is an isomorphism.
	\end{remark}
	
	\begin{lemma}\label{remark: checking rational contraction of fibers}
		If $\pi:X \to S$ and $\rho:Y \to S$ are prestable curves, and $ f:X \to Y $ is an $S$-morphism, then $f$ has property R if and only if the property in Definition \ref{definition: rational contraction} holds on the geometric fibers, i.e. $\smash{ f_{\overline{s}}^\#:{\sh O}_{Y_{\overline{s}}} \to f_{\overline{s},*}{\sh O}_{X_{\overline{s}}} }$ is an isomorphism and $R^1f_{\overline{s},*}{\sh O}_{X_{\overline{s}}} = 0$ for all geometric points $\overline{s} \to S$.
	\end{lemma}
	
	\begin{proof}
		The `only if' direction is trivial. \cite[\href{https://stacks.math.columbia.edu/tag/0E88}{Tag 0E88}]{[stacks]} reduces the `if' direction to the special case when $S$ is the spectrum of a field. For the lack of a reference, we explain this elementary case in some detail. Let $S = \spec K$, and $\alpha:X_{\overline{K}} \to X$ and $\beta:Y_{\overline{K}} \to Y$ the natural projections. The base change map  $\beta^* R^if_*{\sh O}_X \to R^if_{\overline{K},*}(\alpha^*{\sh O}_X) = R^if_{\overline{K},*}{\sh O}_{X_{\overline{K}}}$ is an isomorphism for all $i$ by the cohomology and flat base change theorem. Then $\beta^* R^1f_*{\sh O}_X = 0$, so $R^1f_*{\sh O}_X = 0$. Moreover, $\beta^*f^\#: \beta^*{\sh O}_Y \to \beta^*f_*{\sh O}_X$ is an isomorphism because it fits in a commutative diagram with $\beta^*{\sh O}_Y = {\sh O}_{Y_{\overline{K}}}$, $f^\#_{\overline{K}}$, and $\beta^*f_*{\sh O}_X \to f_{\overline{K},*}{\sh O}_{X_{\overline{K}}}$, which are all isomorphisms. It follows that $f^\#$ is an isomorphism.
	\end{proof}
	
	The following remark will not be used in any proof (so its proof will be omitted), but it is important for context and surely well-known. 
	
	\begin{remark} 
		If $f:X \to Y$ has property R and $S = \spec {\mathbb K}$, then, it can be shown that $f$ is obtained by repeatedly contracting \emph{rational tails} and/or \emph{rational bridges}. (Summary of an inductive argument: (1) unless it is an isomorphism, $f$ must contract some tail or bridge $\Sigma \simeq {\mathbb P}^1 \subset X$ by some standard $H^1$ calculations; (2) $f$ factors as $X \xrightarrow{h} X' \xrightarrow{f'} Y$, where $h:X \to X'$ is the contraction of $\Sigma$, by \cite[Lemma 2.2]{[BM96]}; and (3) $f'$ has property R by some purely formal/cohomological arguments, such as the five-term sequence of the Grothendieck spectral sequence.) Another way to look at this fact amounts to the following slightly imprecisely phrased `picture' of morphisms with property R (when $S = \spec {\mathbb K}$): $f$ is obtained by contracting several disjoint rational trees in $X$ which intersect the union of the other components of $X$ at either $1$ (tail-like) or $2$ (bridge-like) points (Figure \ref{figure: rational contraction}).
	\end{remark}
		\begin{figure}[h]
			\begin{center}
				\begin{tikzpicture}
					\draw[ultra thick] (-5,0) -- (-3,0); 
					\draw (-4.5,-0.3) -- (-4.5,1);  \draw (-5,0.4) -- (-4,0.4); \draw (-5,0.7) -- (-4,0.7); \node at (-4.8,-0.3) {$Y_1$}; 
					
					\draw[ultra thick] (-2.5,0) -- (-0.5,0); \node at (-0.4,-0.3) {$\widetilde{Y_2}$}; 
					
					\draw (-3.7,-0.3) -- (-2.2,1); \draw (-1.8,-0.3) -- (-2.9,1);
					\draw (-2.7,0.2) -- (-1.7,1);
					
					\draw (-1.5,-0.3) to[out=90,in=180] (-1.1,1) to[out=0,in=90] (-0.7,-0.3);
					
					\draw[->] (0,0.3) to node[above]{$f$} (1,0.3);
					
					\draw[ultra thick] (1.5,0) -- (3.5,0); \node at (2.5,0.3) {$Y_1$}; 
					
					\draw[ultra thick] (3,-0.3) to [out = 40,in=0] (3.6,0.7) to [out=180,in=140] (4.2,-0.3);  \node at (4.3,0) {$Y_2$}; 
					
				\end{tikzpicture}
			\end{center}
			\caption{A map $f$ with property R.}
			\label{figure: rational contraction}
		\end{figure}
	
	\begin{lemma}\label{lemma: dense f-trivial}
		In the situation of Definition \ref{definition: rational contraction}, let $ Y^\circ \subseteq Y$ consist of all $y \in Y$ such that $f^{-1}(y)$ is a singleton set-theoretically, and $X^\circ = f^{-1}(Y^\circ)$. Then:
		\begin{enumerate}
			\item\label{item 1: lemma: dense f-trivial} $Y^\circ$ is open in $Y$, and $X^\circ$ is open in $X$.
			\item\label{item 2: lemma: dense f-trivial} $Y \backslash Y^\circ$ is finite over $S$.
			\item\label{item 3: lemma: dense f-trivial} The restriction of $f$ to $X^\circ$ induces an isomorphism $X^\circ \simeq Y^\circ$. Moreover, if $U \subseteq Y$ is open, then the restriction of $f$ to $f^{-1}(U)$ induces an isomorphism $U \simeq f^{-1}(U)$ if and only if $U \subseteq Y^\circ$.
			\item\label{item 4: lemma: dense f-trivial} The formation of $Y^\circ$ and $X^\circ$ commutes with base change.
		\end{enumerate}
	\end{lemma}
	
	\begin{proof}
		The fibers of $f$ are connected by Zariski's connectedness theorem \cite[Corollaire (4.3.2)]{[EGAIII]} and hence they are either positive dimensional or singletons set-theoretically, since they are of finite type over a field. Item \ref{item 1: lemma: dense f-trivial} then follows from Chevalley's theorem \cite[Corollaire (13.1.5)]{[EGAIV]}.  Item \ref{item 4: lemma: dense f-trivial} is set-theoretic and clear, since the set underlying the preimage of an open subscheme is the set-theoretic preimage of the underlying set. Item \ref{item 4: lemma: dense f-trivial} reduces item \ref{item 2: lemma: dense f-trivial} to the case when $S$ is the spectrum of a field, when it is clear -- quasi-finiteness suffices, since $Y\backslash Y^\circ$ is proper over $S$. For item \ref{item 3: lemma: dense f-trivial}, the general observation is that, if $f:X \to Y$ is a proper morphism such that $f^\#:{\sh O}_Y \to f_*{\sh O}_X$ is an isomorphism, then $f$ is an isomorphism if and only if it is a bijection. Indeed, $f$ is a continuous closed bijection, hence a homeomorphism, and we are assuming that $f^\#$ is an isomorphism.
	\end{proof}
	
	The practical use of Lemma \ref{lemma: dense f-trivial} is that it often allows us to easily check on $X^\circ$ or $Y^\circ$ nontrivial statements about $f$ on $X$ or $Y$, and then `bootstrap' to $X$ or $Y$, using tricks such as the following.
	
	\begin{lemma}\label{lemma: injective restrictions on prestable curves}
		If $\pi:X \to S$ is a prestable curve, and $U \subseteq X$ is an open whose complement $X \backslash U$ is finite over $S$, then the restriction map $\Gamma(V,{\sh L}) \to \Gamma(U \cap V,{\sh L})$ is injective for any invertible ${\sh L}$, and any open $V \subseteq X$.
	\end{lemma}
	
	\begin{proof}
		It suffices to prove the lemma for ${\sh L} = {\sh O}_X$. By \cite[\href{https://stacks.math.columbia.edu/tag/0B3L}{Tag 0B3L}]{[stacks]}, it suffices to prove that all associated points of $V$ are in $U \cap V$. If $S$ is the spectrum of a field, then $X$ is reduced and hence $V$ is reduced, so all associated points of $V$ are generic points of irreducible components of $X$ by \cite[\href{https://stacks.math.columbia.edu/tag/05AL}{Tag 05AL}]{[stacks]} and the well-known fact that noetherian reduced schemes have no embedded components, and the claim follows in this special case. In the general case, if $x$ is an associated point of $V$, then it is also an associated point of $ (\pi|_V)^{-1}(\pi(x))$ by \cite[\href{https://stacks.math.columbia.edu/tag/05DB}{Tag 05DB}]{[stacks]}, which boils down the claim to the special case in the previous sentence, completing the proof.
	\end{proof}
	
	\subsection{Differentials} Next, we discuss the differentials of morphisms with property R. These will allow us to push forward vector fields, and turn out to be an essential operation later in the paper. The differentials will be constructed using a simple case of coherent duality. (Logarithmic structures seem a natural approach, but, with hindsight, they seem a worse option overall.) For the basic properties of the relative dualizing sheaf of a prestable curve, please see \cite[\S1]{[Kn83]} and \cite[\href{https://stacks.math.columbia.edu/tag/0E6N}{Tag 0E6N}]{[stacks]}.
	
	\begin{lemma}\label{lemma: upper shriek of invertible is invertible}
		In the situation of Definition \ref{definition: rational contraction}, if ${\sh L}$ is an invertible ${\sh O}_Y$-module, then $f^!{\sh L} \simeq \omega_{X/S} \otimes f^*\omega_{Y/S}^\vee \otimes f^*{\sh L}$ regarded as a complex in degree $0$.
	\end{lemma}
	
	\begin{proof}
		We have $\pi^!{\sh O}_S = \omega_{X/S}[1]$ and $\rho^!{\sh O}_S = \omega_{Y/S}[1]$. Moreover, $f^!$ is the right adjoint of $Rf_*$ because $f$ is proper. By \cite[\href{https://stacks.math.columbia.edu/tag/0A9T}{Tag 0A9T}]{[stacks]},
		\begin{equation}\label{equation: random technical equation with !}
			 f^!\left(\omega_{Y/S}[1] \right) =  f^!\left(\omega_{Y/S}[1] \otimes^{{\mathbf L}}_{{\sh O}_Y} {\sh O}_Y \right) = Lf^*\omega_{Y/S}[1] \otimes^{{\mathbf L}}_{{\sh O}_X} f^!{\sh O}_Y, 
		\end{equation}
		since noetherian ensures quasi-compact and quasi-separated. Therefore,
		\begin{equation*}
			\begin{aligned}
				\omega_{X/S}[1] &= \pi^!{\sh O}_S = f^!(\rho^!{\sh O}_S) \quad \text{by \cite[\href{https://stacks.math.columbia.edu/tag/0ATX}{Tag 0ATX}]{[stacks]}} \\
				&= f^!\left(\omega_{Y/S}[1] \right) = Lf^*\omega_{Y/S}[1] \otimes^{{\mathbf L}}_{{\sh O}_X} f^!{\sh O}_Y \quad \text{by \eqref{equation: random technical equation with !}}
			\end{aligned}
		\end{equation*}
		so $f^!{\sh O}_Y \simeq \omega_{X/S} \otimes f^*\omega_{Y/S}^\vee$ as a complex in degree $0$. However, $f^!{\sh L} = Lf^*{\sh L} \otimes^{{\mathbf L}}_{{\sh O}_X} f^!{\sh O}_Y$ by another application of \cite[\href{https://stacks.math.columbia.edu/tag/0A9T}{Tag 0A9T}]{[stacks]}, and the lemma follows.
	\end{proof}
	
	By Lemma \ref{lemma: upper shriek of invertible is invertible}, we may think of $f^!{\sh L}$ as an invertible sheaf rather than an object of the derived category. Lemma \ref{lemma: upper shriek of invertible is invertible} also shows that the formation of $f^!{\sh L}$ commutes with base changes $S' \to S$, since the formation of $\omega_{X/S}$ and $\omega_{Y/S}$ commutes with base change \cite[\href{https://stacks.math.columbia.edu/tag/0E6R}{Tag 0E6R}]{[stacks]}.
	
	\begin{example}\label{example: example of f upper shriek}
		If $f:X \to Y$ over $S = \spec {\mathbb K}$ has property R, and ${\sh L}$ is a line bundle on $Y$, $f^!{\sh L}$ can be understood concretely using Lemma \ref{lemma: upper shriek of invertible is invertible}. In particular, if $\Sigma \simeq {\mathbb P}^1$ is a contracted irreducible component of $X$, then $ \deg (f^!{\sh L})|_\Sigma = \deg (\omega_X)|_\Sigma = n_\Sigma-2$, where $n_\Sigma$ is the number of nodes of $X$ on $\Sigma$. For instance, Figure \ref{figure: degrees of f upper shriek} illustrates this on the map in Figure \ref{figure: rational contraction}, for an arbitrary ${\sh L}$ on $Y$. On non-contracted components, in this example, we have $(f^!{\sh L})|_{Y_1} = {\sh L}|_{Y_1} \otimes {\sh O}_{Y_1}(p)$  and $(f^!{\sh L})|_{\widetilde{Y_2}} = \nu^*({\sh L}|_{Y_2})$ by Lemma \ref{lemma: upper shriek of invertible is invertible} again, where $\nu$ is the normalization of $Y_2$, and $p$ is the point indicated in Figure \ref{figure: degrees of f upper shriek}. 
		\begin{figure}[h]
			\begin{center}
				\begin{tikzpicture}
					\draw[ultra thick] (-5,0) -- (-3,0); 
					\draw (-4.5,-0.3) -- (-4.5,1); \node at (-4.4,1) { {\tiny{$1$}} };  
					\draw (-5,0.4) -- (-4,0.4); \node at (-5.2,0.4) { {\tiny{$-1$}} }; 
					\draw (-5,0.7) -- (-4,0.7);  \node at (-5.2,0.7) { {\tiny{$-1$}} };
					
					\draw[ultra thick] (-2.5,0) -- (-0.5,0); 
					
					\draw (-3.7,-0.3) -- (-2.2,1); \node at (-2.4,1) { {\tiny{$0$}} }; 
					\draw (-1.8,-0.3) -- (-2.9,1); \node at (-2,0.6) { {\tiny{$-1$}} };
					\draw (-2.7,0.2) -- (-1.7,1); \node at (-3,1) { {\tiny{$1$}} }; 
					
					\draw (-1.5,-0.3) to[out=90,in=180] (-1.1,1) to[out=0,in=90] (-0.7,-0.3); \node at (-0.8,1) { {\tiny{$0$}} }; 
					
					\draw[->] (0,0.3) to node[above]{$f$} (1,0.3);
					
					\draw[ultra thick] (1.5,0) -- (3.5,0); 
					
					\draw[ultra thick] (3,-0.3) to [out = 40,in=0] (3.6,0.7) to [out=180,in=140] (4.2,-0.3); 
					
					\node at (-4.8,-0.3) {$Y_1$}; 
					\node at (-0.4,-0.3) {$\widetilde{Y_2}$}; 
					\node at (2.5,0.3) {$Y_1$}; 
					\node at (4.3,0) {$Y_2$}; 
					
					\node at (-4.3,-0.2) {$p$}; 
					
				\end{tikzpicture}
			\end{center}
			\caption{The integers on the contracted components of the source specify the degrees of $f^!{\sh L}$ on the respective components.}
			\label{figure: degrees of f upper shriek}
		\end{figure}
		
		Another concrete way to understand $f^!{\sh L}$ when $S=\spec{\mathbb K}$ combines:
		\begin{enumerate}
			\item if $f:X \to Y$ is the contraction of a rational bridge, then $f^!{\sh L} = f^*{\sh L}$;
			\item if $f$ is the contraction of a rational tail $T \simeq {\mathbb P}^1 \subset X$, then $f^!{\sh L} \simeq {\sh J} \otimes f^*{\sh L}$, where ${\sh J}|_T \simeq {\sh O}_T(- 1)$, and ${\sh J}|_W = {\sh O}_W(p)$, where $W$ is the closure of $X \backslash T$, and $\{p\} = T \cap W$ (this follows from Lemma \ref{lemma: upper shriek of invertible is invertible} again);
		\end{enumerate}
		and the functoriality ($f^!g^! = (gf)^!$) of `$!$'. Note also that $\deg f^!{\sh L} = \deg {\sh L}$. 
	\end{example}
	
	\begin{situation}\label{situation: situation of logarithmic differential using duality}
		Let $\pi:X \to S$ and $\rho:Y \to S$ be prestable curves over $S$, and $ f:X \to Y$ with property R. Let $x_1,\ldots,x_n:S \to X$ be sections of $\pi$, and $y_i = f x_i :S \to Y$ the corresponding sections of $\rho$, $i=1,\ldots,n$. Assume that:
		\begin{enumerate}
			\item $\pi$ is smooth at $x_i(s)$ and $\rho$ is smooth at $y_i(s)$, for all $s \in S$, and $i=1,\ldots,n$;
			\item the sections $x_1,\ldots,x_n$ respectively $y_1,\ldots,y_n$ are pairwise disjoint.
		\end{enumerate}
	\end{situation}
	
	In Situation \ref{situation: situation of logarithmic differential using duality}, $x_i \subset X$ and $y_i \subset Y$ (cf. Conventions subsection) are effective Cartier divisors.
	
	\begin{proposition}\label{proposition: construction of logarithmic differential using duality}
		In Situation \ref{situation: situation of logarithmic differential using duality}, the following hold.
		\begin{enumerate}
			\item\label{item 1: proposition: construction of logarithmic differential using duality} The isomorphism $f^\#:{\sh O}_Y \to f_*{\sh O}_X$ restricts to an isomorphism
			\begin{equation}\label{equation: structure map restricts to isomorphism if twisted down by sections}
				{\sh O}_Y(-y_1-\cdots-y_n) \simeq f_*{\sh O}_X(-x_1-\cdots-x_n).
			\end{equation} 
			Moreover, $R^1f_*{\sh O}_X(-x_1-\cdots-x_n) = 0$.
			\item\label{item 2: proposition: construction of logarithmic differential using duality} For any invertible ${\sh O}_Y$-module ${\sh L}$, there exists an ${\sh O}_X$-module homomorphism
			\begin{equation}\label{equation: logarithmic derivative in general expressed with duality}
				\phi:f^*({\sh L}(y_1+\cdots+y_n)) \to (f^!{\sh L})(x_1+\cdots+x_n)
			\end{equation}
			such that for any open $U \subseteq Y^\circ$, $\phi|_{f^{-1}(U)}$ is the isomorphism $f^*({\sh L}(y_1+\cdots+y_n))|_{f^{-1}(U)} \simeq (f^!{\sh L})(x_1+\cdots+x_n)|_{f^{-1}(U)}$ induced by $(f|_U)^* = (f|_U)^!$.
		\end{enumerate} 
	\end{proposition}
	
	\begin{proof}
		Let $D = \sum_{i=1}^n x_i$ and $E = \sum_{i=1}^n y_i$. If 
		\[ 0 \to {\sh O}_X(-D) \to {\sh O}_X \to \bigoplus_{i=1}^n x_{i,*}{\sh O}_S \to 0 \] 
		is pushed forward along $f$, we obtain a (solid arrow) commutative diagram
		\begin{center}
			\begin{tikzpicture}
				\matrix [column sep  = 4mm, row sep = 4mm] {
					\node (t0) {$0$}; &
					\node (t1) {$f_*{\sh O}_X(-D)$}; &
					\node (t2) {$f_*{\sh O}_X$}; &
					\node (t3) {$\displaystyle f_*\bigoplus_{i=1}^n x_{i,*}{\sh O}_S$}; &
					\node (t4) {$R^1f_*{\sh O}_X(-D)$}; &
					\node (t5) {$\cdots$};  \\
					\node (b0) {$0$}; &
					\node (b1) {${\sh O}_Y(-E)$}; &
					\node (b2) {${\sh O}_Y$}; &
					\node (b3) {$\displaystyle  \bigoplus_{i=1}^n y_{i,*}{\sh O}_S$}; &
					\node (b4) {$0.$}; & \\
				};
				\draw[->] (t0) -- (t1);
				\draw[->] (t1) -- (t2);
				\draw[->] (t2) -- (t3);
				\draw[->] (t3) -- (t4);
				\draw[->] (t4) -- (t5);
				
				\draw[->] (b0) -- (b1);
				\draw[->] (b1) -- (b2);
				\draw[->] (b2) -- (b3);
				\draw[->] (b3) -- (b4);
				
				\draw[->, dashed] (b1) -- (t1);
				\draw[->] (b2) -- node [midway, right] {$f^\#$} (t2);
				\draw[double equal sign distance] (b3) -- (t3);
			\end{tikzpicture}
		\end{center}
		Then $f_*{\sh O}_X \to f_*\bigoplus_{i=1}^n x_{i,*}{\sh O}_S$ is surjective as ${\sh O}_Y \to \bigoplus_{i=1}^n y_{i,*}{\sh O}_S$ is surjective. We claim that
		\begin{equation}\label{equation: twisted vanishing} R^jf_*{\sh O}_X(-D) = 0 \end{equation}
		for all $j > 0$. First, $R^1f_*x_{i,*}{\sh O}_S = 0$ because $R^1y_{i,*}{\sh O}_S = 0$ (as $y_i$ is finite) and $0 \to R^1f_*x_{i,*}{\sh O}_S \to R^1y_{i,*}{\sh O}_S$ is the beginning of the five-term sequence of the Grothendieck spectral sequence. Second, in the piece
		\begin{equation*} R^{j-1}f_*{\sh O}_X \to R^{j-1}f_*\bigoplus_{i=1}^nx_{i,*}{\sh O}_S \to R^jf_*{\sh O}_X(-D) \to R^jf_*{\sh O}_X \end{equation*}
		of the top row, the first map is surjective (if $j=1$, we've shown above that it is surjective; if $j=2,$ we've shown above that the second term is $0$; if $j \geq 3$, the second term is trivially $0$ \cite[Corollaire (4.2.2)]{[EGAIII]}), and the last term is $0$ by \cite[Corollaire (4.2.2)]{[EGAIII]} again or Definition \ref{definition: rational contraction}, completing the proof of \eqref{equation: twisted vanishing}. 
		
		In particular, there exists a unique ${\sh O}_Y$-module homomorphism $\gamma:{\sh O}_Y(-E) \to f_*{\sh O}_X(-D)$
		that can play the role of the dashed arrow and make the diagram commute, and this $\gamma$ is an isomorphism, completing the proof of part \ref{item 1: proposition: construction of logarithmic differential using duality}. Combining with \eqref{equation: twisted vanishing}, we obtain a quasi-isomorphism
		\begin{equation}\label{equation: isomorphism in derived category after twist} Rf_*{\sh O}_X(-D) \simeq {\sh O}_Y(-E). \end{equation}
		To explain the last step in more detail, let $0 \to {\sh O}_X(-D) \to {\sh J}^0 \to {\sh J}^1 \to \cdots$ be an injective resolution of ${\sh O}_X(-D)$ in ${\cat{Qcoh}}(X)$. We have $H^0(f_*{\sh J}^\bullet) =  {\sh O}_Y(-E)$ and $H^j(f_*{\sh J}^\bullet) =  0$ for $j \neq 0$,
		and hence the map of complexes ${\sh O}_Y(-E) \to f_*{\sh J}^\bullet = [\cdots \to 0 \to f_*{\sh J}^0 \to f_*{\sh J}^1 \to \cdots]$, with ${\sh O}_Y(-E)$ in degree $0$, which maps ${\sh O}_Y(-E) \cong \mathrm{Ker}(f_*{\sh J}^0 \to f_*{\sh J}^1)$ to $f_*{\sh J}^0$ in the natural way, is a quasi-isomorphism. Since $f$ is proper, $Rf_*$ and $f^!$ are adjoint, so there is a bijection
		\begin{equation}
			\begin{aligned}
				\mathrm{Hom}_{D^+(\cat{Qcoh}(Y))}&(Rf_*{\sh O}_X(-D), {\sh O}_Y(-E)) \simeq \\
				&\mathrm{Hom}_{D^+(\cat{Qcoh}(X))}({\sh O}_X(-D), f^!{\sh O}_{Y}(-E)),
			\end{aligned}
		\end{equation}
		and let ${\sh O}_X(-D) \to f^!{\sh O}_Y(-E)$ correspond to \eqref{equation: isomorphism in derived category after twist} under this bijection. Taking $H^0$, we may think of this as a map of ${\sh O}_X$-modules, cf. Lemma \ref{lemma: upper shriek of invertible is invertible}. Twisting by ${\sh O}_X(D) \otimes f^*({\sh L}(E))$ on both sides and keeping Lemma \ref{lemma: upper shriek of invertible is invertible} in mind, we obtain the desired homomorphism $\phi:f^*({\sh L}(E)) \to (f^!{\sh L})(D)$. The fact that $\phi$ restricts to the `obvious' isomorphism on each open subset $U \subseteq Y$ on which $f$ induces an isomorphism $f^{-1}(U) \simeq U$ follows from the fact that duality behaves naturally relative to restricting to open subschemes. 
	\end{proof}
	
	\begin{lemma}\label{lemma: critrion for log differential to be isomorphism}
		In Situation \ref{situation: situation of logarithmic differential using duality}, assume that 
		\begin{equation}\label{equation: crepant condition}
			f_{\overline{s}}^*\omega_{Y_{\overline{s}}}(y_{1,\overline{s}} + \cdots + y_{n,\overline{s}}) \simeq \omega_{X_{\overline{s}}}(x_{1,\overline{s}} + \cdots + x_{n,\overline{s}})
		\end{equation}
		for all geometric points $\overline{s} \to S$. Then there exists a unique homomorphism $\phi$ as in part \ref{item 2: proposition: construction of logarithmic differential using duality} of Proposition \ref{proposition: construction of logarithmic differential using duality}, and this $\phi$ is an isomorphism. Moreover, its formation commutes with base change.
	\end{lemma}
	
	\begin{proof}
		Let $D = \sum_{i=1}^n x_i$ and $E = \sum_{i=1}^n y_i$. By Lemma \ref{lemma: upper shriek of invertible is invertible}, we have an isomorphism
		\begin{equation}\label{equation: isomorphism of hom sheaves in lemma: critrion for log differential to be isomorphism}
			{\sh H}om(f^*({\sh L}(E)), (f^!{\sh L})(D)) \simeq {\sh H}om(f^*(\omega_{Y/S}(E)), \omega_{X/S}(D)).
		\end{equation}
		Let ${\sh H}$ be the left hand side of \eqref{equation: isomorphism of hom sheaves in lemma: critrion for log differential to be isomorphism}, so that $\phi \in \Gamma(X,{\sh H})$. We have ${\sh H}_{\overline{s}} \simeq {\sh O}_{X_{\overline{s}}}$ for all geometric points $\overline{s} \to S$, by \eqref{equation: isomorphism of hom sheaves in lemma: critrion for log differential to be isomorphism} and the assumption, hence $\phi_{\overline{s}} \in \Gamma(X_{\overline{s}},{\sh H}_{\overline{s}})$ is nowhere vanishing since it isn't identically $0$ (for instance, by Lemma \ref{lemma: dense f-trivial}), and the fact that $X_{\overline{s}}$ is connected and proper. Then $\phi$ is a nowhere vanishing section of ${\sh H}$, thus an isomorphism, proving one of the claims. Uniqueness follows from Lemmas \ref{lemma: dense f-trivial} and \ref{lemma: injective restrictions on prestable curves}. Commutativity with respect to base change follows from uniqueness, combined with items \ref{item 3: lemma: dense f-trivial} and \ref{item 4: lemma: dense f-trivial} in Lemma \ref{lemma: dense f-trivial}.
	\end{proof}
	
	\begin{definition}\label{definition: dual shriek dual}
		Let $f:X \to Y$ be a map with property R, and let ${\sh L}$ be an invertible ${\sh O}_Y$-module. Then define $f^\dsd{\sh L} = (f^!{\sh L}^\vee)^\vee$.
	\end{definition}
	
	We've seen earlier that the formation of $f^!({\sh L}^\vee)$ commutes with base change, and hence so does the formation of $f^\dsd{\sh L}$, since $f^!{\sh L}^\vee$ is invertible by Lemma \ref{lemma: upper shriek of invertible is invertible}.
	
	\begin{example}\label{example: example of f disappointed}
		This is the analogue of Example \ref{example: example of f upper shriek} for $f^{\text{<}}$. Note that the degrees on contracted components switch sign from $f^!{\sh L}$ to $f^{\text{<}}{\sh L}$, cf. Figures \ref{figure: degrees of f disappointed} and \ref{figure: degrees of f upper shriek}. In the same example, we also have $(f^{\text{<}}{\sh L})|_{Y_1} = {\sh L}|_{Y_1} \otimes {\sh O}_{Y_1}(-p) $ and $(f^{\text{<}}{\sh L})|_{\widetilde{Y_2}} = \nu^*({\sh L}|_{Y_2})$, with notation similar to that in Example \ref{example: example of f upper shriek}.
	\end{example}
	
	\begin{figure}[h]
		\begin{center}
			\begin{tikzpicture}
				\draw[ultra thick] (-5,0) -- (-3,0); 
				\draw (-4.5,-0.3) -- (-4.5,1); \node at (-4.3,1) { {\tiny{$-1$}} };  
				\draw (-5,0.4) -- (-4,0.4); \node at (-5.1,0.4) { {\tiny{$1$}} }; 
				\draw (-5,0.7) -- (-4,0.7);  \node at (-5.1,0.7) { {\tiny{$1$}} };
				
				\draw[ultra thick] (-2.5,0) -- (-0.5,0); 
				
				\draw (-3.7,-0.3) -- (-2.2,1); \node at (-2.4,1) { {\tiny{$0$}} }; 
				\draw (-1.8,-0.3) -- (-2.9,1); \node at (-1.9,1) { {\tiny{$1$}} };
				\draw (-2.7,0.2) -- (-1.7,1); \node at (-3.1,1) { {\tiny{$-1$}} }; 
				
				\draw (-1.5,-0.3) to[out=90,in=180] (-1.1,1) to[out=0,in=90] (-0.7,-0.3); \node at (-0.8,1) { {\tiny{$0$}} }; 
				
				\draw[->] (0,0.3) to node[above]{$f$} (1,0.3);
				
				\draw[ultra thick] (1.5,0) -- (3.5,0); 
				
				\draw[ultra thick] (3,-0.3) to [out = 40,in=0] (3.6,0.7) to [out=180,in=140] (4.2,-0.3); 
				
				\node at (-4.8,-0.3) {$Y_1$}; 
				\node at (-0.4,-0.3) {$\widetilde{Y_2}$}; 
				\node at (2.5,0.3) {$Y_1$}; 
				\node at (4.3,0) {$Y_2$}; 
				
				\node at (-4.3,-0.2) {$p$}; 
				
			\end{tikzpicture}
		\end{center}
		\caption{The integers on the contracted components of the source indicate the degrees of $f^{\text{<}}{\sh L}$ on the respective components.}
		\label{figure: degrees of f disappointed}
	\end{figure}
	
	\begin{proposition}\label{proposition: pushing forward vector fields in general}
		In Situation \ref{situation: situation of logarithmic differential using duality}, there exists a unique homomorphism
		\begin{equation}\label{equation: pushing forward logarithmic vector fields in general}
			\psi: f_*((f^\dsd {\sh L})(-x_1-\cdots-x_n)) \to {\sh L}(-y_1-\cdots-y_n).
		\end{equation} 
		such that, for any open $U \subseteq Y$ above which $f$ restricts to an isomorphism $f^{-1}(U) \simeq U$, $\psi|_U$ is the natural isomorphism induced by $f^\dsd{\sh L}|_U \cong f^*{\sh L}|_U$. 
		
		If condition \eqref{equation: crepant condition} holds for all geometric points $\overline{s} \to S$, then $\psi$ is an isomorphism. 
		
		Moreover, the formation of $\psi$ is compatible with base change in the following sense: if $h:S' \to S$ is a morphism, $X',Y',\ldots$ is the pullback along $h$ of the data in Situation \ref{situation: situation of logarithmic differential using duality}, $m:Y' = S' \times_S Y \to Y$ is the projection map, and $\psi'$ is the analogue of \eqref{equation: pushing forward logarithmic vector fields in general}, then the diagram
		\begin{center}
			\begin{tikzpicture}
				\node (a) at (0,0) {$f'_*((f'^\dsd {\sh L}')(-x'_1-\cdots-x'_n))$};
				\node (b) at (0,1.1) {$m^*f_*((f^\dsd {\sh L})(-x_1-\cdots-x_n))$};
				\node (c) at (5.5,0) {${\sh L}'(-y'_1-\cdots-y'_n)$};
				\node (d) at (5.5,1.1) {$m^*{\sh L}(-y_1-\cdots-y_n)$};
				
				\draw[->] (b) -- (a);
				\draw[->] (b) -- node [midway, above] {$m^*\psi$} (d);
				\draw[->] (a) -- node [midway, above] {$\psi'$} (c);
				\draw[double equal sign distance] (d) -- (c);
			\end{tikzpicture}
		\end{center}
		in which the left vertical map comes from the standard base change map $m^*f_*(\cdot) \to f'_*p^* (\cdot)$ (where $p:X' = S' \times_S X \to X$ is the projection map) and the remark before Definition \ref{definition: dual shriek dual}, is commutative.
	\end{proposition}
	
	\begin{proof}
		Let $D = \sum_{i=1}^n x_i$ and $E = \sum_{i=1}^n y_i$, and let $f^*({\sh L}^\vee(E)) \to (f^!{\sh L}^\vee)(D)$ obtained by replacing ${\sh L}$ with ${\sh L}^\vee$ in \eqref{equation: logarithmic derivative in general expressed with duality}. Dualizing, we obtain 
		$ (f^\dsd{\sh L})(-D) \to f^*({\sh L}(-E)) $.
		Pushing the last homomorphism forward along $f$ with Remark \ref{remark: pushforward of pullback of line bundles} in mind, we obtain $f_*((f^\dsd {\sh L})(-D)) \to {\sh L}(-E)$, which is \eqref{equation: pushing forward logarithmic vector fields in general}. This proves existence. The isomorphism criterion follows easily from Lemma \ref{lemma: critrion for log differential to be isomorphism} and the construction of $\psi$ above. Both uniqueness and compatibility with base change follow from Lemmas \ref{lemma: dense f-trivial} and \ref{lemma: injective restrictions on prestable curves}. Since uniqueness can be regarded as a sort of special case of compatibility with base change, we only explain the latter. Let $Y^\circ \subseteq Y$ as in Lemma \ref{lemma: dense f-trivial}. It is clear that the restriction of the square diagram to $m^{-1}(Y^\circ)$ is commutative. Then the square diagram commutes a fortiori. Indeed, the two images in ${\sh L}'(-y'_1-\cdots-y'_n)$ of any local section of $f'_*((f'^\dsd {\sh L}')(-x'_1-\cdots-x'_n))$ above some open $V \subseteq Y$ via the two possible routes restrict to the same local section of ${\sh L}'(-y'_1-\cdots-y'_n)$ on $V \cap m^{-1}(Y^\circ)$, so they must coincide on $V$ by Lemma \ref{lemma: injective restrictions on prestable curves}, which applies in our situation in light of items \ref{item 2: lemma: dense f-trivial} and \ref{item 4: lemma: dense f-trivial} in Lemma \ref{lemma: dense f-trivial}.  
	\end{proof}
	
	\section{Bubbling up}\label{section: Bubbling up and the inductive construction}
	
	\subsection{Knudsen stabilization with sections of line bundles}\label{subsection: Knudsen stabilization with sections of line bundles} In \S\ref{subsection: Knudsen stabilization with sections of line bundles}, we generalize to families the operation called `Knudsen stabilization with vector fields' in \S\ref{subsection: Bubbling up and bubbling down}. (We also allow for sections of arbitrary line bundles instead of vector fields, though this won’t be a major improvement in generality.) There is very little to add to Knudsen's well-known work; we will merely check that, if we have a vector field, it lifts uniquely to one with the desired properties.
	
	\begin{situation}\label{situation: Knudsen stabilization with sections of line bundles}
		Assume that $\pi:C \to S$ is a prestable curve, $x$ is a section of $\pi$, ${\sh L}$ is an invertible ${\sh O}_C$-module, $w_1,\ldots,w_m:S \to C$ are disjoint smooth sections of $\pi$, and $\sigma$ is a global section of ${\sh L}(-w_1-\cdots-w_m)$.
	\end{situation}
	
	In \S\ref{subsection: Knudsen stabilization with sections of line bundles}, we consider such data \emph{degenerate} if $\pi$ is singular at $x(s)$ for some $s \in S$, or $x(s) = w_i(s)$, for some $s \in S$ and $i \in \{1,\ldots,m\}$.
	
	In Situation \ref{situation: Knudsen stabilization with sections of line bundles}, we construct $S_\kukp =S$, $C_\kukp$, $\pi_\kukp$, $x_\kukp$, ${\sh L}_\kukp$, $\sigma_\kukp$, $w_{1,\kukp},\ldots,w_{m,\kukp}$ such that the new data satisfies the requirements of Situation \ref{situation: Knudsen stabilization with sections of line bundles} and is nondegenerate, and the construction commutes with base change (in a manner which satisfies the natural `cocycle conditions', at least implicitly). Moreover, the construction will also provide a morphism $f:C_\kukp \to C$ with property R, which satisfies some natural compatibilities: $f x_\kukp = x$, $f w_{i,\kukp} = w_i$, and a few others explained below. If ${\sh L}$ and $\sigma$ weren't in discussion, this would be nothing else but the famous stabilization procedure of Knudsen \cite[\S2]{[Kn83]}.
	
	\begin{example}\label{example: geometric Knudsen stabilization with sections of line bundles}
		We discuss the desired effect of the construction when $S = \spec {\mathbb K}$ and the input data is degenerate, in advance of the general statement. Observation \eqref{equation: crepant condition holds in first bubbling up} in this example will be referred to in Construction \ref{construction: first bubbling up construction}, though it is easy to see that there is no logical circularity, and this order facilitates exposition. 
		
		The local picture of Knudsen stabilization is reviewed in Figure \ref{figure: geometric Knudsen stabilization}. Specifically, if $x$ is a node of $C$ (or $x = w_i$ for some $i$), we insert a component $\Sigma \simeq {\mathbb P}^1$ at $x$, and place $x_\kukp$ (respectively $x_\kukp$ and $w_{i,\kukp}$) on $\Sigma$. 
		\begin{figure}[h]
			\begin{center}
				\begin{tikzpicture}
					\draw (-2,2) -- (2,3); \draw (-2,3) -- (2,2);  \fill[black] (0,2.5) circle (2pt) node[below] {$x$};
					\draw[<-] (2.5,2.5) to node[above] {$f$} (3.5,2.5); 
					\draw (4,2) -- (8,2); 
					\draw (4,3) -- (8,3); 
					\draw (6,2-0.2) -- (6,3.2); 
					\fill[black] (6,2.5) circle (2pt) node[right] {$x_{\kukp}$};
					
					\draw (-2,0) -- (2,0); \fill[black] (0,0) circle (2pt) node[below] {$x = w_j$};
					\draw[<-] (2.5,0.3) to node[above] {$f$} (3.5,0.3); 
					\draw (4,0) -- (8,0); \draw (6,-0.2) -- (6,1);
					\fill[black] (6,0.8) circle (2pt) node[right] {$x_{\kukp}$};
					\filldraw[fill=white] (6,0.4) circle (2pt) node[right] {$w_{j,\kukp}$};
				\end{tikzpicture}
			\end{center}
			\caption{Knudsen stabilization (local picture).}
			\label{figure: geometric Knudsen stabilization}
		\end{figure}
		The line bundle will be ${\sh L}_{\kukp} = f^\dsd {\sh L}$, and the choice of the section $\sigma_\kukp$ is forced by the agreement of $\sigma$ and $\sigma_\kukp$ on the open where $f$ is an isomorphism. Inspecting the two cases in Figure \ref{figure: geometric Knudsen stabilization}, we can also check that
		\begin{equation}\label{equation: crepant condition holds in first bubbling up}
			f^*\omega_C(w_1 + \cdots + w_m) \simeq \omega_{C_\kukp}(w_{1,\kukp} + \cdots + w_{m,\kukp}), 
		\end{equation}
		that is, condition \eqref{equation: crepant condition} is satisfied in the current setup.
		
		Finally, note that if ${\sh L} = \omega_C^\vee$, then ${\sh L}_\kukp = \omega_{C_\kukp}^\vee$, so this construction indeed generalizes the one called `Knudsen stabilization with vector fields' in \S\ref{subsection: Bubbling up and bubbling down}.
	\end{example}
	
	\begin{remark}\label{remark: universal property of projectivizations}
		We will frequently use the description of the functor of points of projectivizations of coherent modules, as written at the beginning of \cite[\S2]{[Kn83]}. This remark serves as an indirect link to the respective property, so, when we say `by Remark \ref{remark: universal property of projectivizations}', we actually mean `by the remark at the beginning of \cite[\S2]{[Kn83]}'.
	\end{remark}
	
	\begin{construction}\label{construction: first bubbling up construction}
		In Situation \ref{situation: Knudsen stabilization with sections of line bundles}, let $\delta:{\sh O}_C \to {\sh O}_C(w_1 + \cdots + w_m) \oplus {\sh I}_{x,C}^\vee$ be the homomorphism of ${\sh O}_C$-modules such that $\delta(1) = 1 \oplus \iota$, where $\iota:{\sh I}_{x,C} \hookrightarrow {\sh O}_C$ is the inclusion. Note that $\delta$ is injective, and let ${\sh K}$ be its cokernel. Then 
		\begin{equation}\label{equation: short exact sequence in first bubbling up}
			0 \to {\sh O}_C \to  {\sh O}_C(w_1 + \cdots +w_m) \oplus {\sh I}_{x,C}^\vee \to {\sh K} \to 0
		\end{equation}
		is exact. Define
		\begin{equation}
			C_\kukp = {\mathbb P} ({\sh K}),
		\end{equation}
		and let $f:C_\kukp \to C$ be the natural projection, and $\pi_\kukp = \pi f$. For the definition of the lifts $x_\kukp,w_{1,\kukp},\ldots,w_{m,\kukp}$ of $x,w_1,\ldots,w_m$, please see \cite[\S2]{[Kn83]}. As in \cite[\S2]{[Kn83]}, 
		\begin{enumerate}
			\item the morphism $\pi_\kukp:C_\kukp \to S$ is a prestable curve; 
			\item for each $s \in S$, $x_\kukp(s) \neq w_{i,\kukp}(s)$ for all $i$, and $\pi_\kukp$ is smooth at $x_\kukp(s)$; 
			\item $f$ has property R (by Lemma \ref{remark: checking rational contraction of fibers} and item \ref{item: Knudsen functoriality} below); 
			\item $f$ induces an isomorphism $f^{-1}(U) \simeq U$ if $U = \pi^{\operatorname{sm}} \backslash \bigcup_{i=1}^m (x \cap w_i)$; and 
			\item\label{item: Knudsen functoriality} the formation of the output data commutes with base change.
		\end{enumerate}
		(Although \cite[\S2]{[Kn83]} operates under the assumption that $(C,w_1,\ldots,w_m)$ is a \emph{stable} $m$-marked curve over $S$, while Situation \ref{situation: Knudsen stabilization with sections of line bundles} ensures \emph{prestability} only, this weaker assumption suffices to establish the claims above using precisely the same arguments as in \cite[\S2]{[Kn83]}. A similar situation occurs in \cite[\S1.3]{[LM00]}.) Let
		\begin{equation}\label{equation: definition of line bundle in first bubbling up}
			{\sh L}_\kukp = f^\dsd {\sh L},
		\end{equation}
		cf. Definition \ref{definition: dual shriek dual}. The homomorphism \eqref{equation: pushing forward logarithmic vector fields in general} from Proposition \ref{proposition: pushing forward vector fields in general} reads
		\begin{equation}\label{equation: logarithmic differential in first bubbling up}
			f_*({\sh L}_\kukp(-w_{1,\kukp}-\cdots-w_{m,\kukp})) \to {\sh L}(-w_1-\cdots-w_m)
		\end{equation}
		in our case, and in fact it is an isomorphism by Proposition \ref{proposition: pushing forward vector fields in general}, since the condition on geometric fibers \eqref{equation: crepant condition} holds, by \eqref{equation: crepant condition holds in first bubbling up}. (Indeed, we already know at least that the behaviour of the curve and markings agree with that described in Example \ref{example: geometric Knudsen stabilization with sections of line bundles}, so we may invoke \eqref{equation: crepant condition holds in first bubbling up}.) Define $\sigma_\kukp$ to be the preimage of $\sigma$ under \eqref{equation: logarithmic differential in first bubbling up}.
	\end{construction}
	
	\begin{remark}\label{remark: projective rational contraction in first type bubbling up}
		Here are a few remarks on Construction \ref{construction: first bubbling up construction}.
		\begin{enumerate}
			\item The map $f$ is projective in the sense of \cite[\S5.5]{[EGAII]}. Indeed, ${\sh I}_{x,C}^\vee$ is a finitely generated ${\sh O}_S$-module since the dual of a finitely generated module over a noetherian ring is finitely generated.
			\item\label{item: vanishing locus item in remark: projective rational contraction in first type bubbling up} The scheme-theoretic vanishing locus of $x_\kukp^*\sigma_\kukp$ coincides with the scheme-theoretic vanishing locus of $x^*\sigma$. We may argue as follows. Applying Lemma \ref{lemma: critrion for log differential to be isomorphism} with ${\sh L}^\vee$ in the role of ${\sh L}$ and dualizing, we obtain an isomorphism ${\sh L}_\kukp(-w_{1,\kukp}-\cdots-w_{m,\kukp}) \simeq f^*{\sh L}(-w_1-\cdots-w_m)$ whose adjoint maps $\sigma_\kukp$ to $\sigma$, by Construction \ref{construction: first bubbling up construction}. Pulling back along $x_\kukp$, we obtain an isomorphism $x_\kukp^*{\sh L}_\kukp(-w_{1,\kukp}-\cdots-w_{m,\kukp}) \simeq x^*{\sh L}(-w_1-\cdots-w_m)$ under which $x_\kukp^*\sigma_\kukp$ corresponds to $x^*\sigma$.
			\item\label{item: scalar multiplication in remark: projective rational contraction in first type bubbling up} If $(S,C,\pi,x,w_{1},\ldots,w_{m},{\sh L},\sigma)$ satisfies Situation \ref{situation: Knudsen stabilization with sections of line bundles}, and $\alpha \in \Gamma(S,{\sh O}_S)$, then $(S,C,\pi,x,w_{1},\ldots,w_{m},{\sh L},\pi^*\alpha\cdot \sigma)$ also satisfies Situation \ref{situation: Knudsen stabilization with sections of line bundles}, and, moreover, the output of Construction \ref{construction: first bubbling up construction} on the second set of input data will be (isomorphic to) $(S,C_\kukp,\pi_\kukp,x_\kukp,w_{1,\kukp},\ldots,w_{m,\kukp},{\sh L}_\kukp,\pi_\kukp^*\alpha \cdot\sigma_\kukp)$. Indeed, if $\psi$ denotes \eqref{equation: logarithmic differential in first bubbling up}, then $ \psi(\pi_\kukp^*\alpha \cdot \sigma_\kukp) = \psi(f^*\pi^*\alpha \cdot \sigma_\kukp) = \pi^*\alpha \cdot \psi(\sigma_\kukp) = \pi^*\alpha \cdot \sigma$, where the second equality is simply by ${\sh O}_C$-linearity.
		\end{enumerate}
	\end{remark}
	
	\begin{proposition}\label{proposition: first bubbling up is functorial}
		Construction \ref{construction: first bubbling up construction} commutes with base change.
	\end{proposition}
	
	\begin{proof}
		For $C_\kukp,\pi_\kukp,f,x_\kukp,w_{i,\kukp}$ this is completely analogous to \cite[p. 176]{[Kn83]} (the key is that ${\sh K}$ is stably reflexive, everything else is clear). For ${\sh L}_\kukp$, this is an easy consequence of Lemma \ref{lemma: upper shriek of invertible is invertible}, as noted after Definition \ref{definition: dual shriek dual} too. For $\sigma_\kukp$, this follows from Lemma \ref{lemma: critrion for log differential to be isomorphism}.
	\end{proof}
	
	To summarize, we have the following.
		
	\begin{theorem}\label{theorem: summary of first bubbling up}
		Construction \ref{construction: first bubbling up construction} has the following features:
		\begin{enumerate}
			\item the output data is nondegenerate;
			\item if the input data is nondegenerate, then the output data is isomorphic to the input data;
			\item if $S = \spec {\mathbb K}$, and the input data is degenerate, then the output data is (isomorphic to) that described in Example \ref{example: geometric Knudsen stabilization with sections of line bundles}; and
			\item it commutes with base change.
		\end{enumerate}
	\end{theorem}
	
	\begin{proof}
		The first two items are stated in Construction \ref{construction: first bubbling up construction}, the third item is elementary to check from Construction \ref{construction: first bubbling up construction} and left to the reader (and well-known, being completely analogous to \cite{[Kn83]}), while the last item is Proposition \ref{proposition: first bubbling up is functorial}. 
	\end{proof}
	
	Finally, we state a version of \cite[Lemma 2.5]{[Kn83]}, which will be used later to construct an inverse of this bubbling up operation, under suitable circumstances.

\begin{proposition}\label{proposition: towards up of down = id from Knudsen}
	Let $(S,C,\pi, x, {\sh L}, w_1,\ldots, w_m,\sigma)$ and $(S,Y,\varpi, y, {\sh J}, u_1,\ldots, u_m,\rho)$ satisfy the requirements of Situation \ref{situation: Knudsen stabilization with sections of line bundles}, and assume moreover that the second set of data is nondegenerate, i.e. $y(s) \neq u_i(s)$ for $i=1,\ldots,m$ and $\varpi$ is smooth at $y(s)$, for all $s \in S$. Let $q:Y \to C$ be a morphism with property R such that:
	\begin{enumerate}
		\item $x=qy$, $w_i=qu_i$, ${\sh J} = q^\dsd {\sh L}$, and the homomorphism 
		\[ q_*{\sh J}(-u_1-\cdots-u_m) \to {\sh L}(-w_1-\cdots-w_m) \] 
		obtained as a special case of \eqref{equation: pushing forward logarithmic vector fields in general} from ${\sh J} = q^\dsd {\sh L}$ maps $\rho$ to $\sigma$;
		\item\label{item: geometric fibers assumption in proposition: towards up of down = id from Knudsen} for any geometric point $\overline{s} \to S$, there exists an isomorphism $\alpha: Y_{\overline{s}} \to (C_{\kukp})_{\overline{s}}$ over $\overline{s}$, such that $q_{\overline{s}} = f_{\overline{s}} \alpha$, 
		and $\alpha$ is compatible in the natural sense with all the rest of the data (restrictions of line bundles, sections, etc.).
	\end{enumerate}
	Then, there exists a canonical isomorphism $q_\kukp:Y \to C_\kukp$ over $S$ (compatible with all given sections of $\varpi$ and $\pi_\kukp$, and with $\sigma_\kukp$ and $\rho$), such that $q = fq_\kukp$.
\end{proposition}

To emphasize the nuance of the proposition, the isomorphisms $Y_{\overline{s}} \simeq (C_{\kukp})_{\overline{s}}$ in item \ref{item: geometric fibers assumption in proposition: towards up of down = id from Knudsen} are not required a priori to `fit together' in any nice way, but the conclusion is essentially the statement that they do.

\begin{proof}
	The proof is analogous to the proof of \cite[Lemma 2.5]{[Kn83]}. By Definition \ref{definition: rational contraction}, ${\sh O}_C \simeq q_*{\sh O}_Y$, and, by part \ref{item 1: proposition: construction of logarithmic differential using duality} of Proposition \ref{proposition: construction of logarithmic differential using duality}, ${\sh O}_C(-W) \simeq q_*{\sh O}_Y(-U)$, where $W = \sum_{i=1}^m w_i$ and $U = \sum_{i=1}^m u_i$. Following the analogous arguments in the proof of \cite[Lemma 2.5]{[Kn83]}, the natural pairing
	\[ {\sh I}_{x,C} \otimes q_*{\sh O}_Y(y-U) \to q_*{\sh O}_Y(-y) \otimes q_*{\sh O}_Y(y-U) \to q_*{\sh O}_Y(-U) \simeq {\sh O}_C(-W) \]
	induces a homomorphism $q_*{\sh O}_Y(y - U) \to {\sh I}_{x,C}^\vee(-W)$, which is an isomorphism on geometric fibers by assumption \ref{item: geometric fibers assumption in proposition: towards up of down = id from Knudsen}, and hence an isomorphism. We have a commutative diagram with exact rows:
	\begin{center}
		\begin{tikzpicture}
			\matrix [column sep  = 4mm, row sep = 4mm] {
				\node (t0) {$0$}; &
				\node (t1) {$q_*{\sh O}_Y(-U)$}; &
				\node (t2) {$q_*{\sh O}_Y(y - U) \oplus q_*{\sh O}_Y$}; &
				\node (t3) {$q_*{\sh O}_Y(y)$}; &
				\node (t4) {$0$}; & \\
				\node (b0) {$0$}; &
				\node (b1) {${\sh O}_C(-W)$}; &
				\node (b2) {${\sh I}_{x,C}^\vee(-W) \oplus {\sh O}_C$}; &
				\node (b3) {${\sh K}(-W)$}; &
				\node (b4) {$0.$}; & \\
			};
			\draw[->] (t0) -- (t1);
			\draw[->] (t1) -- (t2);
			\draw[->] (t2) -- (t3);
			\draw[->] (t3) -- (t4);
			
			\draw[->] (b0) -- (b1);
			\draw[->] (b1) -- (b2);
			\draw[->] (b2) -- (b3);
			\draw[->] (b3) -- (b4);
			
			\draw[double equal sign distance] (b1) -- (t1);
			\draw[double equal sign distance] (b2) -- (t2);
			\draw[double equal sign distance] (b3) -- (t3);
		\end{tikzpicture}
	\end{center}
	Exactness of the top row follows from part \ref{item 1: proposition: construction of logarithmic differential using duality} of Proposition \ref{proposition: construction of logarithmic differential using duality}, while the bottom row is simply \eqref{equation: short exact sequence in first bubbling up} twisted by ${\sh O}_C(-W)$. 
	
	As in \cite{[Kn83]}, the next step is to apply \cite[Corollary 1.5]{[Kn83]}; for the reader's convenience, we include the details of verifying the hypotheses of the corollary. We claim that for any closed point $z \in C$ (with residue field denoted by $K$), either $q^{-1}(z) \simeq \spec K$, or $q^{-1}(z) \simeq {\mathbb P}^1_K$ and ${\sh O}_Y(y)|_{q^{-1}(z)} \simeq {\sh O}(1)$. First, let's assume that $z$ is contained in the maximal open of $C$ where $f$ is an isomorphism (Lemma \ref{lemma: dense f-trivial}). Then, in fact, $z$ is in the maximal open of $C$ where $q$ is an isomorphism, by assumption \ref{item: geometric fibers assumption in proposition: towards up of down = id from Knudsen}, the last item in Lemma \ref{lemma: dense f-trivial}, and the fact that two open subsets $U_1,U_2 \subseteq C$ which agree on `geometric fibers' (that is, $U_{1,\overline{s}} = U_{2,\overline{s}}$ as subsets of $C_{\overline{s}}$, for any geometric point $\overline{s} \to S$) must coincide. The last fact follows from the fact that $C_{\overline{s}} \to C_s$ is surjective (on the underlying topological spaces), which in turn follows from $\overline{s} \to s$ faithfully flat, and the fact that `faithfully flat' is preserved by base change. Hence, in this case, $q^{-1}(z) \simeq \spec K$. It remains to deal with the case when $z$ is not contained in the maximal open of $C$ where $f$ is an isomorphism (Lemma \ref{lemma: dense f-trivial}). Then, $z \in x(S)$ and the fiber of $(S,C,\pi,\ldots)$ over $s \in S$ is degenerate. Moreover, if $\pi(z) = s \in S$, then $z = x(s)$ necessarily and, in particular, $s \in S$ has the same residue field $K$ as $z \in C$. Since $(q^{-1}(z))_{\overline{s}}$ and $(f^{-1}(z))_{\overline{s}}$ are the fibers of $z \in C$ by the morphisms $Y_{\overline{s}} \to C$ and $(C_\kukp)_{\overline{s}} \to C$ respectively, assumption \ref{item: geometric fibers assumption in proposition: towards up of down = id from Knudsen} implies that they are isomorphic, so $(q^{-1}(z))_{\overline{s}} \simeq (f^{-1}(z))_{\overline{s}} \simeq {\mathbb P}^1_{\overline{K}}$. However, $q^{-1}(z)$ contains the $K$-point (i.e. rational point) $y(s) \in q^{-1}(z)$ by the degeneracy assumption and assumption \ref{item: geometric fibers assumption in proposition: towards up of down = id from Knudsen}, so $q^{-1}(z) \simeq {\mathbb P}^1_K$ since genus $0$ curves with rational points are projective lines. Moreover, ${\sh O}_Y(y)|_{q^{-1}(z)} \simeq {\sh O}(1)$ since $y(s) \in q^{-1}(z)$, completing the proof of the claims made earlier in this paragraph.
	
	Therefore, ${\sh O}_Y(y)$ satisfies conditions (1) and (2) in \cite[Corollary 1.5]{[Kn83]} relative to $q:Y \to C$.  By \cite[Corollary 1.5]{[Kn83]}, $q^*{\sh K}(-W) = q^*q_*{\sh O}_Y(y) \to {\sh O}_Y(y)$ is surjective. Therefore, by Remark \ref{remark: universal property of projectivizations}, we obtain a morphism 
	\[ q_{\kukp}:Y \to C_\kukp \]
	such that $q = fq_\kukp$. It can be checked from assumption \ref{item: geometric fibers assumption in proposition: towards up of down = id from Knudsen} that $q_{\kukp}$ is an isomorphism on geometric fibers (for the purposes of this claim, we may assume that $S = \spec {\mathbb K}$, when everything is straightforward), and therefore an isomorphism everywhere. All compatibilities are straightforward to check.
\end{proof}
	
	\subsection{Inflating at nonsingular zero of a section}\label{subsection: Bubbling up at nonsingular zero of a section} In \S\ref{subsection: Bubbling up at nonsingular zero of a section}, we generalize the operation `Inflating at zero vector' from \S\ref{subsection: Bubbling up and bubbling down} to families (and arbitrary line bundles).
	
	\begin{situation}\label{situation: Bubbling up at nonsingular zero of a section}
		Assume that $\pi:C \to S$ is a prestable curve, $x$ is a section of $\pi$, ${\sh L}$ is an invertible ${\sh O}_C$-module, and $\sigma$ is a global section of ${\sh L}$, such that $\pi$ is smooth at $x(s)$, for all $s \in S$. 
	\end{situation}
	
	We deem such data \emph{degenerate} if $\sigma$ vanishes at $x(s)$ for some $s \in S$ (vanishes at $x(s)$ means $\sigma_{x(s)} \in {\mathfrak m} {\sh L}_{x(s)}$, where ${\mathfrak m} \subset {\sh O}_{C,x(s)}$ is the maximal ideal).
	
	In Situation \ref{situation: Bubbling up at nonsingular zero of a section}, we will construct $S_\zuzp =S$, $C_\zuzp$, $\pi_\zuzp$, $x_\zuzp$, ${\sh L}_\zuzp$, $\sigma_\zuzp$ such that the new (output) data satisfies the respective requirements of Situation \ref{situation: Bubbling up at nonsingular zero of a section}, $x_\zuzp^*\sigma_\zuzp$ is nowhere vanishing (the output data is nondegenerate), and the construction commutes with base change. Once again, the construction will also provide a map $f:C_\zuzp \to C$ with property R, which satisfies various compatibilities, similarly to \S\ref{subsection: Knudsen stabilization with sections of line bundles}.
	
	\begin{example}\label{example: geometric inflating at nonsingular zero of a section}
		Similarly to Example \ref{example: geometric Knudsen stabilization with sections of line bundles}, we state the effect of the construction when $S = \spec {\mathbb K}$ and the data is degenerate in advance of its general statement. The local picture on the level of curves and sections is very simple, by Figure \ref{figure: inflating at nonsingular zero of a section} below. Specifically, we insert $\Sigma \simeq {\mathbb P}^1$ at $x$, and place $x_\zuzp$ on $\Sigma$.
		\begin{figure}[h]
			\begin{center}
				\begin{tikzpicture}
					\draw (-2,0) -- (2,0); \fill[black] (0,0) circle (2pt) node[below] {$x$};
					\draw (2,0) to node[above] {$\sigma(x) = 0$} (-2,0); 
					\draw[<-] (2.5,0.3) to node[above] {$f$} (3.5,0.3); 
					\draw (4,0) -- (8,0); \draw (6,-0.2) -- (6,1); \fill[black] (6,0.8) circle (2pt) node[right] {$x_{\zuzp}$};
					\draw[dotted] (4,0) to (8,0); \draw[dotted] (6,-0.2) to node[right] {$\sigma_{\zuzp}(x_{\zuzp}) \neq 0$} (6,1);
				\end{tikzpicture}
			\end{center}
			\caption{Inflating at nonsingular zero of a section (local picture).}
			\label{figure: inflating at nonsingular zero of a section}
		\end{figure}
		The line bundle is again ${\sh L}_{\zuzp} = f^{\dsd} {\sh L}$. The section $\sigma_\zuzp$ is required to not vanish at $x_{\zuzp}$ and to agree with $\sigma$ at all generic points of $C$; please see Remark \ref{remark: why is second bubbling up confusing} for further clarifications. In contrast to Example \ref{example: geometric Knudsen stabilization with sections of line bundles}, $f^*\omega_C \not\simeq \omega_{C_\zuzp}$, i.e. condition \eqref{equation: crepant condition} fails. Finally, if ${\sh L} = \omega_C^\vee$, then ${\sh L}_{\zuzp} = \omega_{C_\zuzp}^\vee$, so this construction indeed generalizes the construction called `inflating at zero vector' in \S\ref{subsection: Bubbling up and bubbling down}.
	\end{example}

	\begin{remark}\label{remark: why is second bubbling up confusing}
		The behaviour of $\sigma_\zuzp$ in Example \ref{example: geometric inflating at nonsingular zero of a section} is rather subtle. There are infinitely many sections satisfying the requirements, but they are `indistinguishable' up to automorphisms compatible with all structure (including with $f$). To sketch the argument, let us identify (non-canonically) $\Sigma$ with ${\mathbb P}^1$ such that $x_\zuzp$ maps to $0$ and the node of $C_\zuzp$ on $\Sigma$ maps to $\infty$. Then, the restriction of $\sigma_\zuzp$ to $\Sigma \simeq {\mathbb P}^1$ should correspond to a vector field $v$ on ${\mathbb P}^1$ such that $v(0) \neq 0$, $v(\infty) = 0$, and 
		\[ v = \left( \alpha y + b y^2 \right) \frac{d}{dy} = -\left( \alpha x + b \right) \frac{d}{dx} \quad \text{where $y=\frac{1}{x} = \frac{Y}{X} \in K({\mathbb P}^1$)}, \] 
		for a \emph{specified} $\alpha \in {\mathbb K}$, by Example \ref{example: example of f disappointed} and the considerations at the beginning of \S\ref{section: introduction}. (For instance, if ${\sh L}|_U \simeq \omega_C^\vee|_U$ is an isomorphism in a neighbourhood $U$ of $x$; then ${\sh L}_\zuzp|_{f^{-1}(U)} \simeq \omega_{C_\zuzp}^\vee|_{f^{-1}(U)}$, and the `logarithmic vector field' condition at the node of $C_\zuzp$ on $\Sigma$ shows that $\alpha$ is fixed. Please see Example \ref{example: coresidues in local coordinates} also.) 
		Hence, the claimed indistinguishability amounts to the fact that the standard ${\mathbb K}^\times$-action on ${\mathbb P}^1$, $a \cdot [X:Y] = [aX:Y]$, induces a (free and) transitive action on
		\[ \left\{ \left( c-\alpha x \right) \frac{d}{dx} : c \in {\mathbb K}^\times  \right\} \subset \Gamma({\sh T}_{{\mathbb P}^1}), \]
		which is true. Thus, the fact that a \emph{canonical} choice of $\sigma_\zuzp$ exists (as Construction \ref{construction: second bubbling up construction} shows) is quite remarkable. In Remark \ref{remark: first remark}, we proposed the analogy with similar subtleties in \cite{[Kn83]} as a way to come to terms with these issues.  
	\end{remark}
	
	We will first carry out the required calculations in Proposition \ref{proposition: all of second bubbling up elementary calculations} below, and only then return to Situation \ref{situation: Bubbling up at nonsingular zero of a section} and state the construction.
	
	\begin{proposition}\label{proposition: all of second bubbling up elementary calculations}
		Let $X$ be a (noetherian) scheme and $\sigma$ a section of a rank $2$ locally free ${\sh O}_X$-module ${\sh E}$. Assume that the homomorphism ${\sh O}_X \to {\sh E}$ which maps $1$ to $\sigma$ is injective. Let 
		$$ Y = {\mathbb P}{\sh F} \quad \text{where} \quad {\sh F} = \operatorname{Coker}({\sh O}_X \xrightarrow{\times \sigma} {\sh E}), $$
		$f:Y \to X$ the projection, and ${\sh O}_Y(1)$ the twisting sheaf.
		\begin{enumerate}
			\item\label{item 1: proposition: all of second bubbling up elementary calculations} The following hold.
			\begin{enumerate}
				\item\label{item 1a: proposition: all of second bubbling up elementary calculations} $f$ is a projective (in the sense of \cite[\S5.5]{[EGAII]}) global lci morphism.
				\item\label{item 1b: proposition: all of second bubbling up elementary calculations} The adjoint ${\sh F} \to f_*{\sh O}_Y(1)$ of $f^*{\sh F} \to {\sh O}_Y(1)$ is an isomorphism. Moreover, this isomorphism is compatible with base changes $h:X' \to X$ such that the multiplication by $h^*\sigma$ map ${\sh O}_{X'} \to h^*{\sh E}$ is injective.
				\item\label{item 1c: proposition: all of second bubbling up elementary calculations} There exists an isomorphism
				$f^!{\sh O}_X \simeq (f^*\det {\sh F})(-1)$; in particular, we think of $f^!{\sh O}_X$ as a sheaf. 
				Moreover, if $U$ is an open subset of $X$ such that $\sigma$ vanishes at no point of $U$, then $f$ restricts to an isomorphism $V:=f^{-1}(U) \simeq U$, and there exists a commutative diagram of isomorphisms
				\begin{center}
					\begin{tikzpicture}
						\node (a) at (-3,0) {$(f^!{\sh O}_X)|_V$};
						\node (b) at (0,0) {${\sh O}_V$};
						\node (c) at (3.7,0) {$((f^*\det {\sh F})(-1))|_V$};
						
						\draw [<-] (b) -- node [midway, above] {$\simeq$} (a);
						\draw [<-] (b) -- node [midway, above] {$\simeq$} (c);
						\draw [->] (a) to [out = 15, in = 165] (c);
					\end{tikzpicture}
				\end{center}
				in which the curved arrow is the restriction of the isomorphism above, the left arrow comes from $f^!|_V=f^*|_V$, since $f|_V$ is an isomorphism, and the one on the right commutes with any base change $X' \to X$.
			\end{enumerate}
			\item\label{item 2: proposition: all of second bubbling up elementary calculations} Assume that $g:X \to T$ is a flat morphism, and that for all morphisms $T' \to T$, the multiplication by $h^*\sigma$ map ${\sh O}_{X'} \to h^*{\sh E}$ is injective (where $h:X' = T' \times_T X \to X$ is the projection). Then $fg$ is flat.
			\item\label{item nodal: proposition: all of second bubbling up elementary calculations} If $X$ is a prestable curve over ${\mathbb K}$ and the vanishing locus of $\sigma$ is  a reduced closed subscheme of dimension $0$ contained in the nonsingular locus of $X$, then $Y$ is also a prestable curve over ${\mathbb K}$, and $f$ has property R.
		\end{enumerate}
	\end{proposition}
	
	\begin{proof}
		The assumption that the map ${\sh O}_X \to {\sh E}$ such that $1 \mapsto \sigma$ (i.e. the multiplication by $\sigma$ map) is injective in fact implies a seemingly stronger hypothesis which will be used in this proof, that  the germ $\sigma_x$ of $\sigma$ at $x$ is nonzero, and not a zero-divisor in the ring $\operatorname{Sym}{\sh E}_x$, for all $x \in X$. We will sometimes refer to this fact as the `condition on germs'. Indeed, $1 \mapsto \sigma_x$ is an injective homomorphism ${\sh O}_{X,x} \to {\sh E}_x$, and then the claim follows from a theorem of McCoy \cite[Theorem 3]{[Mc42]} that zero-divisors in polynomial rings are in fact annihilated by non-zero scalars (in our case, applied to a linear homogeneous polynomial in two variables).
			
		We have a short exact sequence
		\begin{equation}\label{equation: simple exact sequence in proposition: all of second bubbling up elementary calculations}
			0 \to {\sh O}_X \to {\sh E} \to {\sh F} \to 0.
		\end{equation}
		Let $P = {\mathbb P}{\sh E}$, and $p:P \to X$ the projection which exhibits $P$ as a ${\mathbb P}^1$-bundle over $X$. Since $\operatorname{Sym}^k {\sh E} \to \operatorname{Sym}^k {\sh F}$ is surjective for all $k$, it follows from \cite[\href{https://stacks.math.columbia.edu/tag/07ZK}{Tag 07ZK}]{[stacks]} that the graded ${\sh O}_X$-algebras homomorphism $\operatorname{Sym} {\sh E} \to \operatorname{Sym} {\sh F}$ induces a closed immersion $j:Y \to P$. 
		
		\begin{claim}\label{claim: claim inside proposition: all of second bubbling up elementary calculations}
			If $\psi$ is the image of $p^*\sigma$ under the map $p^*{\sh E} \to {\sh O}_P(1)$, then $\psi$ is a regular section of ${\sh O}_P(1)$, and $j(Y)$ is the Cartier divisor on $P$ cut out by $\psi$.
		\end{claim}
		
		\begin{proof}
			We may assume that $X = \spec R$ and that ${\sh E}$ is trivial. Then $P = \mathrm{Proj}\mathop{}R[T,S]$, and let $\sigma$ and $\psi$ correspond to the linear polynomial $aT+bS \in R[T,S]$. First, we claim that $aT+bS$ is a nonzero non-zero-divisor in $R[T,S]$. Indeed, the `condition on germs' implies that this is the case in $R_{\mathfrak p}[T,S]$ for all prime ideals ${\mathfrak p} \in \spec R$, so any $p \in R[T,S]$ such that $(aT+bS)p(T,S) = 0$ is killed by all homomorphisms $R[T,S] \to R_{\mathfrak p}[T,S]$, and thus must be equal to $0$. Then $at+b$ and $a+bs$ are nonzero non-zero-divisors in $R[t]$ and $R[s]$, so $\psi$ indeed cuts out a Cartier divisor on $P$. Moreover, by construction, this Cartier divisor is $\mathrm{Proj}\mathop{}R[T,S]/(aT+bS) \to \mathrm{Proj}\mathop{}R[T,S]$, which is nothing but $j$.
		\end{proof}
		
		\ref{item 1a: proposition: all of second bubbling up elementary calculations}: The $f=pj$ factorization settles both issues. 
		
		\ref{item 1b: proposition: all of second bubbling up elementary calculations}: Note that ${\sh I}_{j(Y),P} = {\sh O}_P(-1)$ by Claim \ref{claim: claim inside proposition: all of second bubbling up elementary calculations} and $j_*{\sh O}_Y \otimes {\sh O}_P(1) = j_*({\sh O}_Y(1))$ by construction. Twisting $0 \to {\sh I}_{j(Y),P} \to {\sh O}_P \to j_*{\sh O}_Y \to 0$ by ${\sh O}_P(1)$, we obtain the exact sequence
		\begin{equation}\label{equation: other exact sequence in proposition: all of second bubbling up elementary calculations}
			0 \to {\sh O}_P \to {\sh O}_P(1) \to j_*{\sh O}_Y(1) \to 0.
		\end{equation}
		We have $R^1p_*{\sh O}_P = 0$, and ${\sh O}_X \to p_*{\sh O}_P$ and ${\sh E} \to p_*{\sh O}_P(1)$ are isomorphisms (the claims are local on $X$, so they reduce to $X$ affine and ${\sh E}$ trivial), so we have a commutative diagram
		\begin{center}
			\begin{tikzpicture}[scale=1]
				\node (a0) at (0,1) {$0$};
				\node (a1) at (1.2,1) {${\sh O}_X$};
				\node (a2) at (3,1) {${\sh E}$};
				\node (a3) at (5,1) {${\sh F}$};
				\node (a4) at (6.4,1) {$0$};
				\node (b0) at (0,0) {$0$};
				\node (b1) at (1.2,0) {$p_*{\sh O}_P$};
				\node (b2) at (3,0) {$p_*{\sh O}_P(1)$}; 
				\node (b3) at (5,0) {$f_*{\sh O}_Y(1)$}; 
				\node (b4) at (6.4,0) {$0$};
				\draw[->] (b0) to (b1);
				\draw[->] (b1) to (b2);
				\draw[->] (b2) to (b3);
				\draw[->] (b3) to (b4);
				\draw[->] (a0) to (a1);
				\draw[->] (a1) to (a2);
				\draw[->] (a2) to (a3);
				\draw[->] (a3) to (a4);
				\draw[->] (a1) -- node [midway, left] {$\simeq$} (b1);
				\draw[->] (a3) to (b3);
				\draw[->] (a2) -- node [midway, left] {$\simeq$} (b2);	
			\end{tikzpicture}
		\end{center}
		in which the top row is \eqref{equation: simple exact sequence in proposition: all of second bubbling up elementary calculations}, the bottom row is the pushforward of \eqref{equation: other exact sequence in proposition: all of second bubbling up elementary calculations} along $p$, and the left and central vertical arrows are isomorphisms. It follows that the right vertical arrow is also an isomorphism, as desired. For compatibility with base change, we first clarify that the map $h^*f_*{\sh O}_Y(1) \to f'_*{\sh O}_{Y'}(1)$ is the usual map discussed in context of `cohomology and base change'. It is clear that $h^*{\sh F} \cong {\sh F}'$, since \eqref{equation: simple exact sequence in proposition: all of second bubbling up elementary calculations} remains exact after applying $h^*$. The commutativity of the square diagram which expresses that ${\sh F} \to f_* {\sh O}_Y(1)$ commutes with base change is standard; in particular, $h^*f_*{\sh O}_Y(1) \to f'_*{\sh O}_{Y'}$ is a fortiori an isomorphism, since the other three maps in the square are isomorphisms.
		
		\ref{item 1c: proposition: all of second bubbling up elementary calculations}: Taking determinants in the Euler sequence of $P$, we obtain
		\begin{equation}\label{equation: relative dualizing sheaf of rank 2 bundle}
			\omega_{P/X} = p^*\det{\sh E}\otimes {\sh O}_P(-2).
		\end{equation} 
		Keeping in mind that $j$ is the immersion of an effective Cartier divisor, we have
		\begin{equation}\label{equation: upper shriek of structure sheaf in general for projectivization}
			\begin{aligned}
				f^!{\sh O}_X &= j^!p^!{\sh O}_X = j^!(\omega_{P/X}[1]) = j^*\omega_{P/X}[1] \otimes j^!{\sh O}_P \\
				&= j^*\omega_{P/X}[1] \otimes j^*{\sh O}_P(j(Y))[-1] \quad \text{by \cite[\href{https://stacks.math.columbia.edu/tag/0B4B}{Tag 0B4B}]{[stacks]}} \\
				&= f^*\det{\sh E}\otimes {\sh O}_Y(-1) \quad \text{by \eqref{equation: relative dualizing sheaf of rank 2 bundle} and ${\sh O}_P(j(Y)) \cong {\sh O}_P(1)$,} \\
				&= f^*\det{\sh F}\otimes {\sh O}_Y(-1) \quad \text{since \eqref{equation: simple exact sequence in proposition: all of second bubbling up elementary calculations} implies $\det {\sh F} = \det {\sh E}$,}
			\end{aligned}
		\end{equation}
		as desired.
		
		For the claim regarding the restriction to $U$, it is alright to assume that $U = X$ since both \eqref{equation: relative dualizing sheaf of rank 2 bundle} and \eqref{equation: upper shriek of structure sheaf in general for projectivization} obviously commute with restricting to open subschemes of $X$. With this assumption, ${\sh F}$ is invertible and isomorphic to $\det {\sh E}$, and $Y = {\mathbb P}{\sh F} = X$ with ${\sh O}_Y(1) = {\sh F}$. This provides the isomorphism on the right, and the rest is clear.
		
		\ref{item 2: proposition: all of second bubbling up elementary calculations}: Clearly, $P \to T$ is flat, because it is a composition of flat morphisms $P \to X \to T$. Recall from Claim \ref{claim: claim inside proposition: all of second bubbling up elementary calculations} that the section $\psi$ of ${\sh O}_P(1)$ cuts out the effective Cartier divisor $Y$ on $P$. However, the assumption in the statement of \ref{item 2: proposition: all of second bubbling up elementary calculations} implies that the latter remains true after any base change $T' \to T$, so we may conclude by \cite[\href{https://stacks.math.columbia.edu/tag/056Y}{Tag 056Y}]{[stacks]}.
		
		\ref{item nodal: proposition: all of second bubbling up elementary calculations}: Indeed, $j(Y) = \Sigma + p^{-1}(\{\sigma = 0\})$ as divisors on $P$ for some section $\Sigma \subset P$ of the ${\mathbb P}^1$-bundle $P \to X$, and the claim follows easily. (Lemma \ref{remark: checking rational contraction of fibers} is implicitly relied on.)
	\end{proof}
	
	We can now state the main construction.
	
	\begin{construction}\label{construction: second bubbling up construction}
		In Situation \ref{situation: Bubbling up at nonsingular zero of a section}, consider the section $(-\sigma,1)$ of ${\sh L} \oplus {\sh O}_C(x)$. (Note the sign!) We are in a situation covered by Proposition \ref{proposition: all of second bubbling up elementary calculations}, with $C$ in the role of $X$, ${\sh L} \oplus {\sh O}_C(x)$ in the role of ${\sh E}$, and $(-\sigma,1)$ in the role of $\sigma$. 
		
		Let $\gamma:{\sh O}_C \to {\sh L} \oplus {\sh O}_C(x)$ be the homomorphism such that $\gamma(1) = (-\sigma, 1)$. Then
		\begin{equation}\label{equation: initial exact sequence in bubbling up at nonsingular zero of a section}
			0 \to {\sh O}_C \xrightarrow{\gamma} {\sh L} \oplus {\sh O}_C(x) \xrightarrow{\kappa} {\sh K} \to 0
		\end{equation}
		is exact, where ${\sh K}$ is the cokernel of $\gamma$. Define
		\begin{equation}
			C_\zuzp = {\mathbb P} ({\sh K}),
		\end{equation}
		and let $f:C_\zuzp \to C$ be the natural projection, and $\pi_\zuzp = \pi f$. The complex ${\sh O}_S \xrightarrow{x^*\gamma}  x^*{\sh L} \oplus x^*{\sh O}_C(x) \to x^*{\sh O}_C(x) \to 0$ is exact at $x^*{\sh O}_C(x)$, and we get a surjective homomorphism $x^*{\sh K} \to x^*{\sh O}_C(x)$ by the universal property of cokernels, since $x^*$ is right exact. This map induces a lift
		\begin{equation}
			x_\zuzp:S \to C_\zuzp
		\end{equation}
		of $x:S \to C$ by Remark \ref{remark: universal property of projectivizations}. It is clear that all constructions so far commute with base change.
		
		In particular, part \ref{item nodal: proposition: all of second bubbling up elementary calculations} of Proposition \ref{proposition: all of second bubbling up elementary calculations} implies that the geometric fibers of $\pi_\zuzp$ are curves with at worst nodal singularities. On the other hand, part \ref{item 2: proposition: all of second bubbling up elementary calculations} of Proposition \ref{proposition: all of second bubbling up elementary calculations} implies that $\pi_\zuzp$ is flat, since, once more, our constructions are functorial and the non-zero-divisor condition holds universally just as well as it holds in the given case -- simply repeat the first paragraph in this construction after the base change. Thus $\pi_\zuzp$ is a prestable curve. Part \ref{item nodal: proposition: all of second bubbling up elementary calculations} of Proposition \ref{proposition: all of second bubbling up elementary calculations} again, the functoriality of the construction so far, and Lemma \ref{remark: checking rational contraction of fibers} imply that $f$ has property R. Let
		\begin{equation}\label{equation: definition of line bundle in secod bubbling up}
			{\sh L}_\zuzp = f^\dsd {\sh L},
		\end{equation}
		cf. Definition \ref{definition: dual shriek dual} and the remark thereafter that $f^{\dsd} {\sh L}$ is an invertible sheaf. Some preliminary calculations are needed before we can define $\sigma_\zuzp$. We have
		\begin{equation}\label{equation: calculation of upper sigh in second bubbling up}
			\begin{aligned}
				f^\dsd {\sh L} &= f^*{\sh L} \otimes f^\dsd {\sh O}_C \quad \text{by the analogous property for $f^!$} \\
				&= f^*{\sh L} \otimes f^*{\sh L}^\vee(-x) \otimes {\sh O}_{C_\zuzp}(1) \quad \text{by part \ref{item 1c: proposition: all of second bubbling up elementary calculations} of Proposition \ref{proposition: all of second bubbling up elementary calculations} } \\
				&= f^*{\sh O}_C(-x) \otimes {\sh O}_{C_\zuzp}(1) \\
				&= {\sh O}_{{\mathbb P}{\sh K}(-x)}(1) \quad \text{by a well-known fact,}
			\end{aligned} 
		\end{equation}
		and it follows that
		\begin{equation}\label{equation: calculation of pushforward of upper sigh in second bubbling up}
			\begin{aligned}
				f_*f^\dsd {\sh L} &= f_*(f^*{\sh O}_C(-x) \otimes {\sh O}_{C_\zuzp}(1)) \quad \text{by \eqref{equation: calculation of upper sigh in second bubbling up}} \\
				&= {\sh O}_C(-x) \otimes f_* {\sh O}_{C_\zuzp}(1) \quad \text{by the projection formula} \\
				&= {\sh K}(-x) \quad \text{by part \ref{item 1b: proposition: all of second bubbling up elementary calculations} of Proposition \ref{proposition: all of second bubbling up elementary calculations}.}
			\end{aligned}
		\end{equation}
		If we combine the composition ${\sh O}_C \xrightarrow{0 \oplus \operatorname{id}} {\sh L}(-x) \oplus {\sh O}_C \xrightarrow{\kappa \otimes \operatorname{id}} {\sh K}(-x)$ with \eqref{equation: calculation of pushforward of upper sigh in second bubbling up}, we obtain an ${\sh O}_C$-module homomorphism ${\sh O}_C \to f_*f^\dsd{\sh L} = f_*{\sh L}_\zuzp$. Let
		\begin{equation}
			\sigma_\zuzp \in \Gamma(C_\zuzp,{\sh L}_\zuzp)
		\end{equation}
		be the image of $1 \in \Gamma(C,{\sh O}_C)$ under the last homomorphism.
	\end{construction}
	
	\begin{remark}\label{remark: Remarks on construction of second bubbling up}
		Here are some technical remarks on Construction \ref{construction: second bubbling up construction}.
		\begin{enumerate}
			\item\label{item: sign convention in remark: Remarks on construction of second bubbling up} The composition ${\sh L}|_{C\backslash x} \xrightarrow{\operatorname{id} \oplus 0} ({\sh L} \oplus {\sh O}_C(x))|_{C \backslash x} \xrightarrow{\kappa} {\sh K}|_{C \backslash x}$ is an isomorphism ${\sh L}|_{C\backslash x} \simeq {\sh K}|_{C \backslash x}$. The section $1 \in \Gamma(C\backslash x,{\sh O}_C)$ is mapped to $\sigma|_{C\backslash x}$ under the composition 
			\begin{equation*}
				{\sh O}_C|_{C\backslash x}={\sh O}_C(x)|_{C\backslash x} \xrightarrow{0 \oplus \operatorname{id} } ({\sh L} \oplus {\sh O}_C(x))|_{C \backslash x} \to {\sh K}|_{C \backslash x} \simeq {\sh L}|_{C\backslash x},
			\end{equation*} 
			which begins to explain the sign convention in Construction \ref{construction: second bubbling up construction}.
			\item\label{item: isomorphic restriction in remark: Remarks on construction of second bubbling up} The map $f$ restricts to an isomorphism $f^{-1}(C \backslash x) \simeq C \backslash x$ since ${\sh L}|_{C\backslash x} \simeq {\sh K}|_{C \backslash x}$, cf. item \ref{item: sign convention in remark: Remarks on construction of second bubbling up}. 
			\item\label{item: description of adjoint in remark: Remarks on construction of second bubbling up} The adjoint of \eqref{equation: calculation of pushforward of upper sigh in second bubbling up} is the homomorphism $f^*{\sh K}(-x) \to f^!{\sh L}$ which induces (via Remark \ref{remark: universal property of projectivizations}) the isomorphism $C_\zuzp \simeq {\mathbb P}{\sh K}(-x) = {\mathbb P}{\sh K}$.
			\item\label{item: what if no degeneracy in remark: Remarks on construction of second bubbling up} If the data $(S,C,\ldots)$ from Situation \ref{situation: Bubbling up at nonsingular zero of a section} is nondegenerate, then the output data $(S_\zuzp=S,C_\zuzp,\ldots)$ is canonically isomorphic to the input data $(S,C,\ldots)$. We also note that \eqref{equation: initial exact sequence in bubbling up at nonsingular zero of a section} takes the form $0 \to {\sh O}_C \to {\sh L} \oplus {\sh O}_C(x) \to {\sh L}(x) \to 0$ in this situation.
		\end{enumerate}
	\end{remark}
	
	First, let us check that Construction \ref{construction: second bubbling up construction} commutes with base change.  
	
	\begin{proposition}\label{proposition: second bubbling up is functorial}
		The formation of $C_\zuzp$, $\pi_\zuzp$, $x_\zuzp$, ${\sh L}_\zuzp$, $\sigma_\zuzp$, and $f$ commutes with base change.
	\end{proposition}
	
	\begin{proof}
		The upshot is that the relevant constructions visibly commute with base changes $S' \to S$, with the exception of those involving duality (as customary, we will not check the `cocycle conditions' for compositions $S'' \to S' \to S$ for the isomorphisms which express commutativity with respect to base change). Thus functoriality of $C_\zuzp,\pi_\zuzp,f,x_\zuzp$ is straightforward. The fact that ${\sh L}_\zuzp$ commutes with base change comes from ${\sh L}_\zuzp = f^*{\sh L} \otimes f^*\omega_{C/S} \otimes \omega^\vee_{C_\zuzp/S}$ (by \eqref{equation: definition of line bundle in secod bubbling up}, Definition \ref{definition: dual shriek dual}, and Lemma \ref{lemma: upper shriek of invertible is invertible}) and the fact that the formation of the relative dualizing sheaf of prestable curves commutes with base change, e.g. \cite[\href{https://stacks.math.columbia.edu/tag/0E6R}{Tag 0E6R}]{[stacks]}. 
		
		The functoriality of most ingredients involved in the functoriality of $\sigma_\zuzp$ is quite straightforward and left to the reader, with one exception: the application of part \ref{item 1c: proposition: all of second bubbling up elementary calculations} of Proposition \ref{proposition: all of second bubbling up elementary calculations}. Specifically, given a base change $h:S' \to S$, let $C',\pi',x'_\zuzp,{\sh L}',\sigma'_\zuzp$ be the pullback of the data in Situation \ref{situation: Bubbling up at nonsingular zero of a section} along $h$, if $q:C'_\zuzp \to C_\zuzp$ is the induced morphism, then we need the diagram
		\begin{center}
			\begin{tikzpicture}
				\node (a) at (0,0) {$f'^!{\sh O}_{C'}$};
				\node (b) at (0,1.1) {$q^*f^!{\sh O}_C$};
				\node (c) at (3,0) {$f'^*{\sh L}'(x') \otimes {\sh O}_{C'_\zuzp}(1)$};
				\node (d) at (3,1.1) {$q^*(f^*{\sh L}(x) \otimes {\sh O}_{C_\zuzp}(1))$};
				
				\draw [->] (b) -- (a);
				\draw [->] (d) -- (c);
				\draw [->] (a) -- (c);
				\draw [->] (b) -- (d);
			\end{tikzpicture}
		\end{center}
		in which all sheaves are invertible ${\sh O}_{C_\zuzp}$-modules, and all homomorphisms are isomorphisms (the horizontal ones come from the application of part \ref{item 1c: proposition: all of second bubbling up elementary calculations} of Proposition \ref{proposition: all of second bubbling up elementary calculations}, the left vertical one comes from $f^!{\sh O}_C = \omega_{C_\zuzp/S} \otimes f^* \omega_{C/S}^\vee$ and the fact that the formation of relative dualizing sheaves commutes with base change) to commute. Luckily, we can circumvent  the issue of base change in Grothendieck duality \cite{[Co00]} using the following trick. The commutativity of the square boils down to a statement of the form that an automorphism (say, $\alpha$) of $q^*f^!{\sh O}_C$ is the identity (compose all arrows around the square, two reversed). We have $$ {\sh A}ut(q^*f^!{\sh O}_C) \cong {\sh O}_{C'_\zuzp}^\times \quad \text{(the sheaf of nowhere vanishing functions on $C'_\zuzp$)}$$ 
		since $q^*f^!{\sh O}_C$ is invertible. However, by the second half of part \ref{item 1c: proposition: all of second bubbling up elementary calculations} of Proposition \ref{proposition: all of second bubbling up elementary calculations}, it is clear that the restriction of the square diagram above to $(f')^{-1}(C' \backslash x')$ commutes, so the invertible section corresponding to $\alpha$ restricts to $1$ on $(f')^{-1}(C' \backslash x')$. It then suffices to check that the restriction map on sections of the structure sheaf of ${C'_\zuzp}$ from global sections to sections over $C'_\zuzp \backslash (f')^{-1}(x')$ is injective. Since $f'^\#$ is an isomorphism, this boils down to the statement that the restriction map on sections of ${\sh O}_{C'}$ from global sections to sections over $C' \backslash x'$ is injective, which follows from Lemma \ref{lemma: injective restrictions on prestable curves}.
	\end{proof}

	Second, let us check that $\sigma_\zuzp$ is indeed a lift of $\sigma$.
	
	\begin{lemma}\label{lemma: vector fields match in second bubbling up}
		In the situation of Construction \ref{construction: second bubbling up construction}, the homomorphism \eqref{equation: pushing forward logarithmic vector fields in general} reads 
		\begin{equation}\label{equation: map in lemma: vector fields match in second bubbling up} f_*{\sh L}_\zuzp = f_*f^\dsd{\sh L} \to {\sh L}. \end{equation} 
		Then the map $\Gamma(C_\zuzp,{\sh L}_\zuzp) \to \Gamma(C,{\sh L})$ on global sections maps $\sigma_\zuzp \mapsto \sigma$.
	\end{lemma}
	
	\begin{proof}
		It suffices to check that $f_*{\sh L}_\zuzp \to {\sh L}$ maps the restriction of $\sigma_\zuzp$ to $f^{-1}(C \backslash x)$ to the restriction of $\sigma$ to $C \backslash x$. Indeed, $\Gamma(C,{\sh L}) \to \Gamma(C \backslash x, {\sh L})$ is injective by Lemma \ref{lemma: injective restrictions on prestable curves}, so this suffices. We have ${\sh K}(-x)|_{C \backslash x} \simeq {\sh L}|_{C \backslash x}$ by item \ref{item: sign convention in remark: Remarks on construction of second bubbling up} in Remark \ref{remark: Remarks on construction of second bubbling up} and $f_*f^\dsd{\sh L}|_{C \backslash x} \simeq {\sh L}|_{C \backslash x}$ by item \ref{item: isomorphic restriction in remark: Remarks on construction of second bubbling up} in Remark \ref{remark: Remarks on construction of second bubbling up} compatibly with the restriction of \eqref{equation: calculation of pushforward of upper sigh in second bubbling up}, that is, the resulting triangle is commutative. To justify this compatibility, it is necessary to revisit \eqref{equation: calculation of pushforward of upper sigh in second bubbling up} and \eqref{equation: calculation of upper sigh in second bubbling up}, but everything is clear. Then the diagram
		\begin{center}
			\begin{tikzpicture}
				\node (a) at (0,0) {${\sh O}_C|_{C \backslash x}$};
				\node (b) at (3,0) {${\sh K}(-x)|_{C \backslash x}$};
				\node (c) at (6,0) {${\sh L}|_{C \backslash x}$};
				\node (d) at (4.5,-1.3) {$f_*f^\dsd{\sh L}|_{C \backslash x}$};
				
				\draw[double equal sign distance] (b) -- (c);
				\draw[double equal sign distance] (c) -- (d);
				\draw[double equal sign distance] (b) -- node [midway, left] {$\eqref{equation: calculation of pushforward of upper sigh in second bubbling up}|_{C \backslash x}$} (d);
				\draw [->] (a) -- (b);
				\draw [->] (a) to [out = 300, in = 170] (d);
			\end{tikzpicture}
		\end{center}
		in which the curved arrow is the restriction of the homomorphism ${\sh O}_C \to f_*f^\dsd{\sh L}$ used in Construction \ref{construction: second bubbling up construction} is obviously commutative. It remains to see where $1 \in \Gamma(C \backslash x,{\sh O}_C)$ ends up in the diagram above. On one hand, its image in ${\sh L}|_{C \backslash x}$ is $\sigma|_{C \backslash x}$ by item \ref{item: sign convention in remark: Remarks on construction of second bubbling up} in Remark \ref{remark: Remarks on construction of second bubbling up}. On the other hand, its image in $f_*f^\dsd{\sh L}|_{C \backslash x}$ is the restriction of $\sigma_\zuzp$ to $f^{-1}(C \backslash x)$ by Construction \ref{construction: second bubbling up construction}. Then the claim in the beginning of this proof follows, since the diagram is commutative, and $f_*f^\dsd{\sh L} \to {\sh L}$ restricts on $C\backslash x$ to the isomorphism in the diagram, cf. part \ref{item 2: proposition: construction of logarithmic differential using duality} of Proposition \ref{proposition: construction of logarithmic differential using duality}.
	\end{proof}

\begin{remark}\label{remark: scaling field in second bubbling up}
	Here are some remarks related to Lemma \ref{lemma: vector fields match in second bubbling up}.
	\begin{enumerate}
		\item\label{item: item 1 in remark: scaling field in second bubbling up} It may or may not be the case that there exists a unique global section of ${\sh L}_{\zuzp}$ which maps to $\sigma$ by \eqref{equation: map in lemma: vector fields match in second bubbling up}. 
		\begin{enumerate}
			\item For instance, when $S = \spec {\mathbb K}$ and the data is degenerate, this is not the case, as explained at length in Remark \ref{remark: why is second bubbling up confusing}.
			\item\label{item: item 1b in remark: scaling field in second bubbling up} However, if, for instance, we assume that $C$ and $C_\zuzp$ are integral, and $f$ is a birational morphism, then $\Gamma(C_\zuzp,{\sh L}_{\zuzp}) \to \Gamma(C,{\sh L})$ is injective, so $\sigma$ can have only one preimage. Note also that, in this case, if $\zeta$ and $\zeta_{\zuzp}$ are the generic points of $C$ and $C_\zuzp$ and $K = K(C) = K(C_\zuzp)$ is the field of fractions, then there is a natural identification of the fibers at the generic points $K \otimes {\sh L} \simeq K \otimes {\sh L}_{\zuzp}$ as (one-dimensional) $K$-vector spaces, under which $\sigma(\zeta) = \sigma_{\zuzp}(\zeta_{\zuzp})$. 
		\end{enumerate}
		\item\label{item: item 2 in remark: scaling field in second bubbling up} Note that if $(S,C,\pi,x,{\sh L},\sigma)$ satisfies Situation \ref{situation: Bubbling up at nonsingular zero of a section}, and $\alpha \in \Gamma(S,{\sh O}_S^\times)$, then $(S,C,\pi,x,{\sh L},\pi^*\alpha \cdot \sigma)$ also satisfies Situation \ref{situation: Bubbling up at nonsingular zero of a section}. Let $(S,C'_\zuzp,\pi'_\zuzp,x'_{\zuzp},{\sh L}'_{\zuzp},\sigma'_{\zuzp})$ be the output of Construction \ref{construction: second bubbling up construction} on this data. The commutative diagram 
		\begin{center}
			\begin{tikzpicture}
				\matrix [column sep  = 18mm, row sep = 8mm] {
					\node (t0) {$0$}; &
					\node (t1) {${\sh O}_C$}; &
					\node (t2) {${\sh L} \oplus {\sh O}_C(x)$}; &
					\node (t3) {${\sh K}$}; &
					\node (t4) {$0$}; & \\
					\node (b0) {$0$}; &
					\node (b1) {${\sh O}_C$}; &
					\node (b2) {${\sh L} \oplus {\sh O}_C(x)$}; &
					\node (b3) {${\sh K}$}; &
					\node (b4) {$0$}; & \\
				};
				\draw[->] (t0) -- (t1);
				\draw[->] (t1) to node[above] {\tiny{$1 \mapsto (-\sigma,1)$}} (t2);
				\draw[->] (t2) to node[above] {\tiny{$\kappa$}} (t3);
				\draw[->] (t3) -- (t4);
				
				\draw[->] (b0) -- (b1);
				\draw[->] (b1) to node[above] {\tiny{$1 \mapsto (-\pi^*\alpha \cdot \sigma,1)$}} (b2);
				\draw[->] (b2)to node[above] {\tiny{$\kappa \begin{bmatrix} \pi^*\alpha & 0 \\ 0 & 1  \end{bmatrix}^{-1}$}} (b3);
				\draw[->] (b3) -- (b4);
				
				\draw[double equal sign distance] (b1) -- (t1);
				\draw[<-] (b2) to node[left] {\tiny{$\begin{bmatrix} \pi^*\alpha & 0 \\ 0 & 1  \end{bmatrix}$}} (t2);
				\draw[double equal sign distance] (b3) -- (t3);
			\end{tikzpicture}
		\end{center}
		in which the top row is \eqref{equation: initial exact sequence in bubbling up at nonsingular zero of a section} and the bottom row is the version of \eqref{equation: initial exact sequence in bubbling up at nonsingular zero of a section} for the new set of data, shows that the cokernels ${\sh K}$ for the two constructions are isomorphic, so we have a natural identification $C_{\zuzp} = C'_{\zuzp}$ compatible with the projections to $S$. By \eqref{equation: definition of line bundle in secod bubbling up}, we can also identify ${\sh L}'_{\zuzp} = {\sh L}_{\zuzp}$.  
		
		If we also assume furthermore for simplicity that the hypotheses from item \ref{item: item 1b in remark: scaling field in second bubbling up} (that $C_{\zuzp}$ and $C$ are integral and $f$ is birational) hold, then, by \ref{item: item 1b in remark: scaling field in second bubbling up}, $\sigma'_{\zuzp}(\zeta_{\zuzp}) = \pi^*\alpha \cdot \sigma(\zeta) = \pi^*\alpha \cdot \sigma_{\zuzp}(\zeta_{\zuzp})$, and hence $\sigma'_{\zuzp} = \pi^*\alpha \cdot \sigma_{\zuzp}$ everywhere. It is also clear that $x_{\zuzp}(\zeta) = x'_{\zuzp}(\zeta)$, so $x_{\zuzp} = x'_{\zuzp}$.
	\end{enumerate}   
\end{remark}
	
	Finally, let's check that Construction \ref{construction: second bubbling up construction} eliminated the degeneracy.
	
	\begin{lemma}\label{lemma: second bubbling up eliminated degeneracy}
		In Construction \ref{construction: second bubbling up construction}, $x_\zuzp^*\sigma_\zuzp^*$ is a nowhere vanishing global section of $x_\zuzp^*{\sh L}_\zuzp^*$.
	\end{lemma}
	
	\begin{proof}
		Let's first check that ${\sh L}_\zuzp = f^\dsd{\sh L}$ is trivial along $x_\zuzp$. We have
		\begin{equation}\label{equation: calculation of restriction along section of upper sigh in second bubbling up}
			\begin{aligned}
				x_\zuzp^*f^\dsd{\sh L} &= x_\zuzp^*{\sh O}_{{\mathbb P}{\sh K}(-x)}(1) \quad \text{by \eqref{equation: calculation of upper sigh in second bubbling up}} \\
				&= x^*{\sh O}_C(-x) \otimes x_\zuzp^*{\sh O}_{{\mathbb P}{\sh K}}(1) \quad \text{taking a step back in \eqref{equation: calculation of upper sigh in second bubbling up}} \\
				&= x^*{\sh O}_C(-x) \otimes x^*{\sh O}_C(x) = {\sh O}_S \quad \text{by the definition of $x_\zuzp$} \\
			\end{aligned}
		\end{equation}
		as desired. Consider the following diagram of ${\sh O}_S$-modules
		\begin{center}
			\begin{tikzpicture}
				\node (a) at (-0.5,0) {${\sh O}_S$};
				\node (b) at (2,0) {$x^*{\sh K}(-x)$};
				\node (c) at (5,0) {$x_\zuzp^*f^*f_*{\sh O}_{{\mathbb P}{\sh K}(-x)}(1)$};
				\node (d) at (2,-1.5) {${\sh O}_S$};
				\node (e) at (5,-1.5) {$x_\zuzp^*{\sh O}_{{\mathbb P}{\sh K}(-x)}(1)$};
				\node (f) at (8,-1.5) {$x_\zuzp^*f^\dsd{\sh L}$};
				
				\draw [->] (a) -- node [midway, below]{$\operatorname{id}_{{\sh O}_S}$} (d);
				\draw [double equal sign distance] (d) -- node[midway, above]{\eqref{equation: calculation of restriction along section of upper sigh in second bubbling up}} (e);
				\draw [double equal sign distance] (e) -- node[midway, above]{\eqref{equation: calculation of restriction along section of upper sigh in second bubbling up}} (f);
				\draw [double equal sign distance] (b) -- (c);
				\draw [->] (a) to (b);
				\draw [->>] (b) to (d);
				\draw [->>] (b) to (e);
				\draw [->] (c) to (e);
			\end{tikzpicture}
		\end{center}
		in which the map ${\sh O}_S \to x^*{\sh K}(-x)$ is the composition ${\sh O}_S \to x^*({\sh L}(-x) \oplus {\sh O}_C) \to x^*{\sh K}(-x)$, the map $x^*{\sh K}(-x) \to {\sh O}_S$ is a twist of the map $x^*{\sh K} \to x^*{\sh O}_C(x)$ used in Construction \ref{construction: second bubbling up construction} to define $x_\zuzp$, the map $x^*{\sh K}(-x) \to x_\zuzp^*{\sh O}_{{\mathbb P}{\sh K}(-x)}(1)$ also comes from the definition of $x_\zuzp$, and the isomorphism $x^*{\sh K}(-x) = x_\zuzp^*f^*f_*{\sh O}_{{\mathbb P}{\sh K}(-x)}(1)$ comes from $x_\zuzp^*f^*=x^*$ and
		\begin{equation}\label{equation: equation in proof of lemma: second bubbling up eliminated degeneracy}
			\begin{aligned}
				f_*{\sh O}_{{\mathbb P}{\sh K}(-x)}(1) &= f_* ( {\sh O}_{{\mathbb P}{\sh K}}(1) \otimes f^*{\sh O}_C(-x)) \\
				&= f_*{\sh O}_{{\mathbb P}{\sh K}}(1) \otimes {\sh O}_C(-x) \quad \text{by the projection formula} \\
				&= {\sh K}(-x) \quad \text{by part \ref{item 1b: proposition: all of second bubbling up elementary calculations} of Proposition \ref{proposition: all of second bubbling up elementary calculations}},
			\end{aligned}
		\end{equation}
		similarly to some steps in \eqref{equation: calculation of pushforward of upper sigh in second bubbling up}. It is not hard to check that the diagram is commutative. 
		
		Consider the image of $1 \in \Gamma(S,{\sh O}_S)$ in $\Gamma(S,x_\zuzp^*f^\dsd{\sh L})$. On one hand, if we take the upper route in the diagram, we see that this image is $x_\zuzp^*\sigma_\zuzp$ by revisiting Construction \ref{construction: second bubbling up construction}. Indeed, the composition of \eqref{equation: equation in proof of lemma: second bubbling up eliminated degeneracy} with the pushforward of \eqref{equation: calculation of upper sigh in second bubbling up} along $f$ is \eqref{equation: calculation of pushforward of upper sigh in second bubbling up}, hence the image of 1 in $\Gamma(S,x_\zuzp^*f^*f_*f^\dsd{\sh L})$ is $x_\zuzp^*f^*\sigma_\zuzp = x^*\sigma_\zuzp$, and the claim follows. On the other hand, the lower route is clearly an isomorphism, so the upper route is also an isomorphism. Hence $x_\zuzp^*\sigma_\zuzp$ is nowhere vanishing. 
	\end{proof}
	
	To summarize, we have the following.
	
	\begin{theorem}\label{theorem: summary of second bubbling up}
		Construction \ref{construction: second bubbling up construction} has the following features:
		\begin{enumerate}
			\item the output data is nondegenerate;
			\item if the input data is nondegenerate, then the output data is isomorphic to the input data;
			\item if $S = \spec {\mathbb K}$, and the input data is degenerate, then the output data is (isomorphic to) that described in Example \ref{example: geometric inflating at nonsingular zero of a section}; and
			\item it commutes with base change.
		\end{enumerate}
	\end{theorem}
	
	\begin{proof}
		The first claim is Lemma \ref{lemma: second bubbling up eliminated degeneracy}, the second claim is item \ref{item: what if no degeneracy in remark: Remarks on construction of second bubbling up} in Remark \ref{remark: Remarks on construction of second bubbling up}, the third claim is elementary and left to the reader, and the last claim is Proposition \ref{proposition: second bubbling up is functorial}.
	\end{proof}
	
	At this point, we consider Theorem \ref{theorem: main theorem abstract} proved.
	
Finally, we prove a technical proposition which contains most of the work needed later to prove that the bubbling down operation constructed in \S\ref{section: Bubbling down and representability} is inverse to bubbling up. 

\begin{proposition}\label{proposition: towards up of down = id}
	Let $(S,C,\pi,x,{\sh L},\sigma)$ and $(S,Y,\varpi,y,{\sh J},\rho)$ both satisfy the requirements of Situation \ref{situation: Bubbling up at nonsingular zero of a section}, and assume in addition that $y^*\rho$ is a nowhere vanishing section of $y^*{\sh J}$. Let $q:Y \to C$ with property R, such that:
	\begin{enumerate}
		\item\label{item: item 1 in proposition: towards up of down = id} $x = qy$, ${\sh J} = q^\dsd {\sh L}$ and $q_*{\sh J} = q_*q^\dsd{\sh L} \to {\sh L}$, cf. \eqref{equation: pushing forward logarithmic vector fields in general}, maps $\rho \mapsto \sigma$;
		\item\label{item: geometric fibers assumption in proposition: towards up of down = id} for any geometric point $\overline{s} \to S$, there exists an isomorphism $\alpha: Y_{\overline{s}} \to (C_{\zuzp})_{\overline{s}}$ over $\overline{s}$, such that $q_{\overline{s}} = f_{\overline{s}} \alpha$, 
		and $\alpha$ is compatible in the natural sense with all the rest of the data (restrictions of line bundles, sections, etc.).
	\end{enumerate}
	Then there exists a canonical isomorphism $q_\zuzp:Y \to C_\zuzp$ such that $q = fq_\zuzp$, $x_\zuzp = q_\zuzp y$, and $q_\zuzp^*\sigma_\zuzp \mapsto \rho$ under the isomorphism $q_\zuzp^*f^\dsd{\sh L} \simeq q^\dsd{\sh L}$.
\end{proposition}

\begin{proof}
	By item \ref{item 1: proposition: construction of logarithmic differential using duality} in Proposition \ref{proposition: construction of logarithmic differential using duality}, we have an isomorphism ${\sh O}_C(-x) \simeq q_*{\sh O}_Y(-y)$. Proposition \ref{proposition: construction of logarithmic differential using duality} gives a homomorphism $q^*({\sh L}^\vee(x)) \to (q^!{\sh L}^\vee)(y)$. By assumption \ref{item: geometric fibers assumption in proposition: towards up of down = id} and Example \ref{example: example of f upper shriek}, this homomorphism is an isomorphism on geometric fibers, and hence an isomorphism. Dualizing, we obtain
	\begin{equation}\label{equation: equation 0 in technical prop}
		q^*({\sh L}(-x)) \simeq (q^{\dsd}{\sh L})(-y).
	\end{equation}
	By \eqref{equation: equation 0 in technical prop} and Remark \ref{remark: pushforward of pullback of line bundles},
	\begin{equation}\label{equation: equation 1 in technical prop}
		q_*[(q^\dsd{\sh L}(-y))] = q_*q^* {\sh L}(-x) = {\sh L}(-x).
	\end{equation}
	We claim that the following diagram in which the top row is induced by $\sigma$ and the bottom one by $\rho$ is commutative.
	\begin{center}
		\begin{tikzpicture}
			\matrix [column sep  = 9mm, row sep = 5mm] {
				\node (nw) {${\sh O}_C(-x)$}; &
				\node (ne) {${\sh L}(-x)$};  \\
				\node (sw) {$q_*{\sh O}_Y(-y)$}; &
				\node (se) {$q_*(q^\dsd{\sh L})(-y)$}; \\
			};
			\draw[->] (nw) -- (ne);
			\draw[double equal sign distance] (ne) -- node [midway, right] {\eqref{equation: equation 1 in technical prop}} (se);
			\draw[->] (nw) --  node [midway, left] {$\simeq$} (sw);
			\draw[->] (sw) -- (se);
		\end{tikzpicture}
	\end{center}
	Indeed, the claim amounts to a statement of the form that two elements of the group
	$ \operatorname{Hom}({\sh O}_C(-x),  {\sh L}(-x) ) = \Gamma(C,{\sh L}) $
	coincide. However, it is clear that they coincide over $C \backslash x$ by assumption \ref{item: item 1 in proposition: towards up of down = id}, the fact that $q$ induces an isomorphism $q^{-1}(C\backslash x) \simeq C \backslash x$, and part \ref{item: isomorphic restriction in remark: Remarks on construction of second bubbling up} of Remark \ref{remark: Remarks on construction of second bubbling up}, so they must coincide on $C$ by Lemma \ref{lemma: injective restrictions on prestable curves}. 
	
	Since $(S,Y,\ldots)$ is nondegenerate in the sense of the current section, part \ref{item: what if no degeneracy in remark: Remarks on construction of second bubbling up} of Remark \ref{remark: Remarks on construction of second bubbling up} gives a short exact sequence $	0 \to {\sh O}_Y \to {\sh J} \oplus {\sh O}_Y(y) \to {\sh J}(y) \to 0$.
	If we twist this sequence by ${\sh O}_Y(-y)$ and push forward along $q$, we obtain a short exact sequence
	\begin{equation}\label{equation: equation 2 in technical prop}
		0 \to q_*{\sh O}_Y(-y) \to q_*[q^\dsd({\sh L}(-y))] \oplus q_*{\sh O}_Y \to q_*q^\dsd{\sh L} \to 0,
	\end{equation}
	since $R^1q_*{\sh O}_Y(-y) = 0$ by item \ref{item 1: proposition: construction of logarithmic differential using duality} in Proposition \ref{proposition: construction of logarithmic differential using duality} once more, since $qy=x$ is a smooth section. We have a (only solid arrow, for now) diagram with exact rows
	\begin{equation}\label{equation: 10-term diagram}
		\begin{tikzpicture}
			\def\aa{0.8};
			\node (a0) at (-1.9,1.6-\aa) {$0$};
			\node (a1) at (-0.3,1.6-\aa) {${\sh O}_C(-x)$};
			\node (a2) at (2.8,1.6-\aa) {${\sh L}(-x) \oplus {\sh O}_C$};
			\node (a3) at (6.2,1.6-\aa) {${\sh K}(-x)$};
			\node (a4) at (7.6,1.6-\aa) {$0$};
			\node (b0) at (-1.9,0-\aa) {$0$};
			\node (b1) at (-0.3,0-\aa) {$q_*{\sh O}_Y(-y)$};
			\node (b2) at (2.8,0-\aa) {$q_*(q^\dsd{\sh L})(-y) \oplus q_*{\sh O}_Y$}; 
			\node (b3) at (6.2,0-\aa) {$q_*q^\dsd{\sh L}$}; 
			\node (b4) at (7.6,0-\aa) {$0$};
			\draw[->] (b0) to (b1);
			\draw[->] (b1) to (b2);
			\draw[->] (b2) -- node [midway, above] {$\left[\begin{smallmatrix} -1 & \rho \end{smallmatrix} \right]$} (b3);
			\draw[->] (b3) to (b4);
			\draw[->] (a0) to (a1);
			\draw[->] (a1) to (a2);
			\draw[->] (a2) to (a3);
			\draw[->] (a3) to (a4);
			\draw[double equal sign distance] (a1) -- (b1);
			\draw[->, dashed] (a3) -- node [midway, right] {$\simeq$} (b3);
			\draw[double equal sign distance] (a2) -- node [midway, left] {$\left[\begin{smallmatrix} \eqref{equation: equation 1 in technical prop} & 0 \\ 0 & q^\# \end{smallmatrix} \right]$} (b2);	
		\end{tikzpicture}
	\end{equation}
	The top row is \eqref{equation: initial exact sequence in bubbling up at nonsingular zero of a section} twisted by ${\sh O}_C(-x)$, the bottom row is \eqref{equation: equation 2 in technical prop}. This diagram is commutative, by the commutativity of the first diagram in this proof, and part \ref{item 1: proposition: construction of logarithmic differential using duality} of Proposition \ref{proposition: construction of logarithmic differential using duality}. A trivial diagram chase shows that there exists a (unique) homomorphism
	${\sh K}(-x) \to q_*q^\dsd{\sh L}$
	that makes the diagram above commute if assigned as the dashed arrow and moreover, it is an isomorphism ${\sh K}(-x) \simeq q_*q^\dsd{\sh L}$.  We claim that its adjoint 
	\begin{equation}\label{equation: important adjoint in technical prop}
		q^*{\sh K}(-x) \to q^\dsd{\sh L}
	\end{equation} 
	is surjective. Consider the `adjoint' diagram of \eqref{equation: 10-term diagram}.
	\begin{equation}\label{equation: 9-term diagram}
		\begin{tikzpicture}[scale=1.1]
			\node (a1) at (0,1) {$q^*{\sh O}_C(-x)$};
			\node (a2) at (2.5,1) {$q^*{\sh L}(-x) \oplus {\sh O}_Y$};
			\node (a3) at (5,1) {$q^*{\sh K}(-x)$};
			\node (a4) at (6.4,1) {$0$};
			\node (b0) at (-1.4,0) {$0$};
			\node (b1) at (0,0) {${\sh O}_Y(-y)$};
			\node (b2) at (2.5,0) {$(q^\dsd{\sh L})(-y) \oplus {\sh O}_Y$}; 
			\node (b3) at (5,0) {$q^\dsd{\sh L}$}; 
			\node (b4) at (6.4,0) {$0$};
			\draw[->] (b0) to (b1);
			\draw[->] (b1) to (b2);
			\draw[->] (b2) to (b3);
			\draw[->] (b3) to (b4);
			\draw[->] (a1) to (a2);
			\draw[->] (a2) to (a3);
			\draw[->] (a3) to (a4);
			\draw[->] (a1) to (b1);
			\draw[->] (a3) -- (b3);
			\draw[double equal sign distance] (a2) to (b2);	
		\end{tikzpicture}
	\end{equation}
	The central vertical map is still an isomorphism by \eqref{equation: equation 0 in technical prop}. The map $q^*{\sh K}(-x) \to q^\dsd{\sh L}$ we are interested in is the right vertical map in the diagram, and it is surjective by the snake lemma. Twisting, we obtain a surjection $q^*{\sh K} \to  q^*{\sh O}_C(x) \otimes q^\dsd{\sh L}$. By Remark \ref{remark: universal property of projectivizations}, this induces an $S$-morphism $q_\zuzp:Y \to C_\zuzp$ such that $fq_\zuzp = q$. 
	
	Let's first check that $q_\zuzp y = x_\zuzp$. In light of Remark \ref{remark: universal property of projectivizations}, $q_\zuzp y$ corresponds to the surjection $x^*{\sh K} \to x^*{\sh O}_C(x) \otimes y^*q^\dsd{\sh L}$. Earlier, $x_\zuzp$ was defined to correspond to the natural map $x^*{\sh K} \to x^*{\sh O}_C(x)$. Since $y^*\rho$ is nowhere vanishing, the homomorphism ${\sh O}_S \to y^*{\sh J}$ is an isomorphism, and hence so is its twist $x^*{\sh O}_C(x) \to x^*{\sh O}_C(x) \otimes y^*q^\dsd{\sh L}$. It remains to check that the triangle with the three homomorphisms among $x^*{\sh K}$, $x^*{\sh O}_C(x)$, $x^*{\sh O}_C(x) \otimes y^*q^\dsd{\sh L}$ is commutative. To this end, consider the diagram
	\begin{center}
		\begin{tikzpicture}[scale = 1.1]
			\node (sw) at (0,0) {$y^*[(q^\dsd{\sh L})(-y)] \oplus {\sh O}_S$};
			\node (se) at (5,0) {$y^*(q^\dsd{\sh L})$};
			\node (nw) at (0,1) {$x^*{\sh L}(-x) \oplus {\sh O}_S$}; 
			\node (n) at (3,1) {$x^*{\sh K}(-x)$}; 
			\node (ne) at (5,2) {${\sh O}_S$};
			\draw [->] (sw) -- node [midway, above] {$ \begin{bmatrix} 0 & y^*\rho \end{bmatrix} $} (se);
			\draw [->] (sw) -- node [midway, left] {$\simeq$} (nw);
			\draw [->>] (nw) -- (n);
			\draw [->>] (n) -- (ne);
			\draw [->] (ne) -- node [midway, right] {$\simeq$} (se);
			\draw [->>] (n) -- (se); 
			\draw [->] (nw) to [out = 30, in = 180] (ne);
		\end{tikzpicture}
	\end{center}
	in which all arrows in the bottom (trapezoidal) face are obtained as pullbacks along $y$ of homomorphisms in \eqref{equation: 9-term diagram}, the $x^*{\sh K}(-x) \to {\sh O}_S$ arrow is a twist of the map $x^*{\sh K} \to x^*{\sh O}_C(x)$ discussed earlier, the curved arrow is projection to the second factor, and the right vertical arrow is the map discussed earlier which maps $1 \mapsto y^*\rho$. We argue commutativity as follows. First, the top (curved, triangular) face commutes, essentially by construction. Second, the bottom (trapezoidal) face commutes, thanks to \eqref{equation: 9-term diagram}. Third, the convex hull commutes. Indeed, we may check this separately on the submodule $y^*[(q^\dsd{\sh L})(-y)] \oplus 0$ and the section $(0,1)$ of $y^*[(q^\dsd{\sh L})(-y)] \oplus {\sh O}_S$; via either route to $y^*(q^\dsd {\sh L})$, the former is mapped to $0$, and the latter to $y^*\rho$.
	Since $x^*{\sh L}(-x) \oplus {\sh O}_S \to x^*{\sh K}(-x)$ is surjective, the three points above imply that the right triangular face commutes as well. After twisting by $x^*{\sh O}_C(x)$, this is precisely what had to be checked.
	
	Now it is easy to check that the morphism $q_\zuzp:Y \to C_\zuzp$ over $S$ constructed earlier is an isomorphism on geometric fibers, and hence an isomorphism.
	
	It remains to check that $q_\zuzp^*\sigma_\zuzp \mapsto \rho$ under the isomorphism $q_\zuzp^*f^\dsd{\sh L} \simeq q^\dsd{\sh L}$. This isomorphism can be obtained, for instance, from $q_\zuzp^*f^\dsd = q_\zuzp^\dsd f^\dsd = q^\dsd$. Moreover, we have maps $q^*{\sh K}(x) = q_\zuzp^*f^*{\sh K}(-x) \to q_\zuzp^* f^\dsd {\sh L}$ coming from \eqref{equation: calculation of pushforward of upper sigh in second bubbling up} and the adjoint property, and $q^*{\sh K}(-x) \to q^\dsd {\sh L}$, cf. \eqref{equation: important adjoint in technical prop}. We claim that these $3$ homomorphisms fit it a commutative triangle. We will give an indirect argument. The starting point is part \ref{item: description of adjoint in remark: Remarks on construction of second bubbling up} in Remark \ref{remark: Remarks on construction of second bubbling up}. Pulling back along $q_\zuzp$, we see that 
	\[ q^*{\sh K}(-x) = q_\zuzp^*f^*{\sh K}(-x) \to q_\zuzp^*f^\dsd{\sh L} \] 
	corresponds to $q_\zuzp:Y \to C_\zuzp = {\mathbb P}{\sh K}(-x)$ in the sense of Remark \ref{remark: universal property of projectivizations}. On the other hand, by the same remark, $q_\zuzp:Y \to {\mathbb P}{\sh K}(-x)$ also corresponds to \eqref{equation: important adjoint in technical prop}, since \eqref{equation: important adjoint in technical prop} is a twist of the morphism defining $q_\zuzp$ as a morphism to ${\mathbb P}{\sh K}$. It follows that there exists an isomorphism $q_\zuzp^*f^\dsd{\sh L} \simeq q^\dsd{\sh L}$, such that the triangle diagram with this isomorphism and the maps $q^*{\sh K}(-x) \to q_\zuzp^*f^\dsd {\sh L}$ and $q^*{\sh K}(-x) \to q^\dsd {\sh L}$ commutes. Although there is no obvious formal reason why this isomorphism $q_\zuzp^*f^\dsd{\sh L} \simeq q^\dsd{\sh L}$ coincides with the earlier isomorphism $q_\zuzp^*f^\dsd{\sh L} \simeq q^\dsd{\sh L}$, this is true, and the reason is the following. Since all geometry becomes trivial over $C \backslash x$, it is straightforward to check that the two isomorphisms above coincide over $q^{-1}(C\backslash x) \simeq C\backslash x$. However, ${\sh A}ut(q^\dsd{\sh L}) \simeq {\sh O}_Y^\times$ and $\Gamma(Y,{\sh O}_Y^\times) \to \Gamma(Y \backslash y, {\sh O}_Y^\times)$ is injective since $\Gamma(Y,{\sh O}_Y) \to \Gamma(Y \backslash y, {\sh O}_Y)$ is injective by Lemma \ref{lemma: injective restrictions on prestable curves}, so the equality of the isomorphisms a fortiori extends to all of $Y$. This concludes the proof of the desired commutativity. Let $\tau \in \Gamma(Y,q^*{\sh K}(-x))$ be the image of $1 \in \Gamma(Y,{\sh O}_Y)$ under the composition 
	\[ {\sh O}_Y \to q^*({\sh L}(-x) \oplus {\sh O}_C) \to q^*{\sh K}(-x). \]
	Then the image of $\tau$ in $q_\zuzp^*f^\dsd {\sh L}$ is $q_\zuzp^*\sigma_\zuzp^*$ by Construction \ref{construction: second bubbling up construction}, while the image of $\tau$ in $q^\dsd{\sh L}$ is $\rho$ by the commutativity of the right square in \eqref{equation: 9-term diagram}. Then the commutativity of the triangle above concludes the proof.
\end{proof}

	\section{Bubbling down}\label{section: Bubbling down and representability}
	
	Recall the categories of curves ${\cat C}^{c,+}_{g,m,n}$ and ${\cat C}^c_{g,m,n}$ from Definition \ref{definition: categories of curves}.

	\begin{proposition}\label{proposition: statement that the bubbling up functors were constructed}
		Constructions \ref{construction: first bubbling up construction} and \ref{construction: second bubbling up construction} induce functors
		\begin{equation}\label{equation: sequence of raw up functors}
			{\cat C}^{1,+}_{g,m,n} \xrightarrow{\kukp} {\cat C}^{2,+}_{g,m,n} \xrightarrow{\zuzp} {\cat C}^{3,+}_{g,m,n},
		\end{equation}
		which, in turn, restrict to functors
		\begin{equation}\label{equation: sequence of up functors}
			{\cat C}^1_{g,m,n} \xrightarrow{\kukp} {\cat C}^2_{g,m,n} \xrightarrow{\zuzp} {\cat C}^3_{g,m,n}.
		\end{equation}
	\end{proposition}
	
	\begin{proof} We take ${\sh L} = \omega_{C/S}^\vee$ in Construction \ref{construction: first bubbling up construction} (resp. Construction \ref{construction: second bubbling up construction}); then $(\omega_{C/S}^\vee)_\kukp = \omega_{C_\kukp/S_\kukp}^\vee$ (resp. $(\omega_{C/S}^\vee)_\zuzp = \omega_{C_\zuzp/S_\zuzp}^\vee$), by \eqref{equation: definition of line bundle in first bubbling up}, \eqref{equation: definition of line bundle in secod bubbling up}, Definition \ref{definition: dual shriek dual}, and Lemma \ref{lemma: upper shriek of invertible is invertible}. To be very accurate, in Construction \ref{construction: first bubbling up construction}, $\sigma$ is the restriction $\omega_{C/S} \to {\sh O}_C(-w_1-\cdots-w_m)$ of $\phi$ rather than $\phi$ itself. In Construction \ref{construction: second bubbling up construction}, $\sigma$ is simply $\phi$. The apparent problem is that $\overline{x}$ is not specified in Construction \ref{construction: first bubbling up construction}, and neither $\overline{x}$ nor $\overline{w}$ is specified in Construction \ref{construction: second bubbling up construction}. However, the conditions in Definitions \ref{definition: categories of curves} and \ref{definition: more categories of curves} force these sections to be contained in the maximal open subset $U$ such that $f^{-1}(U) \simeq U$ (this is the open in Lemma \ref{lemma: dense f-trivial}), and then lifting them is a trivial matter. Propositions \ref{proposition: first bubbling up is functorial} and \ref{proposition: second bubbling up is functorial} state that $\kukp$ and $\zuzp$ commute with base change. Finally, the functors in \eqref{equation: sequence of raw up functors} restrict to the functors in \eqref{equation: sequence of up functors}. Indeed, condition \ref{item: item 3 in definition: more categories of curves} in Definition \ref{definition: more categories of curves} may be checked on geometric fibers, and is then elementary.
	\end{proof}
	
	The two operations above suffice to construct the degeneration in Theorem \ref{theorem: isotrivial degeneration theorem}. However, they do not suffice to interpret the central fiber in any meaningful way, and, in particular, to prove Theorem \ref{theorem: Pn bar as moduli space}. 
	
	Therefore, in this section, we construct inverse operations to the two operations above, and in particular prove Theorem \ref{theorem: universal curve theorem simplified statement}. Artificial as it may be, it is often feasible to treat the two inverse operations simultaneously.
	
	We start with some geometric preliminaries. Recall that a coherent ${\sh O}_X$-module ${\sh F}$ is \emph{normally generated} if the canonical map $\Gamma(X,{\sh F})^{\otimes k} \to \Gamma(X,{\sh F}^{\otimes k})$ is surjective for all $k \geq 1$ (e.g. \cite[Definition 1.7]{[Kn83]}).
	
	\begin{proposition}\label{proposition: normal generation in geometric case}
		Let $g,m,n,c \geq 0$ be integers such that $2g+2n+m+c \geq 4$ and $c \in \{1,2\}$. Let $(S =\spec {\mathbb K},C,\ldots)$ be an object of ${\cat C}^c_{g,m,n}({\mathbb K})$ with notation as in Definition \ref{definition: more categories of curves}, and
		$$ {\sh L} = \omega_{C} (w_1+\cdots+w_m + 2x_1+\cdots+2x_n + (c-1)x). $$  
		Then:
		\begin{enumerate}
			\item\label{item: item 1 in proposition: normal generation in geometric case} ${\sh L}^{\otimes k}$ is generated by global sections if $k \geq 2$;
			\item\label{item: item 2 in proposition: normal generation in geometric case} $\operatorname{H}^1(C,{\sh L}^{\otimes k}) = 0$ if $k \geq 2$;
			\item\label{item: item 3 in proposition: normal generation in geometric case} ${\sh L}^{\otimes k}$ is normally generated if $k \geq 3$.
		\end{enumerate}
	\end{proposition}

	\begin{proof}
		First, let us prove the proposition in the case when all components of $C$ on which $\phi$ is not identically zero are rational tails. Let $T_1,\ldots,T_a$ ($a \geq 0$) be these components. Let $Y$ be the nodal curve obtained by contracting $T_1,\ldots,T_a$; then we have both a contraction $C \to Y$, and an immersion $j:Y \hookrightarrow C$. Note that 
		\begin{equation}\label{equation: position of w in proposition: normal generation in geometric case}
			w_1,\ldots,w_m \notin T_1 \cup \cdots \cup T_a
		\end{equation}
		unless $a=1$ and $C=T_1$ because $\phi$ must vanish at the point where $T_\alpha$ is attached to the rest of the curve, and hence cannot vanish anywhere else on $T_\alpha$, as $\deg \omega_C^\vee|_{T_\alpha} = 1$. We assume \eqref{equation: position of w in proposition: normal generation in geometric case} holds since the alternative is trivial. We have 
		$ x_1,\ldots,x_n \in T_1 \cup \cdots \cup T_a $
		since $\phi$ cannot vanish at $x_1,\ldots,x_n$. Let $y_1,\ldots,y_r$ be the following points on $Y$:
		\begin{itemize}
			\item $w_1,\ldots,w_m$ (or $j^{-1}(w_1),\ldots,j^{-1}(w_m)$ to be exceedingly accurate);
			\item the images of the tails $T_1,\ldots,T_a$ under the contraction $C \to Y$; and
			\item $x$, but only in case $c=2$ and $x \notin T_1 \cup \cdots \cup T_a$.
		\end{itemize} 
		Thus, $r=m+a+1$ if $c=2$ and $x \notin T_1 \cup \cdots \cup T_a$, and $r=m+a$ otherwise. Note that $y_1,\ldots,y_r$ are smooth points of $Y$, and that 
		$ \omega_Y(y_1 + \cdots + y_r) \simeq j^*{\sh L} $. 
		Then $(Y,y_1,\ldots,y_r) \in \overline{\modu M}_{g,r}({\mathbb K})$ since ${\sh L}$ is ample by condition \ref{item: item 3 in definition: more categories of curves} in Definition \ref{definition: more categories of curves} and the fact that ample restricts to ample. The analogues of items \ref{item: item 1 in proposition: normal generation in geometric case} -- \ref{item: item 3 in proposition: normal generation in geometric case} above for $Y,y_1,\ldots,y_r$ hold by \cite[Theorem 1.8]{[Kn83]}. All items follow inductively, by reattaching the tails $T_1,\ldots,T_a$ back to $Y$ one by one. The inductive step is essentially the argument in the last paragraph in the proof of \cite[Theorem 1.8]{[Kn83]}.
		
		It remains to consider the case when $\phi$ is not identically $0$ on some component $\Sigma$ of $C$ which is not a rational tail. Such cases occur extremely rarely. Note that, if $\Sigma_1 \neq \Sigma_2$ are irreducible components of $C$ such that $\Sigma_1 \cap \Sigma_2 \neq \emptyset$, $\phi|_{\Sigma_1} \equiv 0$, and $\phi|_{\Sigma_2} \not\equiv 0$, then $\Sigma_2$ is a rational tail intersecting $\Sigma_1$. (By the elementary description of fields in \S\ref{section: introduction}, $\phi|_{\Sigma_2}$ regarded as a global section of $T_{\Sigma_2}$ vanishes \emph{doubly} (by the assumption $\phi|_{\Sigma_1} \equiv 0$) at any point of $\Sigma_1 \cap \Sigma_2$ and also vanishes at all other nodes of $C$ on $\Sigma_2$, so the claim follows by simple degree considerations.) However, this scenario with $\Sigma_2 = \Sigma$ is ruled out, so the only possible conclusion is that $\phi$ doesn't vanish identically on any irreducible component of $C$. The only nodal curves $C$ which admit (logarithmic) vector fields which don't vanish identically on any irreducible component are smooth genus $1$ curves, rational chains, and rational cycles, so it remains to check all items in all these cases. If $C$ is irreducible, then all claims amount to well-known facts about elliptic curves. In the remaining cases (rational cycles or chains), it straightforward to check that there exist smooth, \emph{distinct} points $p_1,\ldots,p_N \in C$ such that ${\sh O}_C(w_1+\cdots+w_m + 2x_1+\cdots+2x_n + (c-1)x) \simeq {\sh O}_C(p_1+\cdots+p_N) $, or equivalently, ${\sh L} \simeq \omega_C(p_1+\cdots+p_N)$. Since ${\sh L}$ is ample, $(C,p_1,\ldots,p_N) \in \overline{{\modu M}}_{g,N}({\mathbb K})$, and the claims follow from \cite[Theorem 1.8]{[Kn83]}.
	\end{proof}
	
	\begin{definition}\label{definition: bubbling down}
		Let $g,m,n,c \geq 0$ be integers such that $2g+2n+m+c \geq 5$ and $c \in \{2,3\}$. Let ${\obj o}=(S,C,\pi,\overline{w},\overline{x},x,\phi)$ be an object of ${\cat C}_{g,m,n}^c$. A \emph{bubbling down} of ${\obj o}$ is an object ${\obj o}_\down = (S,C_\down,\pi_\down,\overline{w}_\down,\overline{x}_\down,x_\down,\phi_\down)$ of ${\cat C}^{c-1}_{g,m,n}$, together with a morphism $h:C \to C_\down$ with property R, such that:
		\begin{enumerate}
			\item\label{item: item 1 in definition: bubbling down} $\pi_\down h = \pi$, $hx_j=x_{j,\down}$ for all $j$, $hx=x_\down$, $hw_i=w_{i,\down}$ for all $i$; and
			\item\label{item: item 2 in definition: bubbling down} $\phi_\down$ is the image of $\phi$ under  
			$ h_*\omega^\vee_{C/S} \to \omega_{C_\down/S}^\vee$,
			which is a special case of \eqref{equation: pushing forward logarithmic vector fields in general}, in light of Definition \ref{definition: dual shriek dual} and Lemma \ref{lemma: upper shriek of invertible is invertible}.
		\end{enumerate} 
	\end{definition}
	
	\begin{remark}\label{remark: alternative statement of condition 2}
		If $\phi$ and $\phi_\down$ are regarded as sections of $\omega^\vee_{C/S}(- w_1 - \cdots - w_n)$ and $\omega_{C_\down/S}^\vee(- w_{1,\down} - \cdots - w_{n,\down})$, condition \ref{item: item 2 in definition: bubbling down} in Definition \ref{definition: bubbling down} is equivalent to asking that the homomorphism
		$ h_*\omega^\vee_{C/S}(- \sum_{i=1}^n w_i) \to \omega_{C_\down/S}^\vee(- \sum_{i=1}^n w_{i,\down})$,
		which is also an example of \eqref{equation: pushing forward logarithmic vector fields in general}, maps $\phi$ to $\phi_\down$.
	\end{remark}
	
	A base change of a bubbling down is a bubbling down by Proposition \ref{proposition: pushing forward vector fields in general}.
	
	\begin{proposition}\label{proposition: existence and uniqueness of bubbling down}
		Let $g,m,n,c$ as in Definition \ref{definition: bubbling down}. The bubbling down of any object of ${\cat C}_{g,m,n}^c$ exists and is unique up to unique isomorphism in ${\cat C}_{g,m,n}^{c-1}$ compatible with the maps with property R.
	\end{proposition}
	
	\begin{proof} To prove existence, we follow the proof of \cite[Proposition 3.10]{[BM96]} extremely closely (which in turn relates to \cite[\S1]{[Kn83]}). Let
		\begin{equation*}
			{\sh L} = \omega_{C/S}( w_1+\cdots+w_m + 2x_1 + \cdots + 2x_n + (c-2)x ),
		\end{equation*}
		and define
		\begin{equation}\label{equation: construction of bubbling down as Proj}
			C_\down = \mathrm{Proj}_S\mathop{}{\sh S} \quad \text{where} \quad {\sh S} = \bigoplus_{k \geq 0} {\sh S}_k = \bigoplus_{k \geq 0} \pi_*{\sh L}^{\otimes k},
		\end{equation}
		with structure map $\pi_\down:C_\down \to S$.
		
		First, we analyse in depth the case $S = \spec {\mathbb K}$. An irreducible component $\Sigma$ of $C$ is a \emph{component to be contracted} if $\Sigma \simeq {\mathbb P}^1$, $x \in \Sigma$, $x_1,\ldots,x_n \notin \Sigma$, and the following conditions dependent on $c$ are also satisfied:
		\begin{itemize}
			\item if $c=3$, then $\Sigma$ contains precisely one node of $C$, and $w_1,\ldots,w_m \notin \Sigma$;
			\item if $c=2$, then \emph{either}
			\begin{enumerate}
				\item $\Sigma$ contains precisely one node of $C$, and one of $w_1,\ldots,w_m$; \emph{or}
				\item $\Sigma$ contains precisely two nodes of $C$, and $w_1,\ldots,w_m \notin \Sigma$.
			\end{enumerate}
		\end{itemize}
		Clearly, there is at most one component to be contracted. We construct $\chi:C \to Y$ which contracts the component to be contracted (if there is one) and nothing else, such that $Y$ is nodal. The construction is by elementary means, similar to \cite{[BM96]}. Then, $Y$ with the markings $\chi(\overline{w})$, $\chi(\overline{x})$, $\chi(x)$, and the image of $\phi$ under the homomorphism $\chi_* \omega_C^\vee \to \omega_Y^\vee$ (an example of \eqref{equation: pushing forward logarithmic vector fields in general}) is a bubbling down of the given data (note also that $\chi$ has property R). Let
		\[ {\sh J} = \omega_{Y}(  \chi(w_1) + \cdots + \chi(w_m) + 2 \chi(x_1)+\cdots+ 2 \chi(x_n) + (c-2)\chi(x) ). \]
		
		\begin{claim}\label{claim: various elementary things about C to Y}
			For all $k \geq 0$, $\chi^*{\sh J}^{\otimes {k}} \simeq {\sh L}^{\otimes k}$, ${\sh J}^{\otimes {k}} \simeq \chi_*{\sh L}^{\otimes {k}}$, and $R^1\chi_*{\sh L}^{\otimes k} = 0$.
		\end{claim}
		
		\begin{proof}
			The proof is straightforward. We will only mention that the second claim follows from the first claim and Remark \ref{remark: pushforward of pullback of line bundles}.
		\end{proof}
		
		\begin{claim}\label{claim: claim inside of proposition: existence and uniqueness of bubbling down}
			The following hold:
			\begin{enumerate}
				\item\label{item: item 1 in claim: claim inside of proposition: existence and uniqueness of bubbling down} ${\sh L}^{\otimes k}$ is generated by global sections if $k \geq 2$;
				\item\label{item: item 2 in claim: claim inside of proposition: existence and uniqueness of bubbling down} $\operatorname{H}^1(C,{\sh L}^{\otimes k}) = 0$ if $k \geq 2$;
				\item\label{item: item 3 in claim: claim inside of proposition: existence and uniqueness of bubbling down} ${\sh L}^{\otimes k}$ is normally generated if $k \geq 3$.
			\end{enumerate}
		\end{claim}
		
		\begin{proof}
			Apply Proposition \ref{proposition: normal generation in geometric case} with $c-1$ in the role of $c$, $Y$ in the role of $C$, and ${\sh J}$ in the role of ${\sh L}$ from the statement of the proposition. The claim follows from this and Claim \ref{claim: various elementary things about C to Y}. (Please compare also with Lemma 1.6, Theorem 1.8, and Corollary 1.10 in \cite{[Kn83]} on one hand, and with Proposition 3.9 and Lemmas 3.11 and 3.12 in \cite{[BM96]}, on the other hand.)
		\end{proof}
		
		\begin{claim}\label{claim: recovering bubbling down in elementary case}
			With $C_\down$ defined as in \eqref{equation: construction of bubbling down as Proj}, $Y \simeq C_\down$. Moreover, in the language of \cite[(3.7.1)]{[EGAII]}, the composition $C \xrightarrow{\chi} Y \simeq C_{\down}$ coincides with the morphism associated to ${\sh L}$ and ${\sh O}_C \otimes \bigoplus_{k \geq 0} \Gamma(C,{\sh L}^{\otimes k}) \to \bigoplus_{k \geq 0} {\sh L}^{\otimes k}$. 
		\end{claim}
		
		\begin{proof}
			By \eqref{equation: construction of bubbling down as Proj}, Claim \ref{claim: various elementary things about C to Y}, \cite[Proposition (4.6.3)]{[EGAII]}, and the ampleness of ${\sh J}$,
			\[
			C_\down = \operatorname{Proj}\mathop{} \bigoplus_{k \geq 0} \Gamma(C,{\sh L}^{\otimes k}) = \operatorname{Proj}\mathop{} \bigoplus_{k \geq 0} \Gamma(Y,\chi_*{\sh L}^{\otimes k}) = \operatorname{Proj}\mathop{} \bigoplus_{k \geq 0} \Gamma(Y,{\sh J}^{\otimes k}) = Y.
			\]
			The second claim is trivial if $\chi$ is an isomorphism, so let us assume this is not the case, and let $\Sigma \subset C$ be the component to be contracted. By Claim \ref{claim: various elementary things about C to Y}, we have maps ${\sh O}_C \otimes \Gamma(Y,{\sh J}^{\otimes k}) = {\sh O}_C \otimes \Gamma(C,{\sh L}^{\otimes k}) \to {\sh L}^{\otimes k} $ for all $k \geq 0$. Then, by \cite[(3.7.1)]{[EGAII]}, the claim amounts to the statement that the composition
			\begin{equation}\label{equation: equation 1 in proof of claim: recovering bubbling down in elementary case} 
				C= \mathrm{Proj}_C \bigoplus_{k \geq 0} {\sh L}^{\otimes k} \xrightarrow{j} \mathrm{Proj}_C \bigoplus_{k \geq 0} {\sh O}_C \otimes \Gamma(Y,{\sh J}^{\otimes k}) = C \times Y \to Y 
			\end{equation}
			equals $\chi$. Note that $j$ is a closed immersion (in particular, defined everywhere) by \cite[\href{https://stacks.math.columbia.edu/tag/07ZK}{Tag 07ZK}]{[stacks]}, the general fact that $\mathrm{Proj}_X \bigoplus_{k \geq 0} {\sh A}_k \cong \mathrm{Proj}_X \bigoplus_{k \geq 0} {\sh A}_{k_0k}$ ($X$ is a scheme, $\bigoplus_{k \geq 0} {\sh A}_k$ is a graded ${\sh O}_X$-algebra, and $k_0 > 0$ is an integer), the graded ring isomorphism $\bigoplus_{k \geq 0} \Gamma(C,{\sh L}^{\otimes 3k}) \cong \bigoplus_{k \geq 0} \Gamma(Y,{\sh J}^{\otimes 3k})$, and item \ref{item: item 3 in claim: claim inside of proposition: existence and uniqueness of bubbling down} in Claim \ref{claim: claim inside of proposition: existence and uniqueness of bubbling down}. The fact that \eqref{equation: equation 1 in proof of claim: recovering bubbling down in elementary case} coincides with $\chi$ follows from the easily checked fact that the restriction of \eqref{equation: equation 1 in proof of claim: recovering bubbling down in elementary case} to $C \backslash \Sigma$ coincides with $C \backslash \Sigma \simeq Y \backslash \chi(\Sigma) \hookrightarrow Y$. Indeed, the explicit description of $Y$ and $\chi$ clarifies that this a priori weaker statement suffices.
		\end{proof}
		
		Let us now consider the general case, when the base $S$ is arbitrary. 
		
	\begin{claim}\label{claim: various things about bubbling down over any base}
		The following hold:
		\begin{enumerate}
			\item\label{item: item 1 inside proof of proposition: existence and uniqueness of bubbling down} the formation of $C_\down$ commutes with base change;
			\item\label{item: item 2 inside proof of proposition: existence and uniqueness of bubbling down} $\pi_\down$ is flat and projective;
			\item\label{item: item 3 inside proof of proposition: existence and uniqueness of bubbling down} the morphism associated to ${\sh L}$ and $\pi^*{\sh S} \to \bigoplus_{k \geq 0} {\sh L}^{\otimes k}$ (in the language of \cite[(3.7.1)]{[EGAII]}) is defined everywhere; and
			\item\label{item: item 4 inside proof of proposition: existence and uniqueness of bubbling down} the formation of the morphism $h:C \to C_\down$ from item \ref{item: item 3 inside proof of proposition: existence and uniqueness of bubbling down} commutes with base change.
		\end{enumerate}
	\end{claim}
	
	\begin{proof}
		The proof of the first three items is entirely analogous to the proofs of Claims 1--3 inside the proof of \cite[Proposition 3.10]{[BM96]}, but we will briefly review the arguments for the reader's convenience. (Claim \ref{claim: claim inside of proposition: existence and uniqueness of bubbling down} ensures that the required hypotheses are satisfied, so the arguments in \cite{[BM96]} can be repeated almost verbatim -- indeed, the analogue of item (4) in Lemma 3.12 in loc. cit. is not used in Claims 1--3.) As in \cite{[BM96]}, the idea is to take $d \gg 0$ (in fact, $d=3$ suffices), and to use the fact that 
		\begin{equation}
			C_\down = \mathrm{Proj}_S {\sh S}^{(d)}, \quad \text{where} \quad {\sh S}^{(d)} = \bigoplus_{d|k} {\sh S}_k.
		\end{equation}
		First, note that whenever $d|k$ (or more generally, either $k=0$ or $k \gg 0$), $ {\sh S}_k$ is locally free and its formation commutes with base change. Indeed, this is trivial for $k=0$ since $\pi_*{\sh O}_C = {\sh O}_S$, and follows from Claim \ref{claim: claim inside of proposition: existence and uniqueness of bubbling down} and the `cohomology and base change' theorem otherwise. Hence, the formation of $C_\down$ commutes with base change, and, as in \cite{[BM96]}, flatness of $\pi_\down$ follows from local freeness. 
		
		Moreover, if $d|k$, the homomorphism
		\begin{equation}\label{equation: map that will turn out surjective in proof of claim: various things about bubbling down over any base} 
			\pi^* {\sh S}_k = \pi^*\pi_* {\sh L}^{\otimes k} \to {\sh L}^{\otimes k}
		\end{equation}
		is surjective. Indeed, we have established that $\pi_* {\sh L}^{\otimes k}$ commutes with base change, so the surjectivity may be checked on (geometric) fibers, when it follows from item \ref{item: item 1 in claim: claim inside of proposition: existence and uniqueness of bubbling down} in Claim \ref{claim: claim inside of proposition: existence and uniqueness of bubbling down}. 
		
		We claim that ${\sh S}^{(d)}$ is generated by ${\sh S}_1^{(d)}= {\sh S}_d$ over ${\sh S}^{(d)}_0 = {\sh O}_S$. Since all summands are locally free and commute with base change, in view of Nakayama's lemma, the claim can be checked on fibers and hence on geometric fibers, when it follows from item \ref{item: item 3 in claim: claim inside of proposition: existence and uniqueness of bubbling down} in Claim \ref{claim: claim inside of proposition: existence and uniqueness of bubbling down}. It also follows that the graded ${\sh O}_C$-algebra $\pi^*{\sh S}^{(d)}$ is generated by $\pi^*{\sh S}^{(d)}_1=\pi^*{\sh S}_d$ over ${\sh O}_C$. Projectivity of $\pi_\down$ follows from the claim at the beginning of this paragraph and \cite[Proposition 5.5.1]{[EGAII]}, as in \cite{[BM96]}.
		
		Unwinding the meaning of the definition of $h$ (using \cite[(3.7.1)]{[EGAII]}), we see that $h$ is equal to the composition
		\begin{equation}\label{equation: explicit writing of contraction as morphism associated to something}
			C = \mathrm{Proj}_C \bigoplus_{d|k}  {\sh L}^{\otimes k} \xrightarrow{j} \mathrm{Proj}_C \pi^* {\sh S}^{(d)} = C \times_S C_\down \to C_\down.
		\end{equation}
		The fact that the map denoted $j$ is defined everywhere (in fact, a closed immersion) follows from the surjectivity of \eqref{equation: map that will turn out surjective in proof of claim: various things about bubbling down over any base}, the fact that the graded ${\sh O}_C$-algebra $\pi^*{\sh S}^{(d)}$ is generated by $\pi^*{\sh S}_1^{(d)}$ over ${\sh O}_C$, and \cite[\href{https://stacks.math.columbia.edu/tag/07ZK}{Tag 07ZK}]{[stacks]}. Then, it is clear from \eqref{equation: explicit writing of contraction as morphism associated to something} that the formation of $h$ commutes with base change.
	\end{proof}
		
		By Claims \ref{claim: recovering bubbling down in elementary case} and \ref{claim: various things about bubbling down over any base}, $\pi_\down:C_\down \to S$ is a prestable curve (recall that $Y$ is nodal by construction), and, using Lemma \ref{remark: checking rational contraction of fibers} as well, $h$ has property R.
		
		It is straightforward to construct the rest of the data in the bubbling down. For the sections of $\pi_\down$, we compose with $h$, i.e. $x_\down = hx$, etc., then we simply define $\phi_\down$ to be the image required by Definition \ref{definition: bubbling down}. It remains to check that the data defines indeed an object of ${\cat C}_{g,m,n}^{c-1}$, i.e. that
		$ \omega_{C_\down}(  w_{1,\down} + \cdots + w_{m,\down} + 2 x_{1,\down}+\cdots+ 2 x_{n,\down} + (c-2)x_{\down} ) $
		is $\pi_\down$-ample, and, if $c=3$, that $\pi_\down$ is smooth at $x_\down$ and $w_{i,\down} \cap x_\down = \emptyset$ for all $i$. Both can be checked on geometric fibers, completing the proof of existence.
		
		It remains to prove uniqueness. First, we claim that $U:=h^{-1}(C_\down \backslash x_\down) = C \backslash h^{-1}(h(x))$ is independent of the choice of bubbling down. This claim can obviously be checked on fibers. In fact, it can be checked even on geometric fibers, when it becomes elementary. Indeed, if $s \in S$ and $\overline{s} \to S$ is the corresponding geometric point, then the surjectivity of $C_{\overline{s}} \to C_s$ can be used to show that, set-theoretically, $h^{-1}_s(h_s(x_s))$ is the image of $h^{-1}_{\overline{s}}(h_{\overline{s}}(x_{\overline{s}}))$ under $C_{\overline{s}} \to C_s$. Second, by Lemma \ref{lemma: dense f-trivial} part \ref{item 3: lemma: dense f-trivial}, $h$ restricts to an isomorphism $h|_U:U \to C_\down \backslash x_\down$. Let 
		$$ {\obj o}'_\down = (S,C'_\down,\pi'_\down,\overline{w}'_\down,\overline{x}'_\down,x'_\down,\phi'_\down) \quad \text{with} \quad h':C \to C'_\down $$
		be another bubbling down of ${\obj o}$. By \cite[Lemma 2.2]{[BM96]}, there exists an $S$-morphism $r:C_\down \to C'_\down$ such that $rh=h'$. It is straightforward to check that $r$ must be an isomorphism on the geometric fibers, and hence simply an isomorphism by flatness (this trick is used frequently in \cite{[Kn83]}). The compatibility of $r$ with all sections is elementary. In particular, if $U$ is as above, we have a commutative diagram
		\begin{center}
			\begin{tikzpicture}[scale = 1]
				\node (lu) at (2,1.6) {$U$};
				\node (ru) at (6,1.6) {$C$};
				\node (lc) at (0.3,0.8) {$C_\down \backslash x_\down$};
				\node (rc) at (4.3,0.8) {$C_\down$};
				\node (ld) at (2,0) {$C_\down' \backslash x_\down'$};
				\node (rd) at (6,0) {$C_\down'$};
				
				\draw [right hook->] (lu) to (ru);
				
				\draw [right hook->] (ld) to (rd);
				
				\draw [->] (lu) --  (lc);
				\draw [->] (ru) -- node[right, midway] {$h'$} (rd);
				\draw [->] (ru) -- node[below, midway] {$h$} (rc);
				\draw [->] (rc) -- node[above, midway] {$r$} (rd);
				\draw [->] (lc) --  (ld);
				\draw [right hook->] (lc) to (rc);
				\filldraw[white] (2,0.8) circle (3pt);
				
				\draw [->] (lu) --  (ld);
			\end{tikzpicture}
		\end{center}
		in which the left triangular face consists of isomorphisms only. It remains to check that the isomorphism $r$ is compatible with $\phi_\down$ and $\phi'_\down$, that is, that $\phi_\down$ is mapped to $\phi'_\down$ by the isomorphism $\Gamma(C_\down \backslash x_\down,\omega^\vee_{C_\down/S}) \simeq \Gamma(C_\down' \backslash x_\down',\omega^\vee_{C'_\down/S'})$ induced by $r$. The commutative diagram above clearly induces a commutative diagram on the level of vector fields. The map
		$ \Gamma(C_\down',\omega^\vee_{C_\down'/S}) \to \Gamma(C_\down' \backslash x_\down',\omega^\vee_{C_\down'/S}) $ is injective by Lemma \ref{lemma: injective restrictions on prestable curves}.
		\begin{center}
			\begin{tikzpicture}[scale = 1]
				\node (lu) at (2,2) {$\Gamma(U,\omega^\vee_{C/S})$};
				\node (ru) at (6+1,2) {$\Gamma(C,\omega^\vee_{C/S})$};
				\node (lc) at (0.3-1,1) {$\Gamma(C_\down \backslash x_\down,\omega^\vee_{C_\down/S})$};
				\node (rc) at (4.3-1+1,1) {$\Gamma(C_\down,\omega^\vee_{C_\down/S})$};
				\node (ld) at (2,0) {$\Gamma(C_\down' \backslash x_\down',\omega^\vee_{C'_\down/S'})$};
				\node (rd) at (6+1,0) {$\Gamma(C_\down',\omega^\vee_{C'_\down/S'})$};
				
				\draw [<-] (lu) to (ru);
				
				\draw [<-left hook] (ld) to (rd);
				
				\draw [->] (lu) --  (lc);
				\draw [->] (ru) -- node[right, midway] {\eqref{equation: pushing forward logarithmic vector fields in general}} (rd);
				\draw [->] (ru) -- node[below, midway] {\eqref{equation: pushing forward logarithmic vector fields in general}} (rc);
				\draw [->] (rc) -- (rd);
				\draw [->] (lc) --  (ld);
				\draw [<-left hook] (lc) to (rc);
				\filldraw[white] (2,1) circle (3pt);
				
				\draw [->] (lu) --  (ld);
			\end{tikzpicture}
		\end{center}
		The claim then follows by a simple diagram chase. Uniqueness of the isomorphism is equivalent to triviality of all `automorphisms' of any given bubbling down, and, in fact, there are even no nontrivial automorphisms of $C_\down$ compatible with $h$.
	\end{proof}
	
	If ${\obj o}$ is an object of ${\cat C}_{g,m,n}^1$ (resp. ${\cat C}_{g,m,n}^2$), then ${\obj o}$ together with the morphism $f:C_\kukp \to C$ (resp. $f:C_\zuzp \to C$) provided by Construction \ref{construction: first bubbling up construction} (resp. Construction \ref{construction: second bubbling up construction}), satisfies all the requirements of Definition \ref{definition: bubbling down}. Note that besides minor verifications, this also relies on the quite nontrivial Lemma \ref{lemma: vector fields match in second bubbling up} (keeping Remark \ref{remark: alternative statement of condition 2} in mind). Hence Proposition \ref{proposition: existence and uniqueness of bubbling down} has the following important consequence. 
	
	\begin{corollary}\label{corollary: down of up}
		Let $g,m,n,c \geq 0$ be integers such that $2g+2n+m+c \geq 4$ and $c \in \{1,2\}$, and let ${\cat C}^c_{g,m,n} \xrightarrow{\up} {\cat C}^{c+1}_{g,m,n}$, such that, with the notation in \eqref{equation: sequence of up functors},
		\[  \up = \begin{cases}
			\kukp & \text{if $c=1$} \\
			\zuzp & \text{if $c=2$}. \\
		\end{cases} \]
		Then, the composition $\down \circ \up$ is naturally isomorphic to the identity on ${\cat C}^c_{g,m,n}$.
	\end{corollary}
	
	It remains to prove the other direction.
	
	\begin{proposition}\label{proposition: up of down}
		Let $g,m,n,c \geq 0$ be integers such that $2g+2n+m+c \geq 5$ and $c \in \{2,3\}$, and let ${\cat C}^{c-1}_{g,m,n} \xrightarrow{\up} {\cat C}^c_{g,m,n}$, such that, with the notation in \eqref{equation: sequence of up functors},
		\[  \up = \begin{cases}
			\kukp & \text{if $c=2$} \\
			\zuzp & \text{if $c=3$}. \\
		\end{cases} \]
		Then, the composition $\up \circ \down$ is naturally isomorphic to the identity on ${\cat C}_{g,m,n}^c$.
	\end{proposition}
	
	\begin{proof}
		We apply Proposition \ref{proposition: towards up of down = id from Knudsen} if $c=2$, respectively Proposition \ref{proposition: towards up of down = id} if $c=3$, with the current $(C_\down,\ldots)$, $(C,\ldots)$, and $h$ in the roles of $(C,\ldots)$, $(Y,\ldots)$, and $q$ from the respective propositions. In either case (Proposition \ref{proposition: towards up of down = id from Knudsen} \emph{or} Proposition \ref{proposition: towards up of down = id}), the first condition is automatic by construction, while the second condition follows from the fact that $\up$ and $\down$ are inverses when $S = \spec {\mathbb K}$, which is elementary (by Examples \ref{example: geometric Knudsen stabilization with sections of line bundles} and \ref{example: geometric inflating at nonsingular zero of a section}, Definition \ref{definition: bubbling down}, and the basic considerations on components to be contracted in the proof of Proposition \ref{proposition: existence and uniqueness of bubbling down}). 
	\end{proof}
	
	Corollary \ref{corollary: down of up} and Proposition \ref{proposition: up of down} complete the proof of Theorem \ref{theorem: the main universal curve theorem}, and hence also of Theorem \ref{theorem: universal curve theorem simplified statement}.
	
	\section{Constructing moduli spaces}\label{section: moduli spaces}
	
	We now turn our attention towards Theorems \ref{theorem: Pn bar as moduli space} and \ref{theorem: isotrivial degeneration theorem}.
	
	\subsection{Coresidues}\label{subsection: coresidues} The map to ${\mathbb A}^1$ that exhibits the space $X$ in Theorem \ref{theorem: isotrivial degeneration theorem} as a degeneration is obtained from the following construction, which is essentially the `dual' of taking residues at poles of $1$-forms (please see also Example \ref{example: coresidues in local coordinates} below). Recall that, if $\pi:C \to S$ is a prestable curve and $x:S \to C$ is any smooth section, then $x^*\Omega_{C/S} \cong x^*\omega_{C/S}$ since $\Omega_{C/S}$ and $\omega_{C/S}$ agree on the open where $\pi$ is smooth \cite[p. 163]{[Kn83]}, and $x^*{\sh O}_C(-x(S)) \cong x^*\Omega_{C/S} $, since $x(S)$ is a relative effective Cartier divisor on $C$ over $S$, and both sides are canonically the conormal sheaf of $x$. 
	
	\begin{definition}\label{definition: coresidue and anticoresidue morphisms} 
		Let $\pi:C \to S$ be a prestable curve, $x:S \to C$ a smooth section, and $\phi$ a homomorphism $\omega_{C/S} \to {\sh O}_S(-x)$. The element
		\begin{equation*}
			x^*{\sh O}_C(-x) \cong x^*\Omega_{C/S} \cong x^*\omega_{C/S} \xrightarrow{-x^*\phi} x^*{\sh O}_C(-x)
		\end{equation*}
		(note the sign) of $\operatorname{End}_{{\sh O}_S}(x^*{\sh O}_C(-x)) = \Gamma({\sh O}_S)$ induces a map $S \to \spec {\mathbb Z}[t]$, called the \emph{negative coresidue (NCR)} morphism of the datum. 
		
		In particular, if ${\obj o}=(S,C,\ldots)$ is an object of ${\cat V}^+_{0,1,n}$ or ${\cat C}^+_{0,1,n}$ (with notation as in Definition \ref{definition: categories of curves}), then the NCR morphism of ${\obj o}$ is the NCR morphism relative to the section $w_1:S \to C$, and $\phi$ regarded as a homomorphism $\omega_{C/S} \to {\sh O}_S(-w_1)$.
	\end{definition}
	
	Note that the formation of the NCR map commutes with base change. 
	
	\begin{example}\label{example: coresidues in local coordinates}
		If $p$ is a smooth point on a curve $C$ over ${\mathbb K}$, and $v$ is a vector field which vanishes at $p$, then,
		\begin{equation}\label{equation: vanishing vector field in local coordinates} v = \left( c_1x+c_2x^2+c_3x^3+\cdots \right) \frac{d}{dx} \quad \text{for some $c_1,c_2,c_3,\ldots \in {\mathbb K}$} \end{equation}
		in local coordinates $x$ near $p$. Then, the `coresidue' of $v$ at $p$ is $c_1$. Indeed, if we change coordinates to $u = g(x)$, where $g \in {\mathbb K}[[x]]$, $g(0) = 0$, $g'(0) \neq 0$, then a calculation shows that a formula similar to \eqref{equation: vanishing vector field in local coordinates} with the same $c_1$ (different $c_2,c_3,\ldots$) holds in the $u$ coordinate.  For instance, if $v = (7x+3x^3)\frac{d}{dx}$ and $u=2x$, then $v = (7u+\frac{3}{4}u^3)\frac{d}{du}$ since $\frac{d}{dx} = 2\frac{d}{du}$ (assuming $\mathrm{char}({\mathbb K}) \neq 2$). In this example, the coresidue at $p$ ($x=0$) is $7$ and the negative coresidue is $-7$.
	\end{example}
	
	\begin{theorem}\label{theorem: the main degeneration as a moduli space theorem}
		The category ${\cat V}_{0,1,n}$ admits a terminal object. The NCR morphism of the terminal object is a flat, projective, local complete intersection (lci), geometrically integral morphism of relative dimension $n-1$ to $\spec {\mathbb Z}[t]$.
	\end{theorem}
	
	The base of the terminal object in Theorem \ref{theorem: the main degeneration as a moduli space theorem} will turn out to be precisely the space of the degeneration of $\overline{L}_n$ to $\overline{P}_n$ from Theorem \ref{theorem: isotrivial degeneration theorem}. Theorem \ref{theorem: the main degeneration as a moduli space theorem} will be proved towards the end of \S\ref{subsection: a closer look at}.
	
	\subsection{Iterating bubbling up}\label{subsection: Iterating bubbling up: examples} To obtain interesting spaces from Constructions \ref{construction: first bubbling up construction} and \ref{construction: second bubbling up construction}, we usually proceed as follows: start with a simple setup, then do:
\begin{enumerate}
\item a base change to the universal curve at the last step;
\item one application of Knudsen stabilization with vector fields;
\item one application of Inflating at zero vector;
\end{enumerate}
and repeat. This is stated formally in Construction \ref{construction: main inductive construction} below.
	
	\begin{construction}\label{construction: main inductive construction}
		Given an object ${\obj v}_{n_0}$ of ${\cat V}^+_{g,m,n_0}$ (referred to as the \emph{initial data}), we construct inductively a sequence of objects as follows:
		\begin{enumerate}
			\item ${\obj c}_{n,1} = \varpi^*{\obj v}_n$ in ${\cat C}^{1,+}_{g,m,n}$ where $\varpi$ is the projection (denoted by $\pi$ in Definition \ref{definition: categories of curves}) in ${\obj v}_n$;
			\item\label{item: item 2 in construction: main inductive construction} ${\obj c}_{n,2} = \kukp{\obj c}_{n,1}$ in ${\cat C}^{2,+}_{g,m,n}$;
			\item\label{item: item 3 in construction: main inductive construction} ${\obj c}_{n,3} = \zuzp{\obj c}_{n,2}$ in ${\cat C}^{3,+}_{g,m,n}$; and
			\item ${\obj v}_{n+1}$ corresponds to ${\obj c}_{n,3}$ under ${\cat V}^{+}_{g,m,n+1} \cong {\cat C}_{g,m,n}^{3,+}$.
		\end{enumerate}
		We refer to ${\obj v}_n$ as the $n$th object, to its base as the $n$th space, etc.
	\end{construction}
	Thus, the only choice is the choice of the initial data. Note also that, if any of the objects in the sequence belongs to ${\cat C}^{c}_{g,m,n}$ or ${\cat V}_{g,m,n}$ (that is, we can drop the `+'), then we can always drop the `+' thereafter.
	
	Below, we discuss four examples of what can be obtained by applying Construction \ref{construction: main inductive construction}, all four of which are relevant to Theorems \ref{theorem: Pn bar as moduli space} or \ref{theorem: isotrivial degeneration theorem}. Two other examples not relevant to our main results will be sketched in \S\ref{subsection: positive genus examples}. As agreed in \S\ref{subsection: Configurations on a line modulo scaling or translation}, $\overline{L}_n$ denotes the `Losev-Manin space over $\spec {\mathbb Z}$', that is, the space obtained by repeating the construction in \cite[1.3 and 2.1]{[LM00]} over $\spec {\mathbb Z}$. (Alternatively, $\overline{L}_n$ can be obtained as a blowup, similarly to $\overline{L}_{n,{\mathbb C}}$, though for our purposes, it is quicker and more convenient to take the inductive construction as the definition.) The first two examples below are artificial, but they are meant to prepare the more interesting examples that follow.
	
	\subsubsection{The Losev--Manin space: first artificial construction}\label{example: examples with initial data: Losev-Manin I} Run Construction \ref{construction: main inductive construction} with the following initial data in ${\cat V}_{0,2,1}$: 
	\begin{equation*}
		S=\spec {\mathbb Z}, \quad C={\mathbb P}^1_{\mathbb Z}, \quad w_1 = [1:0], \quad w_2=[0:1], \quad x_1=[1:1], \quad \phi = x\frac{d}{dx},
	\end{equation*}
	in the chart $Y \neq 0$, where $x = X/Y$ and $[X:Y]$ are the projective coordinates on ${\mathbb P}^1_{\mathbb Z}$. The resulting $n$th space is $\overline{L}_n$. This is clear because, at all steps, the vector field will only vanish at $w_1$ and $w_2$, so the application of Construction \ref{construction: second bubbling up construction} (that is, of step \ref{item: item 3 in construction: main inductive construction} in Construction \ref{construction: main inductive construction}) will never have any effect by \eqref{item: what if no degeneracy in remark: Remarks on construction of second bubbling up} in Remark \ref{remark: Remarks on construction of second bubbling up}. Thus we recover the inductive construction of $\overline{L}_n$.
	
	\subsubsection{The Losev--Manin space: second artificial construction}\label{example: examples with initial data: Losev-Manin II} Run Construction \ref{construction: main inductive construction} with the following initial data in ${\cat V}_{0,1,1}$: 
	\begin{equation*}
		S=\spec {\mathbb Z}, \quad C={\mathbb P}^1_{\mathbb Z}, \quad w_1 = [1:0], \quad x_1=[1:1], \quad \phi = x\frac{d}{dx},
	\end{equation*}
	with notation and conventions as in \ref{example: examples with initial data: Losev-Manin I}. We claim that by applying Construction \ref{construction: main inductive construction}, the resulting $n$th space will once again be $\overline{L}_n$. We briefly argue this inductively, comparing with \ref{example: examples with initial data: Losev-Manin I}. Assume inductively that this is the case for a given $n$, and that the object denoted by ${\obj c}_{n,1}$ in Construction \ref{construction: main inductive construction} has base $\overline{L}_{n+1}$ and prestable curve $\overline{L}_{n+1} \times_{\overline{L}_n} \overline{L}_{n+1} \to \overline{L}_{n+1}$, just like in the situation of \ref{example: examples with initial data: Losev-Manin I}. The only difference between these objects is that, in \ref{example: examples with initial data: Losev-Manin II}, the marking $w_2$ is not given, and that point is just a nameless isolated vanishing point of the vector field, the only one besides the marking $w_1$. We claim that after applying Constructions \ref{construction: first bubbling up construction} and \ref{construction: second bubbling up construction} we will obtain essentially the same objects, again with the sole difference that $w_2$ is not considered marked in \ref{example: examples with initial data: Losev-Manin II} (though after applying just Construction \ref{construction: first bubbling up construction}, the curves are truly different).
	\begin{figure}[h]
		\begin{center}
			\begin{tikzpicture}	
				\def\aa{0};
				\def\bb{-3};
				
				\draw (\aa,1) -- (\aa+1,2.5);
				\draw (\aa+0.5,2.5) -- (\aa+1.5,1); 
				\draw (\aa+1,1) -- (\aa+2,2.5);
				\filldraw[dashed, fill=white] (\aa + 0.2,1.3) circle (4pt) node[anchor=east] {$w_2$};
				\filldraw[fill=white] (\aa + 1.8,2.2) circle (4pt) node[anchor=west] {$w_1$};
				
				\fill (\aa + 0.2,1.3) circle (2pt);
				
				\draw [->] (\aa+1,0.7) -- (\aa+1,-0.3);
				
				\node at (\aa+2.2,0.4) {by $\kukp$ in \ref{example: examples with initial data: Losev-Manin I}};
				\node at (\aa+2.2,0) {by $\zuzp$ in \ref{example: examples with initial data: Losev-Manin II}};
				
				\def\aa{0.2}
				\draw (\aa,1+\bb) -- (\aa+1,2.5++\bb);
				\draw (\aa+0.5,2.5+\bb) -- (\aa+1.5,1+\bb); 
				\draw (\aa+1,1+\bb) -- (\aa+2,2.5+\bb);
				\filldraw[fill=white] (\aa + 1.8,2.2+\bb) circle (4pt) node[anchor=west] {$w_1$};
				\draw [dashed] (\aa+0.5-1,2.5+\bb) -- (\aa+1.5-1,1+\bb); 
				\filldraw[dashed, fill=white] (\aa - 0.3,2.2+\bb) circle (4pt) node[anchor=east] {$w_2$};
				
				\def\aa{3.5};
				\def\bb{-3};
				
				\draw (\aa,1) -- (\aa+1,2.5);
				\draw (\aa+0.5,2.5) -- (\aa+1.5,1); 
				\draw (\aa+1,1) -- (\aa+2,2.5);
				\filldraw[dashed, fill=white] (\aa + 0.2,1.3) circle (4pt) node[anchor=east] {$w_2$};
				\filldraw[fill=white] (\aa + 1.8,2.2) circle (4pt) node[anchor=west] {$w_1$};
				
				\fill (\aa + 1.8,2.2) circle (2pt);
				
				\draw [->] (\aa+1,0.7) -- (\aa+1,-0.3);
				
				\node at (\aa+1.5,0.2) {by $\kukp$};
				
				\def\aa{3.2};
				\draw (\aa,1+\bb) -- (\aa+1,2.5++\bb);
				\draw (\aa+0.5,2.5+\bb) -- (\aa+1.5,1+\bb); 
				\draw (\aa+1,1+\bb) -- (\aa+2,2.5+\bb);
				\filldraw[dashed, fill=white] (\aa + 0.2,1.3+\bb) circle (4pt) node[anchor=east] {$w_2$};
				\draw [dashed] (\aa+0.5+1,2.5+\bb) -- (\aa+1.5+1,1+\bb);
				\filldraw[fill=white] (\aa + 2.3,1.3+\bb) circle (4pt) node[anchor=west] {$w_1$};

				\def\aa{7};
				\def\bb{-3};
				
				\draw (\aa,1) -- (\aa+1,2.5);
				\draw (\aa+0.5,2.5) -- (\aa+1.5,1); 
				\draw (\aa+1,1) -- (\aa+2,2.5);
				\filldraw[dashed, fill=white] (\aa + 0.2,1.3) circle (4pt) node[anchor=east] {$w_2$};
				\filldraw[fill=white] (\aa + 1.8,2.2) circle (4pt) node[anchor=west] {$w_1$};
				
				\fill (\aa + 1.25,1.37) circle (2pt);
				
				\draw [->] (\aa+1,0.7) -- (\aa+1,-0.3);
				
				\node at (\aa+1.5,0.2) {by $\kukp$};
				
				\def\aa{7.2}
				\def\cc{0.5}
				\draw (\aa-\cc,1+\bb) -- (\aa+1-\cc,2.5++\bb);
				\draw (\aa+0.5-\cc,2.5+\bb) -- (\aa+1.5-\cc,1+\bb); 
				\draw (\aa+1+\cc,1+\bb) -- (\aa+2+\cc,2.5+\bb);
				\filldraw[fill=white] (\aa + 1.8+\cc,2.2+\bb) circle (4pt) node[anchor=west] {$w_1$};
				\filldraw[dashed, fill=white] (\aa + 0.2-\cc,1.3+\bb) circle (4pt) node[anchor=east] {$w_2$};
				\draw [dashed] (\aa+1.6+\cc,1.3+\bb) -- (\aa+0.5,1.3+\bb);
			\end{tikzpicture}
			\caption{Comparison between \ref{example: examples with initial data: Losev-Manin I} and \ref{example: examples with initial data: Losev-Manin II} (the $3$ nontrivial cases). The newly added marking is drawn as `$\bullet$'; the other $x$-markings are omitted in the picture. The point $w_2$ is a marking in \ref{example: examples with initial data: Losev-Manin I}, but only an unmarked vanishing point of the vector field in \ref{example: examples with initial data: Losev-Manin II}. In the situation on the left, although a ${\mathbb P}^1$ component is inserted at $w_2$ both in \ref{example: examples with initial data: Losev-Manin I} and \ref{example: examples with initial data: Losev-Manin II}, this happens in different ways: in the former case, by Construction \ref{construction: first bubbling up construction} (the degeneracy is inserting a new marking over a $w$ marking); in the latter case, by Construction \ref{construction: second bubbling up construction}  (the degeneracy is the vanishing of the field at the newly inserted marking).}
			\label{figure: comparison between artificial LM constructions}
		\end{center}
	\end{figure}
	Indeed, this may be checked on the open subset of the base where $x \neq w_1$ and $x$ is smooth, since both Construction \ref{construction: first bubbling up construction} and Construction \ref{construction: second bubbling up construction} agree in \ref{example: examples with initial data: Losev-Manin I} and \ref{example: examples with initial data: Losev-Manin II} over a neighbourhood of the complement, such as the open subset of the base where $x \neq w_2$. However, on the former open subset, Construction \ref{construction: second bubbling up construction} is trivial in \ref{example: examples with initial data: Losev-Manin I} and Construction \ref{construction: first bubbling up construction} is trivial in \ref{example: examples with initial data: Losev-Manin II}, so we just need to compare Construction \ref{construction: first bubbling up construction} in \ref{example: examples with initial data: Losev-Manin I} with Construction \ref{construction: second bubbling up construction} in \ref{example: examples with initial data: Losev-Manin II}. It is easy to check that the respective cokernels to be projectivized, cf. \eqref{equation: short exact sequence in first bubbling up} and \eqref{equation: initial exact sequence in bubbling up at nonsingular zero of a section}, agree up to a twist by an invertible sheaf, hence their projectivizations coincide. The other details are straightforward and skipped. In Figure \ref{figure: comparison between artificial LM constructions}, \ref{item: item 2 in construction: main inductive construction} and \ref{item: item 3 in construction: main inductive construction} refer to the items in Construction \ref{construction: main inductive construction}, i.e. Constructions \ref{construction: first bubbling up construction} and \ref{construction: second bubbling up construction}.
	
	\subsubsection{Compactifying the space of configurations modulo translation}\label{example: examples with initial data: Pn first announcement} Run Construction \ref{construction: main inductive construction} with the following initial data in ${\cat V}_{0,1,1}$: 
	\begin{equation*}
		S=\spec {\mathbb Z}, \quad C={\mathbb P}^1_{\mathbb Z}, \quad w_1 = [1:0], \quad x_1=[0:1], \quad \phi = \frac{d}{dx},
	\end{equation*}
	with the usual notation (so the vector field vanishes doubly at $w_1$, but nowhere else). We will prove in \S\ref{subsection: a closer look at} (see the proof of Theorem \ref{theorem: Pn bar as moduli space} towards the end of \S\ref{subsection: a closer look at}) that the resulting $n$th space is the moduli space $\overline{P}_n$ from \S\ref{subsection: Configurations on a line modulo scaling or translation}. For now, we take this explicit construction to be the definition of $\overline{P}_n$, and we will prove in \S\ref{subsection: a closer look at} that it represents the functor in Theorem \ref{theorem: Pn bar as moduli space}.
	
	\subsubsection{The first glimpse of the degeneration in Theorem \ref{theorem: isotrivial degeneration theorem}}\label{example: examples with initial data: isotrivial degeneration first announcement}
	Run Construction \ref{construction: main inductive construction} with the following initial data in ${\cat V}_{0,1,1}$: 
	\begin{equation}\label{equation: data in terminal object of V011}
		S=\spec {\mathbb Z}[t], \quad C={\mathbb P}^1_{{\mathbb Z}[t]}, \quad w_1 = [1:0], \quad x_1=[0:1], \quad \phi = (1+tx)\frac{\partial}{\partial x},
	\end{equation}
	with the usual notation, and $\pi$ is the projection ${\mathbb P}^1_{{\mathbb Z}[t]} \to \spec {\mathbb Z}[t]$. Throughout the application of Construction \ref{construction: main inductive construction}, the bases will retain the morphisms to $\spec {\mathbb Z}[t]$. 
	
	\begin{remark}\label{remark: main constuction has integral bases and curves}
		All objects that occur throughout the application of Construction \ref{construction: main inductive construction} in this situation have the property that their base (what is usually denoted by $S$) and their curve (what is usually denoted by $C$) are integral. This is straightforward to check inductively. Flatness of the curve over the base guarantees that all generic points of the curve map to generic points of the base. On the other, there exists a nonempty open $U \subset S$ such that $\pi^{-1}(U) \simeq {\mathbb P}^1_U$ over $U$, and Cohen-Macaulay-ness rules out embedded components.
	\end{remark}
	
	Construction \ref{construction: main inductive construction} commutes with base change, because Constructions \ref{construction: first bubbling up construction} and \ref{construction: second bubbling up construction} do, cf. Propositions \ref{proposition: first bubbling up is functorial} and \ref{proposition: second bubbling up is functorial}. Moreover, in our situation, it also commutes with multiplying the vector field with the pullback of an invertible regular function on the base by item \ref{item: scalar multiplication in remark: projective rational contraction in first type bubbling up} in Remark \ref{remark: projective rational contraction in first type bubbling up}, item \ref{item: item 2 in remark: scaling field in second bubbling up} in Remark \ref{remark: scaling field in second bubbling up}, and Remark \ref{remark: main constuction has integral bases and curves}. Then we can conclude the following regarding our current construction:
	\begin{enumerate}
		\item The fibers over $(t) \in \spec {\mathbb Z}[t]$ mimic \ref{example: examples with initial data: Pn first announcement}.
		\item\label{item: item 2 in example: examples with initial data: isotrivial degeneration first announcement} After the change of coordinates $X'=tX+Y$, $Y'=Y$ on ${\mathbb P}^1_{{\mathbb Z}[t,t^{-1}]}$, the preimages of $\spec {\mathbb Z}[t,t^{-1}] \subset \spec {\mathbb Z}[t]$ mimic the direct product of \ref{example: examples with initial data: Losev-Manin II} with $\spec {\mathbb Z}[t,t^{-1}]$, with the sole exception that $\phi$ is multiplied by $t$. (In the new coordinates, $\displaystyle \phi = tx'\frac{\partial}{\partial x'}$ where $x' = X'/Y'$.) In particular, if $X_n$ (temporary notation) is the $n$-th space obtained this way, then 
		\begin{equation}\label{equation: restriction above nonzero is trivially LM}
			\spec {\mathbb Z}[t,t^{-1}] \times_{\spec {\mathbb Z}[t]} X_n \simeq \spec {\mathbb Z}[t,t^{-1}] \times \overline{L}_n 
		\end{equation} 
		by \S\ref{example: examples with initial data: Losev-Manin II} and the considerations above.
	\end{enumerate}
	This hints at Theorem \ref{theorem: isotrivial degeneration theorem}. Indeed, we will see (essentially in Corollary \ref{corollary: down of up}) that the space $X_n$ above is $X$ from Theorem \ref{theorem: isotrivial degeneration theorem}.
	
	\subsection{The modular interpretation of the degeneration}\label{subsection: a closer look at} We continue the analysis of the construction in \ref{example: examples with initial data: isotrivial degeneration first announcement}.
	
	\begin{definition}\label{definition: construction of the terminal objects}
		Let ${\obj t}_n$ (respectively ${\obj t}_{n,c}$) be the objects of ${\cat V}_{0,1,n}$ (respectively ${\cat C}_{0,1,n}^c$) constructed in \ref{example: examples with initial data: isotrivial degeneration first announcement}. Let $\eta_n$ (respectively $\eta_{n,c}$) be the morphism from the base of ${\obj t}_n$ (respectively ${\obj t}_{n,c}$) to $\spec {\mathbb Z}[t]$ constructed in \ref{example: examples with initial data: isotrivial degeneration first announcement}.
	\end{definition}
	
	We will prove that ${\obj t}_n$, ${\obj t}_{n,c}$ are terminal objects. 
	
	For now, we state some geometric properties that follow from our construction. In the statement and proof of Proposition \ref{proposition: future terminal objects are nice} below, we agree to write $S[{\obj o}]$, $C[{\obj o}]$, etc. for the base, curve, etc. of an object ${\obj o}$ when there is possibility of confusion, respecting the letters used in Definitons \ref{definition: categories of curves} and \ref{definition: more categories of curves}. 
	
	\begin{proposition}\label{proposition: future terminal objects are nice}
		The following hold.
		\begin{enumerate}
			\item\label{item: first item in corollary: terminal objects are nice} Let ${\obj t}_*$ be either ${\obj t}_n$ or ${\obj t}_{n,c}$. Then the schemes $S[{\obj t}_*]$ and $C[{\obj t}_*]$ are normal and integral, as well as geometrically integral, separated, flat, of finite type, and lci over both $\spec {\mathbb Z}[t]$ and $\spec {\mathbb Z}$, and projective over $\spec {\mathbb Z}[t]$. The geometric generic fiber of $\pi[{\obj t}_*]$ is integral.
			\item\label{item: second item in corollary: terminal objects are nice} 	\begin{enumerate}
				\item\label{item: item 1 in corollary: terminal objects are nice} $x[{\obj t}_{n,1}] \neq w_1[{\obj t}_{n,1}]$ and $x[{\obj t}_{n,1}] \cap w_1[{\obj t}_{n,1}]$ is integral.
				\item\label{item: item 2 in corollary: terminal objects are nice} The scheme-theoretic vanishing locus of $x[{\obj t}_{n,2}]^*\phi[{\obj t}_{n,2}]$ (viewed as a global section of $ x[{\obj t}_{n,2}]^*\omega^\vee_{C[{\obj t}_{n,2}]/S[{\obj t}_{n,2}]} $) is an integral (prime) Cartier divisor on $S[{\obj t}_{n,2}]$.
			\end{enumerate}
			\item\label{item: third item in corollary: terminal objects are nice} The morphism $\eta_{n,c}$ from Definition \ref{definition: construction of the terminal objects} is the NCR morphism of ${\obj t}_{n,c}$, cf. Definition \ref{definition: coresidue and anticoresidue morphisms}, and a similar statement holds for $\eta_n$.
		\end{enumerate}
	\end{proposition}
	
	\begin{proof}
		\ref{item: first item in corollary: terminal objects are nice}: The claims that the bases and curves are integral has already been established in Remark \ref{remark: main constuction has integral bases and curves}. The other claims concerning (geometric) integrality follow using the same techniques. The flatness, projectivity, lci-ness, and normality claims can be proved comfortably by induction, following through the construction in \ref{example: examples with initial data: isotrivial degeneration first announcement}, though we quote the relevant results for the reader's convenience. 
		
		\emph{Projective:} Since projective in the sense of \cite{[EGAII]} is the same as projective in the sense of \cite{[Ha77]} if the base is itself quasi-projective over an affine scheme \cite[p. 103]{[Ha77]}, and composition of Hartshorne-projective morphisms is Hartshorne-projective, it follows inductively from these observations and Remark \ref{remark: projective rational contraction in first type bubbling up}, Constructions \ref{construction: first bubbling up construction} and \ref{construction: second bubbling up construction}, and item \ref{item 1a: proposition: all of second bubbling up elementary calculations} in Proposition \ref{proposition: all of second bubbling up elementary calculations} that `everything in sight' is Hartshorne-projective over $\spec {\mathbb Z}[t]$. 
		
		\emph{Local complete intersection:} Prestable curves are lci over their base \cite[Corollary 13.2.7]{[Ol16]}, compositions of lci morphisms are lci \cite[\href{https://stacks.math.columbia.edu/tag/069J}{Tag 069J}]{[stacks]}, and flat base changes of lci morphisms are lci \cite[\href{https://stacks.math.columbia.edu/tag/069I}{Tag 069I}]{[stacks]}.
		
		\emph{Normal:} We need to check the $R_1$ and $S_2$ conditions. First, $S_2$ is automatic: both the bases and the curves are flat and lci over $\spec {\mathbb Z}[t]$, hence Gorenstein over $\spec {\mathbb Z}[t]$  \cite[\href{https://stacks.math.columbia.edu/tag/0C15}{Tag 0C15}]{[stacks]}, hence Cohen-Macaulay over $\spec {\mathbb Z}[t]$ \cite[\href{https://stacks.math.columbia.edu/tag/0C06}{Tag 0C06}]{[stacks]}, hence absolutely (over $\spec {\mathbb Z}$) Cohen-Macaulay \cite[\href{https://stacks.math.columbia.edu/tag/0C0W}{Tag 0C0W}]{[stacks]}, hence satisfy all $S_k$ conditions \cite[\href{https://stacks.math.columbia.edu/tag/0342}{Tag 0342}]{[stacks]}. For $R_1$, if $\pi:C \to S$ denotes the prestable curve in question, consider the open subset $U \subset C$ defined as the intersection of $\pi^{\mathrm{sm}}$ with the preimage under $\pi$ of the open subset of $S$ where $S \to \spec {\mathbb Z}[t]$ is smooth. It is clear that $C$ is regular at all points of $U$, and moreover, $C \backslash U$ has codimension at least $2$, and the $R_1$ condition follows too.
		
		\ref{item: item 1 in corollary: terminal objects are nice}: Recall that $C[{\obj t}_{n,1}] = C[{\obj t}_n] \times_{S[{\obj t}_n]} C[{\obj t}_n]$, $x[{\obj t}_{n,1}]$ the diagonal of this fiber square, and $w_1[{\obj t}_{n,1}]$ is the pullback of $w_1[{\obj t}_n]$. Thus $x[{\obj t}_{n,1}] \cap w_1[{\obj t}_{n,1}] \simeq S[{\obj t}_n]$, which we know to be integral.
		
		\ref{item: item 2 in corollary: terminal objects are nice}: We write $\phi_0$ for $\phi$ regarded as a section of $\omega^\vee_{C/S}(-w_1)$ (rather than $\omega^\vee_{C/S}$) for clarity. Then, if we write $\{\cdots = 0\}$ for the vanishing loci, we have
		\begin{equation*}
			\begin{aligned}
				\{x[{\obj t}_{n,2}]^*\phi[{\obj t}_{n,2}] = 0\} &= \{x[{\obj t}_{n,2}]^*\phi_0[{\obj t}_{n,2}] = 0 \} \quad \text{since $x[{\obj t}_{n,2}] \cap w_1[{\obj t}_{n,2}] = \emptyset$} \\
				&= \{x[{\obj t}_{n,1}]^*\phi_0[{\obj t}_{n,1}] = 0\} \quad \text{by item \ref{item: vanishing locus item in remark: projective rational contraction in first type bubbling up} in Remark \ref{remark: projective rational contraction in first type bubbling up}} \\
				&\cong \{\phi_0[{\obj t}_n] = 0\} \quad \text{by construction,}
			\end{aligned}
		\end{equation*}
		so it suffices to check that the last one is integral. The restriction to the complement of the central fiber, $\{ \phi_0[{\obj t}_n]) = 0 \} \backslash C[{\obj t}_n]_{(t)}$, is a smooth section of the respective restriction of $\pi[{\obj t}_n]$, hence integral. Moreover, $\{ \phi_0[{\obj t}_n] = 0 \}$ doesn't contain $C[{\obj t}_n]_{(t)}$ (which is irreducible), thus $\{ \phi_0[{\obj t}_n] = 0 \}$ is irreducible and also reduced at all generic points, hence it is reduced everywhere since embedded components are ruled out by Cohen-Macaulay-ness. 
		
		\ref{item: third item in corollary: terminal objects are nice}: We proceed inductively. If $y=Y/X=1/x$, then
		\begin{equation*}
			\phi[{\obj t}_1] = (1 + tx)\frac{\partial}{\partial x} = \left(1 + tx \right)\frac{dy}{dx} \frac{\partial}{\partial y} = -\frac{1+tx}{x^2} \frac{\partial}{\partial y} = -y \left( t + y \right) \frac{\partial}{\partial y}.
		\end{equation*}
		Therefore, 
		\begin{equation*}
			w_1[{\obj t}_1]^*\left( \frac{\phi[{\obj t}_1]}{y} \right) = -t \frac{\partial}{\partial y},
		\end{equation*}
		so the NCR morphism of ${\obj t}_1$ is the identity. Inductively, we have the following. If $c=1$, then the NCR morphism doesn't change at the next step: the claim is clear over the dense open subset of the base where Construction \ref{construction: first bubbling up construction} is trivial (density is left to the reader), so it's true everywhere by reducedness. If $c=2$, then the NCR morphism doesn't change at the next step. This is actually obvious because, in Construction \ref{construction: second bubbling up construction}, the sections $w_1,\ldots,w_m$ are always contained in the open subset of the source where the operation is trivial. If $c=3$, then the NCR morphism at the next step simply pulls back together with the curve and the rest of the data, completing the proof.
	\end{proof}
	
	Finally, we prove Theorem \ref{theorem: the main degeneration as a moduli space theorem} inductively using Theorem \ref{theorem: the main universal curve theorem}. The only missing ingredient is the base case of the induction.
	
	\begin{lemma}\label{lemma: base case of induction in theorem: the main degeneration as a moduli space theorem}
		The object ${\obj t}_1$ (cf. Definition \ref{definition: construction of the terminal objects}) is a terminal object of ${\cat V}_{0,1,1}$. For any object ${\obj o} = (S,C,\ldots)$ of ${\cat V}_{0,1,1}$, the morphism $S \to \spec {\mathbb Z}[t]$ that exhibits ${\obj o}$ as a pullback of ${\obj t}_1$ is the NCR morphism of ${\obj o}$, cf. Definition \ref{definition: coresidue and anticoresidue morphisms}.
	\end{lemma}
	
	\begin{proof} In this proof, let ${\obj t}_1 = (S_0,C_0,\ldots)$, and let ${\obj o} = (S,C,\ldots)$ be an arbitrary object of ${\cat V}_{0,1,1}$. We claim that $C \simeq {\mathbb P}_S^1$ over $S$. This requires a few steps. 
		\begin{itemize}	
			\item All geometric fibers of $\pi$ are irreducible by a simple combinatorial argument using the ampleness of $\omega_{C/S}(w_1+2x_1)$. Otherwise, a leaf of the dual graph will correspond to a component which doesn't contain $x_1$, and this is destabilizing. 
			\item The morphism $\pi:C \to S$ is actually a projective ${\mathbb P}^1$-bundle since it is a geometrically integral conic bundle with a section. This is well-known, we leave as hint that $C \cong {\mathbb P}(\pi_*{\sh O}_C(x_1))$ in fact.
			\item The projective bundle $\pi:C \to S$ is the projectivization of the direct sum of two line bundles because $w_1$ and $x_1$ are disjoint sections.
			\item We have $C \simeq {\mathbb P}_S^1$ by the previous step and $x_1^*{\omega}^\vee_{C/S}$ trivial. 
		\end{itemize}
		
		If ${\obj o}$ is isomorphic to the pullback of ${\obj t}_1$ along a morphism $S \xrightarrow{\alpha} \spec {\mathbb Z}[t]$, then $\alpha$ must be the NCR morphism of ${\obj o}$ by the functoriality of NCR maps and item \ref{item: third item in corollary: terminal objects are nice} in Proposition \ref{proposition: future terminal objects are nice}. Conversely, to see that ${\obj o}$ is indeed isomorphic to the pullback of ${\obj t}_1$ along $\eta$, form the fiber product $S \times_{{\mathbb Z}[t]} C_0 \simeq {\mathbb P}_S^1$ relative to the NCR morphism $\eta:S \to \spec {\mathbb Z}[t]$ of ${\obj o}$ and $\pi_0:C_0 \to \spec {\mathbb Z}[t]$, choose an isomorphism
		$ C \simeq S \times_{{\mathbb Z}[t]} C_0 $ over $S$ 
		compatible with $x_1,w_1,x_{1,0},w_{1,0}$, and note that the remaining ${\sh O}_S^\times$-ambiguity is precisely what is needed to make the vector fields $\phi$ and $\eta^*\phi_0$ match, since they already both vanish on $w_1$ with equal `coresidues', and both are nowhere vanishing on $x_1$. Indeed, after composing with a suitable automorphism which fixes both $x_1$ and $w_1$, we may arrange so that $x_1^*\phi = x_1^*\eta^*\phi_0$, and then $\phi - \eta^*\phi_0$ will be a global section of $\omega_{C/S}^\vee(-2w_1-x_1)$, but 
		$ \pi_*\omega_{C/S}^\vee(-2w_1-x_1) = 0 $ 
		by the cohomology and base change theorem, completing the proof.
	\end{proof}
	
	\begin{theorem}\label{theorem: terminal object theorem most explicit version}
	The object ${\obj t}_n$ (respectively ${\obj t}_{n,c}$) from Definition \ref{definition: construction of the terminal objects} is a terminal object of ${\cat V}_{0,1,n}$ (respectively ${\cat C}_{0,1,n}$).
	\end{theorem}
	
	\begin{proof}
		Follows from Lemma \ref{lemma: base case of induction in theorem: the main degeneration as a moduli space theorem}, Theorem \ref{theorem: the main universal curve theorem} and the obvious fact that the pullback of a terminal object of ${\cat C}^3_{0,1,n-1} \simeq {\cat V}_{0,1,n}$ along the projection from the curve to the base is a terminal object of ${\cat C}_{0,1,n} \simeq {\cat C}^1_{0,1,n}$.
	\end{proof}
	
	Theorem \ref{theorem: the main degeneration as a moduli space theorem} follows from Theorem \ref{theorem: terminal object theorem most explicit version} and Proposition \ref{proposition: future terminal objects are nice}.  
	
	\begin{proof}[Proof of Theorem \ref{theorem: Pn bar as moduli space}]
	We will prove that $F$ is represented by the space $\overline{P}_n$ constructed in \ref{example: examples with initial data: Pn first announcement}. Indeed, $F$ corresponds to the subcategory of ${\cat V}_{0,1,n}$ whose objects have identically $0$ NCR morphisms. Since NCR morphisms commute with base change, it follows from Theorems \ref{theorem: the main degeneration as a moduli space theorem} and \ref{theorem: terminal object theorem most explicit version} that $F$ is represented by $\eta_n^{-1}(\spec {\mathbb Z})$, where $\eta_n: S[{\obj t}_n] \to \spec {\mathbb Z}[t]$ is the NCR morphism of ${\obj t}_n$, and $\spec {\mathbb Z} \hookrightarrow \spec {\mathbb Z}[t]$ by ${\mathbb Z}[t] \to {\mathbb Z}[t]/(t) = {\mathbb Z}$. By the remarks in \ref{example: examples with initial data: isotrivial degeneration first announcement} and Proposition \ref{proposition: future terminal objects are nice}, this is nothing but $\overline{P}_n$. The fact that $\overline{P}_n$ is projective, local complete intersection, flat, and geometrically integral over $\spec {\mathbb Z}$ follows from Proposition \ref{proposition: future terminal objects are nice}.
	\end{proof}

	\begin{corollary}\label{corollary: explicit statement that Ln degenerates to Pn}
		Let $X = S[{\obj t}_n]$ be the base of the terminal object ${\obj t}_n$ of ${\cat V}_{0,1,n}$ cf. Theorem \ref{theorem: terminal object theorem most explicit version}, and $\eta_n:X \to \spec {\mathbb Z}[t]$ its NCR morphism. Then,
		\[ \eta_n^{-1}(\spec {\mathbb Z}) \simeq \overline{P}_n \quad \text{and} \quad \eta_n^{-1}(\spec {\mathbb Z}[t,t^{-1}]) \simeq \spec {\mathbb Z}[t,t^{-1}] \times \overline{L}_n, \]
		the latter over $\spec {\mathbb Z}[t,t^{-1}]$, where $\spec {\mathbb Z} \hookrightarrow \spec {\mathbb Z}[t]$ by ${\mathbb Z}[t] \to {\mathbb Z}[t]/(t) = {\mathbb Z}$. 
	\end{corollary}
	
	\begin{proof}
		To review, (1) ${\obj t}_n$ is both the terminal object of ${\cat V}_{0,1,n}$ and the object constructed in \S\ref{example: examples with initial data: isotrivial degeneration first announcement} (Theorem \ref{theorem: terminal object theorem most explicit version}), (2) $\overline{P}_n$ is both the scheme which represents the functor $F$ and the space constructed in \S\ref{example: examples with initial data: Pn first announcement} (proof of Theorem \ref{theorem: Pn bar as moduli space} above), and (3) $\eta_n$ is both the NCR morphism of ${\obj t}_n$, and the morphism to $\spec {\mathbb Z}[t]$ which comes from the inductive construction in \S\ref{example: examples with initial data: isotrivial degeneration first announcement} (item \ref{item: third item in corollary: terminal objects are nice} in Proposition \ref{proposition: future terminal objects are nice}). Then, the fact that $\eta_n^{-1}(\spec {\mathbb Z}) \simeq \overline{P}_n$ comes directly from the proof of Theorem \ref{theorem: Pn bar as moduli space} above (the third and fourth sentences in the respective proof). The fact that $\eta_n^{-1}(\spec {\mathbb Z}[t,t^{-1}]) \simeq \spec {\mathbb Z}[t,t^{-1}] \times \overline{L}_n$ is simply \eqref{equation: restriction above nonzero is trivially LM}.
 	\end{proof}
		
	By Corollary \ref{corollary: explicit statement that Ln degenerates to Pn}, we know that $\overline{L}_n$ degenerates to $\overline{P}_n$. To complete the proof of Theorem \ref{theorem: isotrivial degeneration theorem}, it only remains to deal with the group actions.
	
	\section{Actions of $\ga$ and $\gm$ on curves and their moduli}\label{section: ga and gm actions on curves and their moduli}
	
	\subsection{Preliminaries} In this subsection, we review some generalities. These may be skipped and referred back to as necessary.
	
	\subsubsection{Review of Weil divisors and associated reflexive sheaves}\label{subsection: review of Weil divisors and associated reflexive sheaves} Let $X$ be an integral normal separated excellent scheme. For any Weil divisor $D$ on $X$, let ${\sh O}_X(D)$ be the sheaf associated to the divisor $D$, that is, 
	$ \Gamma(U,{\sh O}_X(D)) = \{f \in K(X): (\operatorname{div}(f)+D)|_U \geq 0\}$.  
	However, in other sections we will reserve the notation ${\sh O}_X(D)$ for the Cartier case, unless specified otherwise. We collect here some basic facts regarding associated sheaves. I have been unable to find a suitable published reference, but the excellent notes \cite{[Sch]} contain what is needed. 
	\begin{enumerate}
		\item\label{item mW: associated sheaf is reflexive} ${\sh O}_X(D)$ is a rank $1$ reflexive coherent ${\sh O}_X$-module \cite[Proposition 3.4]{[Sch]}; 
		\item\label{item mW: reflexive sheaf is associated} conversely, if ${\sh F}$ is a rank $1$ reflexive coherent ${\sh O}_X$-module, then there exists a Weil divisor $D$ such that ${\sh F} \simeq {\sh O}_X(D)$ \cite[Propositions 3.6 and 3.7]{[Sch]}; 
		\item\label{item mW: ideal sheaf} if $D$ is prime, then ${\sh I}_{D,X} = {\sh O}_X(-D)$ \cite[Proposition 3.4]{[Sch]};
		\item\label{item mW: dual ideal sheaf} if $D$ is prime, then ${\sh I}^\vee_{D,X} = {\sh O}_X(D)$, by \ref{item mW: associated sheaf is reflexive}, \ref{item mW: ideal sheaf}, and \cite[Proposition 3.13.(b)]{[Sch]};
		\item\label{item mW: End and Aut} if ${\sh F}$ is a rank $1$ reflexive coherent ${\sh O}_X$-module, then ${\sh E}nd({\sh F}) = {\sh O}_X$ by item \ref{item mW: reflexive sheaf is associated} and \cite[Proposition 3.13.(c)]{[Sch]}. In particular, ${\sh A}ut({\sh F}) = {\sh O}^\times_X$.
	\end{enumerate}
	
	\begin{lemma}\label{lemma: pullback for Weil divisors lazy}
		Let $X$ and $X'$ be integral normal separated excellent schemes, and $f:X' \to X$ a flat morphism. Let $D \subset X$ be a prime Weil divisor such that $D' = f^{-1}(D)$ is a prime Weil divisor on $X'$. Then $f^*{\sh O}_X(D) \simeq {\sh O}_{X'}(D')$.
	\end{lemma}
	
	\begin{proof} First, we claim that $f^*{\sh I}_{D,X} \simeq {\sh I}_{D',X'}$. This is a consequence of flatness, as follows. We have the following solid arrow commutative diagram
		\begin{center}
			\begin{tikzpicture}[scale=0.8]
				\node (a0) at (-1.7,0) {$0$};
				\node (a1) at (0,0) {${\sh I}_{D',X'}$};
				\node (a2) at (2.8,0) {${\sh O}_{X'}$};
				\node (a3) at (5.8,0) {$\iota'_*{\sh O}_{D'}$};
				\node (a4) at (7.5,0) {$0$};
				\node (b0) at (-1.7,1.4) {$0$};
				\node (b1) at (0,1.4) {$f^*{\sh I}_{D,X}$};
				\node (b2) at (2.8,1.4) {$f^*{\sh O}_X$}; 
				\node (b3) at (5.8,1.4) {$f^*\iota_*{\sh O}_D$}; 
				\node (b4) at (7.5,1.4) {$0$};
				\draw[->] (b0) to (b1);
				\draw[->] (b1) to (b2);
				\draw[->] (b2) to (b3);
				\draw[->] (b3) to (b4);
				\draw[->] (a0) to (a1);
				\draw[->] (a1) to (a2);
				\draw[->] (a2) to (a3);
				\draw[->] (a3) to (a4);
				\draw[<-, dashed] (a1) -- node [midway, right] {$\simeq$} (b1);
				\draw[<-] (a3) -- node [midway, right] {$\simeq$} (b3);	
				\draw[double equal sign distance] (a2) -- (b2);	
			\end{tikzpicture}
		\end{center}
		in which the top row is exact because $f$ is flat. The right vertical isomorphism comes from the flat cohomology and base change theorem \cite[Proposition 9.3]{[Ha77]}. Then there must exist a dashed arrow which makes the diagram commute, proving $f^*{\sh I}_{D,X} \simeq {\sh I}_{D',X'}$. However, as in the proof of \cite[Proposition 1.8]{[Ha80]}, $f^*$ commutes with dualizing, so $f^*{\sh I}^\vee_{D,X} \simeq {\sh I}^\vee_{D',X'}$. Then item \ref{item mW: dual ideal sheaf} completes the proof.
	\end{proof} 
	
	\subsubsection{Some calculations on $\overline{L}_n$}\label{subsubsection: some calculations on the Losev-Manin space over Z} The Losev-Manin space has already been reviewed in \S\ref{subsection: Configurations on a line modulo scaling or translation} and discussed in \S\ref{example: examples with initial data: Losev-Manin I} and \S\ref{example: examples with initial data: Losev-Manin II}. As agreed in \S\ref{subsection: Configurations on a line modulo scaling or translation}, $\overline{L}_n$ is the Losev-Manin space over $\spec {\mathbb Z}$.  \cite[Theorem 2.2]{[LM00]} still holds over $\spec {\mathbb Z}$, and $\overline{L}_{n+1}$ is still the universal curve over $\overline{L}_n$. In particular, $\overline{L}_n$ is smooth over $\spec {\mathbb Z}$. Indeed, it is clear inductively that it is flat and of finite type, since $\overline{L}_{n+1} \to \overline{L}_n$ is flat and of finite type, and the fact that the geometric fibers are regular is stated on page 446 of \cite{[LM00]} -- nothing changes in positive characteristic. (Since no details on smoothness are provided in \cite{[LM00]}, we mention an alternative argument: use the moduli space interpretation and deformation theory. The usual deformation theory of marked curves still applies.) 
	
	Moreover, there is a natural action of $\overline{L}_n \times {\mathbb G}_m$ on $\overline{L}_{n+1}$ over $\overline{L}_n$ -- on the open subset corresponding to Losev-Manin chains of length $1$, this is just the ${\mathbb G}_{m,L_n}$-action on ${\mathbb P}^1_{L_n}$ fixing $0$ and $\infty$. We denote it by $\nu:{\mathbb G}_m \times \overline{L}_{n+1} \to \overline{L}_{n+1}$. This also follows from \cite{[LM00]}, although it is worth noting that we will actually recover this as a byproduct of an inductive argument below (it will take one moment of thought to convince ourselves that there is no logical circularity here).
	
	Let $Y_n = \overline{L}_{n+1} \times_{\overline{L}_n} \overline{L}_{n+1}$. The action $\nu:{\mathbb G}_m \times \overline{L}_{n+1} \to \overline{L}_{n+1}$ pulls back along $Y_n \to \overline{L}_{n+1}$ to an action $\mu:{\mathbb G}_m \times Y_n \to Y_n$. Let $\Delta \subset Y_n$ be the diagonal of the fiber product, and $p_2:{\mathbb G}_m \times Y_n \to Y_n$ the projection to the second factor. 
	
	First, we show that $Y_n$ admits a small resolution (over $\spec {\mathbb Z}$).
	
	\begin{lemma}\label{lemma: small resolution of singular space related to Losev-Manin space}
		Let $R_n = {\mathbb P}{\sh I}^\vee_{\Delta,Y_n}$, and $\lambda: R_n \to Y_n$ the natural projection. Then $R_n$ is smooth over $\spec {\mathbb Z}$, and there exists a closed subscheme $N \subset Y_n$ such that:
		\begin{enumerate}
			\item $N$ has relative (over $\spec {\mathbb Z}$) codimension $3$ in $Y_n$;
			\item $\lambda^{-1}(N)$ has relative codimension $2$ in $R_n$; and
			\item $\lambda$ restricts to an isomorphism $R_n \backslash \lambda^{-1}(N) \simeq Y_n \backslash N$.
		\end{enumerate} 
	\end{lemma}
	
	\begin{proof}
		Roughly, the point is that $R_n$ resolves the singularities of $Y_n$ in the same way $\overline{L}_{n+2}$ does. Let $y_0,y_\infty: \overline{L}_{n+1} \to Y_n$ be $0$ and $\infty$ sections, i.e. the pullbacks of the $0$ and $\infty$ sections of $\overline{L}_{n+1}$ over $\overline{L}_n$. As reviewed earlier, the Losev-Manin spaces are constructed inductively in \cite[1.3 and 2.1]{[LM00]} by $\overline{L}_{n+2} = {\mathbb P}{\sh K}$, where 
		\begin{equation}
			{\sh K} = \operatorname{Coker}\left({\sh O}_{Y_n} \xrightarrow{b \mapsto (b,b)} {\sh I}^\vee_{\Delta,Y_n} \oplus {\sh O}_{Y_n}(y_0 + y_\infty) \right).
		\end{equation}
		Then $\overline{L}_{n+2}$ and $R_n$ are isomorphic above $Y_n \backslash (y_0 \cup y_\infty)$, and hence $R_n \to \spec {\mathbb Z}$ is smooth everywhere in $\lambda^{-1}(Y_n \backslash (y_0 \cup y_\infty))$. Let $N \subset \Delta \simeq \overline{L}_{n+1}$ correspond to the locus where $\overline{L}_{n+1} \to \overline{L}_n$ fails to be smooth, taken with, say, the reduced closed subscheme structure. It is clear that $\lambda$ is an isomorphism above $Y_n \backslash N$. Since $\overline{L}_{n+1} \to \spec {\mathbb Z}$ is smooth, $Y_n \to \spec {\mathbb Z}$ is clearly smooth at all points where at least one of the two projection maps $Y_n \to \overline{L}_{n+1}$ is smooth, and in particular, it is smooth in a neighbourhood of $y_0 \cup y_\infty$. Therefore, $R_n \to \spec {\mathbb Z}$ is also smooth in a neighborhood of $\lambda^{-1}(y_0) \cup \lambda^{-1}(y_\infty)$ as $N \cap (y_0 \cup y_\infty) = \emptyset$, completing the proof of the smoothness claim. The dimension claims are straightforward.
	\end{proof}
	
	\begin{proposition}\label{proposition: linear equivalence on singular space related to Losev-Manin space}
		$\mu^{-1}(\Delta) \sim p_2^{-1}(\Delta)$ on ${\mathbb G}_m \times Y_n$.
	\end{proposition}
	
	\begin{proof}
		Let $D = p_2^{-1}(\Delta)$ and $D' = \mu^{-1}(\Delta)$, and $E$ and $E'$ the `proper transforms' of $D$ and $D'$ on ${\mathbb G}_m \times R_n$, that is, with notation as in Lemma \ref{lemma: small resolution of singular space related to Losev-Manin space}, $E$ is the closure of $(\operatorname{id}_{{\mathbb G}_m} \times \lambda)^{-1}(D \backslash ({\mathbb G}_m \times N))$ with the reduced structure, and $E'$ is the closure of $(\operatorname{id}_{{\mathbb G}_m} \times \lambda)^{-1}(D' \backslash ({\mathbb G}_m \times N))$ with the reduced structure. Let $p_1:{\mathbb G}_m \times Y_n \to {\mathbb G}_m$ and $q_1:{\mathbb G}_m \times R_n \to {\mathbb G}_m$ be the projections to the first factors. Since ${\mathbb G}_m \times R_n$ is regular, $E$ and $E'$ are effective Cartier divisors.
		
		Let $z = (f,p) \in {\mathbb G}_m = \spec {\mathbb Z}[x,x^{-1}]$ be a closed point, where $p$ is a prime number, and $f \in {\mathbb Z}[x]$ is irreducible mod $p$, $f \not\equiv_p x$. First, we claim that
		\begin{equation}\label{equation: fiberwise linear equivalence on fiber square}
			D_z \sim D'_z
		\end{equation}
		on $p_{1}^{-1}(z)$. Let $K = {\mathbb F}_p[x]/(f)$, the residue field of ${\mathbb G}_m$ at $z$. As usual, write $\square_K=\spec K \times_{\spec {\mathbb Z}} \square$. Let $w \in {\mathbb G}_{m,K}(K)$ be a preimage of $z \in {\mathbb G}_m(K)$, and let $1$ denote the point $(t-1) \in {\mathbb G}_{m,K}(K)$. We have $D'_z = D'_{K,w} \sim D'_{K,1} = D_{K,1} = D_{K,w} = D_z$ in $p_1^{-1}(z) \cong Y_{n,K}$, where the linear equivalence comes from \cite[Proposition 1.6]{[Fu98]} (the family over ${\mathbb G}_{m,K}$ may be extended to a family over ${\mathbb P}_K^1$ simply by taking the closure of $D_K$ in ${\mathbb P}_K^1 \times Y_{n,K}$), which proves \eqref{equation: fiberwise linear equivalence on fiber square}. Second, we claim that
		\begin{equation}\label{equation: fiberwise linear equivalence on small resolution}
			E_z \sim E'_z
		\end{equation}
		on $q_1^{-1}(z)$. Note that $E_z$ and $E'_z$ are the proper transforms of $D_z$ and $D'_z$ under the small resolution $\lambda_K:q_1^{-1}(z) = R_{n,K} \to Y_{n,K} = p_{1}^{-1}(z)$. (Since $((\operatorname{id}_{{\mathbb G}_m} \times \lambda)^{-1}(D'))_z$ contains $E'_z$ and has a unique component of codimension $1$ in $({\mathbb G}_m \times R_n)_z$ as $D'_z$ is irreducible and $\operatorname{id}_{{\mathbb G}_{m,K}} \times \lambda_K$ is an isomorphism in codimension $2$, and $E'_z$ is cut out by a single equation (line bundle section) on $({\mathbb G}_m \times R_n)_z$ being the pullback of the Cartier divisor $E$, it follows that $E'_z$ is a fortiori the codimension $1$ irreducible component above, and also that it is the proper transform of $D'_z$. Similarly for $E_z$.) By Lemma \ref{lemma: small resolution of singular space related to Losev-Manin space} and \cite[Proposition 6.5, part b]{[Ha77]}, we have $\operatorname{Cl}(R_{n,K}) \cong \operatorname{Cl}(Y_{n,K})$ compatible with taking proper transforms, so \eqref{equation: fiberwise linear equivalence on small resolution} follows from \eqref{equation: fiberwise linear equivalence on fiber square}.  
		
		Let ${\sh J} = {\sh O}_{{\mathbb G}_m \times R_n}(E'-E)$. We will show that ${\sh J} \simeq {\sh O}_{{\mathbb G}_m \times R_n}$ using a standard argument. First, we claim that 
		\begin{equation}\label{equation: pushforward of difference sheaf is trivial} q_*{\sh J} \simeq {\sh O}_{{\mathbb G}_m}. \end{equation}
		By \eqref{equation: fiberwise linear equivalence on small resolution}, $\dim_{\kappa(z)} \Gamma({\sh J}_z) = 1$, for all closed points $z \in {\mathbb G}_m$. Since all closed and all open subsets of $|{\mathbb G}_m|$ contain at least one closed point, the semicontinuity theorem \cite[III, Theorem 12.8]{[Ha77]} implies that $\dim_{\kappa(z)} \Gamma({\sh J}_z) = 1$ holds for all $z \in {\mathbb G}_m$, not just the closed points. By Grauert's Theorem \cite[III, Corollary 12.9]{[Ha77]}, $q_*{\sh J}$ is invertible. Then \eqref{equation: pushforward of difference sheaf is trivial} follows as $\operatorname{Pic}({\mathbb G}_m)$ is trivial. However, $\dim_{\kappa(z)} \Gamma({\sh J}^\vee_z) = 1$ for all closed points $z \in {\mathbb G}_m$ follows equally well from \eqref{equation: fiberwise linear equivalence on small resolution}, so a completely analogous argument shows that 
		\begin{equation}\label{equation: pushforward of difference sheaf is trivial dual}  q_*({\sh J}^\vee) \simeq {\sh O}_{{\mathbb G}_m}. \end{equation}
		Let $s_1$ and $s_2$ be the global sections of ${\sh J}$ and ${\sh J}^\vee$ which correspond to the section $1$ of ${\sh O}_{{\mathbb G}_m}$ under the isomorphisms \eqref{equation: pushforward of difference sheaf is trivial} and \eqref{equation: pushforward of difference sheaf is trivial dual}. Clearly, $s_1 \otimes s_2 \neq 0$, so $s_1 \otimes s_2$ is nowhere vanishing as ${\sh J} \otimes {\sh J}^\vee \simeq {\sh O}_{{\mathbb G}_m \times R_n}$. Then $s_1$ is nowhere vanishing, so $ {\sh J} \simeq {\sh O}_{{\mathbb G}_m \times R_n}$.  
		Hence $E \sim E'$ on ${\mathbb G}_m \times R_n$. By Lemma \ref{lemma: small resolution of singular space related to Losev-Manin space} and \cite[Proposition 6.5, part b]{[Ha77]}, $\operatorname{Cl}({\mathbb G}_m \times R_n) \cong \operatorname{Cl}({\mathbb G}_m \times Y_n)$ compatible with taking proper transforms, and $D \sim D'$ follows.
	\end{proof}

	\subsubsection{Projectivization of equivariant sheaves} Finally, we review the fact that group actions lift to projectivizations of equivariant sheaves.
	
	\begin{lemma}\label{lemma: lifting actions}
		Let $G$ be an $S$-group scheme, $Y$ an $S$-scheme, and $\alpha$ an action of $G$ on $Y$ relative to $S$. Let ${\sh F}$ be a $G$-equivariant coherent ${\sh O}_Y$-module, $ X = {\mathbb P}{\sh F}$, and $f:X \to Y$ the natural projection map. Then there exists a $G$-action on $X$ over $S$, relative to which $f$ is $G$-equivariant.
	\end{lemma}
	
	\begin{proof}
		Given that ${\mathbb P}$ commutes with base change, this is purely formal and surely well-known. In essence, the discussion in \cite[p. 31]{[MFK94]} still applies -- it doesn't truly matter that we have ${\mathbb P} = \operatorname{Proj}\mathop{}\operatorname{Sym}$ instead of $\operatorname{Spec}\mathop{}\operatorname{Sym}$, and that \cite{[MFK94]} operates in a more restrictive setup (invertible sheaf, etc.).
	\end{proof}
	
	\subsection{The $G$-action on the universal curve}\label{subsection: the G-action on the universal curve} We return to the main logical thread of the paper, and deal with the last remaining aspect, the group actions. The main goal of \S\ref{subsection: the G-action on the universal curve} is to `integrate' the vector field on the universal curve in the terminal object of ${\cat V}_{0,1,n}$ to obtain a group action. 
	
Let $\gamma:G \to \spec {\mathbb Z}[t]$ as in \S\ref{subsection: Configurations on a line modulo scaling or translation}, $e:\spec {\mathbb Z}[t] \to G$ the identity section, and ${\mathfrak g} = e^*{\sh T}_{G/{\mathbb Z}[t]}$ its Lie algebra, where ${\sh T}_{G/{\mathbb Z}[t]}$ is the relative tangent bundle of $\gamma$. For simplicity, let $\tline = \spec {\mathbb Z}[t]$ and $\tline^* = \spec {\mathbb Z}[t,t^{-1}]$. We will use the notation $X[\varepsilon] = \spec {\mathbb Z}[\varepsilon]/(\varepsilon^2) \times_{\spec {\mathbb Z}} X$. Since ${\mathfrak g}$ is the normal sheaf of $\ell \xrightarrow{e} G$, $\frac{\partial}{\partial x} \in \Gamma({\mathfrak g})$ gives a first order thickening $j: \ell[\varepsilon] \to G$ of $e$ by standard deformation theory. 

Let ${\obj o} = (S,C,\ldots,\phi)$ be an object of ${\cat V}^+_{0,1,n}$ or ${\cat C}^+_{0,1,n}$, with notation as in Definition \ref{definition: categories of curves}. Consider the pullback $G \times_{\tline} S$ of $G$ along the NCR morphism $S \to \ell$. An action $\alpha$ of $G \times_{\tline} S$ on $C$ over $S$ induces an automorphism of $G \times_{\ell} C$ as follows: $\alpha: (G \times_\ell S) \times_S C = G \times _{\ell} C \to C$ and the projection to the first factor $ G \times _{\ell} C \to G$ induce a endomorphism of $G \times _{\ell} C$, which is easily checked to be an automorphism (over $G \times_{\ell} S$). This restricts to an automorphism of $C[\varepsilon]$ over $S[\varepsilon]$ (equal to  $\operatorname{id}_C$ on $C$), by $C[\varepsilon] = \ell[\varepsilon] \times_\ell C \subset G \times_\ell C$ via $j:\ell[\varepsilon] \to G$ and similarly for $S[\varepsilon]$. In turn, by standard deformation theory, this automorphism gives an element of $\operatorname{Hom}(\Omega_{C/S},{\sh O}_C)$. The discussion can be summarized in the following diagram.
\begin{center}
	\begin{tikzpicture}[scale = 0.8]
		\node (a) at (-3,2) {$\left\{\text{\begin{tabular}{c}
					actions of $G \times_{\tline} S$ \\ on $C$ over $S$
			\end{tabular}} \right\}$};
		\node (b) at (4,2) {$\left\{\text{\begin{tabular}{c}
					automorphisms of $C[\varepsilon]$ over $S[\varepsilon]$ \\ which restrict to $\operatorname{id}_C$ on $C$
			\end{tabular}} \right\}$};
		\node (c) at (4,0.5) {$\operatorname{Hom}(\Omega_{C/S},{\sh O}_C)$};
		\node (d) at (-1.5,0.5) {$\phi \in \operatorname{Hom}(\omega_{C/S},{\sh O}_C)$};
		\draw [->] (a) to (b);
		\draw [->] (b) to (c); 
		\draw [->] (d) to (c);
	\end{tikzpicture}
\end{center}
The lower horizontal map is obtained by dualizing the map $\Omega_{C/S} \to \omega_{C/S}$ from e.g. \cite[\S1]{[Kn83]}. Henceforth, $G \times_\tline S$, $G \times_\tline C$, etc. are relative to the NCR morphism $S \to \tline$.

\begin{definition}\label{definition: action compatible with vector field definition}
	In the situation above, we say that an action $\alpha$ of $G \times_{\tline} S$ on $C$ over $S$ is \emph{compatible with $\phi$} if the images of $\phi$ and $\alpha$ in $\operatorname{Hom}(\Omega_{C/S},{\sh O}_C)$ in the diagram above coincide. If only the restrictions of these images to an open $U \subset C$ coincide, we say the compatibility holds on $U$.
\end{definition}

\begin{theorem}\label{theorem: group actions on universal curve}
	An action compatible with the respective field exists for the terminal object ${\obj t}_n$ of ${\cat V}_{0,1,n}$, cf. Theorem \ref{theorem: terminal object theorem most explicit version}, and fixes the respective section $w_1$.
\end{theorem}
	
	We recall the following elementary fact which will be used soon.

	\begin{remark}\label{remark: Koszul complex of rank 2 geometrically}
		Let $D, E \subset X$ be effective Cartier divisors such that $D|_E = D \cap E$ is Cartier on $E$, that is, the restriction ${\sh O}_E \to ({\sh O}_X(D))|_E$ of ${\sh O}_X \to {\sh O}_X(D)$ is injective. Then we have a short exact sequence $0 \to {\sh I}_{D+E,X} \to {\sh I}_{D,X} \oplus {\sh I}_{E,X} \to {\sh I}_{D \cap E,X} \to 0$. Indeed, the elements cutting out $D$ and $E$ locally form a regular sequence, and then the exactness of the sequence is essentially the exactness of the Koszul complex.
	\end{remark}
	
	\begin{lemma}\label{lemma: G-action for first bubbling up}
		Let ${\obj o} = (S,C,\pi,w_1,\overline{x},x,\phi)$ be an object of ${\cat C}^1_{0,1,n}$ such that 
		\begin{enumerate}
			\item $C$ and $S$ are integral and separated;
			\item $x \neq w_1$, that is, there exists a point $s \in S$ such that $x(s) \neq w_1(s)$; and
			\item the scheme theoretic intersection $x \cap w_1$ is integral.
		\end{enumerate}   For simplicity, we write $w=w_1$. Let $\alpha$ be a $G \times_{\tline} S$-action compatible with $\phi$ which fixes $w$. Assume that ${\sh I}_{x,C}^\vee$ is $G \times_{\tline} S$-equivariant. Let $\kukp {\obj o} = (S,C_\kukp,\ldots)$ cf. \eqref{equation: sequence of up functors}. Then there exists an action $\alpha_\kukp$ of $G \times_{\tline} S$ on $C_\kukp$ compatible with $\phi_\kukp$, which fixes $w_{1,\kukp}$. 
	\end{lemma}
	
	Formally, the requirement that $\alpha$ fixes $w$ means that the restriction of $\alpha: G \times_S C \to C$ to $G \times_S w$ is the composition of the projection $G \times_S w \to w$ with the closed immersion $w \hookrightarrow C$. 
	
	\begin{proof}
		We will first show that if some twist ${\sh K} \otimes {\sh J}$ of ${\sh K}$ (cf. Construction \ref{construction: first bubbling up construction}) by a line bundle admits an equivariant structure relative to $\alpha$, then the conclusion holds. Indeed, Lemma \ref{lemma: lifting actions} then produces an action $\alpha_\kukp$ of $G \times_{\tline} S$ on $C_\kukp = {\mathbb P}({\sh K} \otimes {\sh J})$ such that $f:C_\kukp \to C$ (Construction \ref{construction: first bubbling up construction}) is $G \times_{\tline} S$-equivariant. The compatibility of $\alpha_\kukp$ with $\phi_\kukp$ holds at least over a dense open subset above which $f$ is an isomorphism, and then it holds everywhere a fortiori, since $\Omega^\vee_{C_\kukp/S}$ is torsion free. Moreover, $\alpha_\kukp$ fixes $w_\kukp \backslash x_\kukp$, and hence it must fix $w_\kukp$.
		
		Since the claim is local on the base, it suffices to analyze separately two situations: $x$ is contained in the open subset of $C$ where $\pi$ is smooth, and $x \cap w = \emptyset$. In the first case, $x$ is an effective Cartier divisor on $C$ and ${\sh I}^\vee_{x,C} = {\sh O}_C(x)$. We have a (solid arrow) commutative diagram with exact rows
		\begin{center}
			\begin{tikzpicture}[scale=1.2]
				\node (a0) at (-0.6,1) {$0$};
				\node (a1) at (1.2,1) {${\sh O}_C(-x-w)$};
				\node (a2) at (4.1,1) {${\sh O}_C(-x) \oplus {\sh O}_C(-w)$};
				\node (a3) at (6.8,1) {${\sh I}_{x \cap w,C}$};
				\node (a4) at (8.6,1) {$0$};
				\node (b0) at (-0.6,0) {$0$};
				\node (b1) at (1.2,0) {${\sh O}_C(-x-w)$};
				\node (b2) at (4.1,0) {${\sh O}_C(-x) \oplus {\sh O}_C(-w)$}; 
				\node (b3) at (6.8,0) {${\sh K}(-x-w)$}; 
				\node (b4) at (8.6,0) {$0$};
				\draw[->] (b0) to (b1);
				\draw[->] (b1) to (b2);
				\draw[->] (b2) to (b3);
				\draw[->] (b3) to (b4);
				\draw[->] (a0) to (a1);
				\draw[->] (a1) to (a2);
				\draw[->] (a2) to (a3);
				\draw[->] (a3) to (a4);
				\draw[double equal sign distance] (a1) -- (b1);
				\draw[->, dashed] (a3) -- node [midway, right] {$\simeq$} (b3);	
				\draw[double equal sign distance] (a2) -- (b2);	
			\end{tikzpicture}
		\end{center}
		in which the top row comes from Remark \ref{remark: Koszul complex of rank 2 geometrically}, and the bottom row is \eqref{equation: short exact sequence in first bubbling up} twisted by ${\sh O}_C(-x-w)$. Then there exists a unique dashed isomorphism ${\sh I}_{x \cap w,C} \simeq {\sh K}(-x-w) $ which makes the diagram commute. However, ${\sh I}_{x \cap w,C}$ is $G \times_{\tline} S$-equivariant because $\alpha$ fixes $w$ and hence $x \cap w$ too, and then ${\sh I}_{x \cap w,C}$ is the kernel of the homomorphism ${\sh O}_C \to (x \cap w \hookrightarrow C)_*{\sh O}_{x \cap w}$ in the category of $G \times_{\tline} S$-equivariant coherent ${\sh O}_C$-modules. In the second case, $ {\sh K}(-w) = {\sh I}_{x,C}^\vee $, and we are done since we are assuming that ${\sh I}_{x,C}^\vee$ is equivariant.
	\end{proof}
	
	\begin{lemma}\label{lemma: G-action for second bubbling up}
		Let ${\obj o} = (S,C,\pi,w_1,\overline{x},x,\phi)$ be an object of ${\cat C}^2_{0,1,n}$ such that
		\begin{enumerate}
			\item $C$ and $S$ are integral and separated; and
			\item if $Z = \{x^*\phi = 0\} \subseteq S$ is the scheme-theoretic vanishing locus of $x^*\phi$, then $Z$ is an integral (prime) effective Cartier divisor on $S$.
		\end{enumerate}
		Let $w=w_1$. Let $\alpha$ be a $G \times_{\tline} S$-action on $C$ compatible with $\phi$ which fixes $w$. Let $\zuzp {\obj o} = (S,C_\zuzp,\ldots)$, cf. \eqref{equation: sequence of up functors}. Then there exists a $G \times_{\tline} S$-action $\alpha_\zuzp$ on $C_\zuzp$ compatible with $\phi_\zuzp$, which fixes $w_{1,\zuzp}$.
	\end{lemma}
	
	\begin{proof}
		As in the proof of Lemma \ref{lemma: G-action for first bubbling up}, if some twist ${\sh K} \otimes {\sh J}$ of ${\sh K}$ (cf. Construction \ref{construction: second bubbling up construction}) by a line bundle admits an equivariant structure relative to $\alpha$, then we are done. We have a commutative diagram
		\begin{center}
			\begin{tikzpicture}[scale=1.1]
				\node (a0) at (-1.4,1) {$0$};
				\node (a1) at (0.3,1) {$\omega_{C/S}(-x)$};
				\node (a2) at (3.1,1) {${\sh O}_C(-x) \oplus \omega_{C/S}$};
				\node (a3) at (6,1) {${\sh I}_{x(Z),C}$};
				\node (a4) at (7.8,1) {$0$};
				\node (b0) at (-1.4,0) {$0$};
				\node (b1) at (0.3,0) {$\omega_{C/S}(-x)$};
				\node (b2) at (3.1,0) {${\sh O}_C(-x) \oplus \omega_{C/S}$}; 
				\node (b3) at (6,0) {${\sh K} \otimes \omega_{C/S}(-x)$}; 
				\node (b4) at (7.8,0) {$0$};
				\draw[->] (b0) to (b1);
				\draw[->] (b1) to (b2);
				\draw[->] (b2) to (b3);
				\draw[->] (b3) to (b4);
				\draw[->] (a0) to (a1);
				\draw[->] (a1) to (a2);
				\draw[->] (a2) to (a3);
				\draw[->] (a3) to (a4);
				\draw[double equal sign distance] (a1) -- (b1);
				\draw[->, dashed] (a3) -- node [midway, right] {$\simeq$} (b3);	
				\draw[double equal sign distance] (a2) -- (b2);	
			\end{tikzpicture}
		\end{center}
		in which the top row comes from Remark \ref{remark: Koszul complex of rank 2 geometrically}, and the bottom row is \eqref{equation: initial exact sequence in bubbling up at nonsingular zero of a section} twisted by $\omega_{C/S}(-x)$. It follows that ${\sh K} \otimes \omega_{C/S}(-x) \simeq {\sh I}_{x(Z),C}$. 
		
		We claim that $x(Z) \subset C$ is fixed by $\alpha$. Formally, this means that the restriction of $\alpha: G \times_\tline C \to C$ to $G \times_\tline x(Z)$ is the composition of the projection $G \times_\tline x(Z) \to x(Z)$ with the closed immersion $x(Z) \hookrightarrow C$. We are thus claiming that two $S$-morphisms $G \times_\tline x(Z) \to C$ coincide, or equivalently, that a morphism $G \times_\tline x(Z) \to C \times_S C$ factors through the diagonal immersion $C \hookrightarrow C \times_S C$. It is elementary to check on (geometric) fibers over $S$ that the image of the restriction of $\alpha$ to $G \times_\tline x(Z)$ is contained in $\pi^{\mathrm{sm}} \subseteq C$, the open subset where $\pi$ is smooth, and it follows that the image of the morphism $G \times_\tline x(Z) \to C \times_S C$ is contained in the open subset $\pi^{\mathrm{sm}} \times_S \pi^{\mathrm{sm}}$. However, the diagonal $\Delta$ of $\pi^{\mathrm{sm}} \times_S \pi^{\mathrm{sm}}$ is a Cartier divisor, and our claim boils down to the statement that the section $1$ of ${\sh O}_{\pi^{\mathrm{sm}} \times_S \pi^{\mathrm{sm}}}(\Delta)$ pulls back to $0$ on $G \times_\tline x(Z)$. However, the claim, and in particular the vanishing of the section above, are elementary to check on geometric fibers over $S$; then they hold on the fibers of $S$, and hence everywhere since $G \times_\tline x(Z)$ is integral as a consequence of the assumption that $x(Z)$ is integral. Then ${\sh I}_{x(Z),C} = \operatorname{Ker}({\sh O}_C \to x_*{\sh O}_Z)$ is equivariant (as in the proof of Lemma \ref{lemma: G-action for first bubbling up}), which completes the proof.
	\end{proof}
	
	To prove Theorem \ref{theorem: group actions on universal curve}, we need to show that the assumption that the dual ideal sheaf in Lemma \ref{lemma: G-action for first bubbling up} is equivariant holds for the terminal object of ${\cat C}^1_{0,1,n}$.
	
	\begin{proposition}\label{proposition: equivariance of dual ideal sheaf} 
		Let ${\obj t}_{n,1}=(S,C,\pi,w_1,\overline{x},x,\phi)$ be the terminal object of ${\cat C}^1_{0,1,n}$ cf. Definition \ref{definition: construction of the terminal objects} and Theorem \ref{theorem: terminal object theorem most explicit version}, and $\alpha$ an action of $G \times_{\tline} S$ on $C$ over $S$ compatible with $\phi$, cf. Definition \ref{definition: action compatible with vector field definition}. Then ${\sh I}_{x,C}^\vee$ is $\alpha$-equivariant. 
	\end{proposition}
	
	\begin{proof} Throughout this proof, we rely (sometimes implicitly) on Proposition \ref{proposition: future terminal objects are nice}. All of $C$, $G \times_\tline C$, and $G \times_\tline G \times_\tline C$ are dense open in some ${\mathbb A}^k \times C$, hence normal and integral by Proposition \ref{proposition: future terminal objects are nice}. Clearly, we may think of $\alpha$ as an action of $G$ on $C$ over $\tline$, in view of the identification $G \times_{\tline} C = (G \times_{\tline} S) \times_S C$. Let $\varpi_2: G \times_{\tline} C \to C$ be the projection to the second factor. First, we will show that
		\begin{equation}\label{equation: equation linear equivalence in proposition: equivariance of dual ideal sheaf}
			\alpha^{-1}(x) \sim \varpi_2^{-1}(x).
		\end{equation} 
		Let $G^* = \tline^* \times_{\tline} G \cong \tline^* \times {\mathbb G}_m$. With notation as in \S\ref{subsubsection: some calculations on the Losev-Manin space over Z}, $\tline^* \times_{\tline} C \cong \tline^* \times Y_n$ and this isomorphism restricts on $\tline^* \times_{\tline} x$ to an isomorphism $\tline^* \times_{\tline} x \cong \tline^* \times \Delta$. With $\mu$ also as in \S\ref{subsubsection: some calculations on the Losev-Manin space over Z}, we have the following commutative diagram.
		\begin{center}
			\begin{tikzpicture}
				\node (a1) at (0.5,0) {$C$};
				\node (a2) at (2.5,0) {$\tline^* \times_{\tline} C$};
				\node (a3) at (5,0) {$\tline^* \times Y_n$};
				\node (b1) at (0.5,1.2) {$G \times_{\tline} C$};
				\node (b2) at (2.5,1.2) {$G^* \times_{\tline} C$};
				\node (b3) at (5,1.2) {$\tline^* \times {\mathbb G}_m \times Y_n$};
				
				\draw [left hook->] (a2) -- (a1);
				\draw [left hook->] (b2) -- (b1);
				\draw [->] (a2) -- node [midway, above] {$\simeq$} (a3);
				\draw [->] (b2) -- node [midway, above] {$\simeq$} (b3);
				\draw [->] (b1) -- node [midway, left] {$\alpha$} (a1);
				\draw [->] (b2) -- (a2);
				\draw [->] (b3) -- node [midway, right] {$(\mathrm{id}_{\tline^*}, \mu)$} (a3);
			\end{tikzpicture}
		\end{center}
		(We briefly explain this claim. A priori, $\alpha$ restricts to some action of ${\mathbb G}_m \times (S \backslash S_{(t)})$ on $C \backslash C_{(t)} \simeq \tline^* \times Y_n$ over $S \backslash S_{(t)}$. If we restrict to, say, the fiber of $(t-1) \in \tline^*$, we recover the ${\mathbb G}_m \times \overline{L}_{n+1}$ action $\mu$ on $Y_n$ in \S\ref{subsubsection: some calculations on the Losev-Manin space over Z}. All the ${\mathbb G}_m$ actions in discussion satisfy various elementary properties (such as having weight $1$ on generic fibers and fixing the two suitable sections) that determine them uniquely, so there is no concern that we have obtained a different action. The same type of uniqueness argument establishes commutativity of the right half of the diagram.) Note that 
		$ \operatorname{Cl}(G^* \times_{\tline} C) \cong \operatorname{Cl}(G \times_{\tline} C) $
		by \cite[Proposition 6.5, part c]{[Ha77]} and $C_{(t)} \sim 0$ (as $C_{(t)}$ is integral by Proposition \ref{proposition: future terminal objects are nice}), so \eqref{equation: equation linear equivalence in proposition: equivariance of dual ideal sheaf} may be checked on the restriction to $G^* \times_\tline C$, i.e. the complement of the fiber over $(t) \in \ell$. Then it becomes 
		$ \tline^* \times \mu^{-1}(\Delta) \sim \tline^* \times p_2^{-1}(\Delta)$,  
		and it follows from Proposition \ref{proposition: linear equivalence on singular space related to Losev-Manin space} and the elementary fact that, for any (integral, normal, separated, noetherian) $X$, $\operatorname{Cl}(X) \cong \operatorname{Cl}(\tline^* \times X)$ via $[D] \mapsto [\tline^* \times D]$ by e.g. \cite[Propositions 6.5.c and 6.6]{[Ha77]} and their proofs. 
		
		Let's temporarily reinstate the ${\sh O}$ notation for associated sheaves from \S\ref{subsection: review of Weil divisors and associated reflexive sheaves}. We have
		${\sh O}_{G \times_{\tline} C} (\varpi_2^{-1}(x)) \simeq \varpi_2^*{\sh O}_C(x)$ and ${\sh O}_{G \times_{\tline} C} (\alpha^{-1}(x)) \simeq \alpha^*{\sh O}_C(x)$ by Lemma \ref{lemma: pullback for Weil divisors lazy}. On the other hand, $ {\sh O}_{G \times_{\tline} C} (\alpha^{-1}(x))  \simeq {\sh O}_{G \times_{\tline} C} (\varpi_2^{-1}(x)) $ by \eqref{equation: equation linear equivalence in proposition: equivariance of dual ideal sheaf}, so $\alpha^*{\sh O}_C(x) \simeq \varpi_2^*{\sh O}_C(x)$, or $\alpha^*{\sh I}^\vee_{x,C} \simeq \varpi_2^*{\sh I}^\vee_{x,C}$ by item \ref{item mW: dual ideal sheaf}. Let ${\sh F} = {\sh I}^\vee_{x,C}$ for simplicity, so 
		$ \alpha^*{\sh F} \simeq \varpi_2^*{\sh F}$. 
		We will see that what we've done so far suffices for proving Proposition \ref{proposition: equivariance of dual ideal sheaf}.
		
		First, we claim that we there exists an isomorphism $\psi: \alpha^*{\sh F} \simeq \varpi_2^*{\sh F}$ which is `unitary', that is, $(e_G \times_{\tline} \mathrm{id}_C)^*\phi = \mathrm{id}_{\sh F}$. Let $\psi_0: \alpha^*{\sh F} \simeq \varpi_2^*{\sh F}$ be an isomorphism. Note that $(e_G \times_{\tline} \mathrm{id}_C)^*\psi_0$ is an automorphism of ${\sh F}$, and then it is easy to check that if $\psi$ is the composition
		\begin{equation*}
			\alpha^*{\sh F} \xrightarrow{\psi_0} \varpi_2^*{\sh F} \xrightarrow{(\varpi_2^*(e_G \times_{\tline} \mathrm{id}_C)^*\psi_0)^{-1}} \varpi_2^*{\sh F},
		\end{equation*}
		then $\psi$ is unitary.
		
		We will show that $\psi$ actually satisfies the `cocycle condition' automatically. The cocycle condition is a statement that two isomorphisms between two sheaves isomorphic to $q_3^*{\sh F}$ coincide, where $q_3:G \times_{\tline} G \times_{\tline} C \to C$ is the projection to the third factor. However, since $\psi$ is unitary, these isomorphisms at least coincide over $B \times_{\tline} C$, where $B = G \times_{\tline} \{e_G\} \cup \{e_G\} \times_{\tline} G \subset G \times_{\tline} G$. Then it suffices to prove that the only automorphism of $q_3^*{\sh F}$ which restricts to the identity on $B \times_{\tline} X$ is the identity. Note that $q_3^*{\sh F}$ is reflexive by \cite[Proposition 1.8]{[Ha80]}. It follows from item \ref{item mW: End and Aut} in the review of Weil divisors (\S\ref{subsection: review of Weil divisors and associated reflexive sheaves}) that 
		${\sh A}ut(q_3^*{\sh F}) \cong {\sh O}^\times_{G \times_{\tline} G \times_{\tline} C}$. Let $\spec {\mathbb K} \to C$ be the geometric generic point of $X$. Because $C$ is integral and $(G \times_{\tline} G \times_{\tline} C)_{\mathbb K} \simeq {\mathbb G}^2_{m,{\mathbb K}}$, to conclude the proof, it suffices to check the following: the only (multiplicatively) invertible regular function on ${\mathbb G}^2_{m,{\mathbb K}}$ which restricts to the constant $1$ on ${\mathbb G}_{m,{\mathbb K}} \times \{1\} \cup \{1\} \times {\mathbb G}_{m,{\mathbb K}}$ is the constant $1$. Indeed, the invertible regular functions are all monomials, so the claim is clear.
	\end{proof}
	
	\begin{proof}[Proof of Theorem \ref{theorem: group actions on universal curve}]
		With the notation in Definition \ref{definition: construction of the terminal objects}, we will prove by induction that ${\obj t}_n$ and ${\obj t}_{n,c}$ admit group actions compatible with their vector fields in the sense of Definition \ref{definition: action compatible with vector field definition}. For ${\obj t}_1$, it can be checked that the standard action of $G$ on itself extends to an action on the relative compactification $G \hookrightarrow \ell \times {\mathbb P}^1$, and that the extended $G$-action is compatible with the vector field. Going from ${\obj t}_n$ (or equivalently, ${\obj t}_{n-1,3}$) to ${\obj t}_{n,1}$ is elementary, just pull back the action accordingly. Going from ${\obj t}_{n,1}$ to ${\obj t}_{n,2}$ is Lemma \ref{lemma: G-action for first bubbling up} backed up by Propositions \ref{proposition: equivariance of dual ideal sheaf} and \ref{proposition: future terminal objects are nice}, while going from ${\obj t}_{n,2}$ to ${\obj t}_{n,3}$ is Lemma \ref{lemma: G-action for second bubbling up} backed up by Proposition \ref{proposition: future terminal objects are nice}.
	\end{proof}
	
	\begin{remark}\label{remark: Pn bar as a moduli space low tech}
		It is elementary to see that there is a bijection between the set of stable $n$-marked ${\mathbb G}_a$-rational trees over ${\mathbb K}$ (\S\ref{subsection: Configurations on a line modulo scaling or translation}) up to isomorphism and the set of objects of ${\cat V}_{0,1,n}({\mathbb K})$ with $0$ NCR morphisms up to isomorphism. It is then obvious from Theorem \ref{theorem: Pn bar as moduli space} that the fibers over the closed points of $\overline{P}_{n,{\mathbb K}}$ are the set of stable $n$-marked ${\mathbb G}_a$-rational trees over ${\mathbb K}$.
	\end{remark}
	
	\subsection{The $G^n_\tline$-action on the moduli space} Finally, we make use of the modular interpretation to `transfer' the group actions on the curves in Theorem \ref{theorem: group actions on universal curve} to group actions on the moduli spaces, thereby completing the proof of Theorem \ref{theorem: isotrivial degeneration theorem}. Let ${\obj t}_n=(S,C,\pi,\overline{x},w_1,\phi)$ be the terminal object of ${\cat V}_{0,1,n}$, cf. Theorem \ref{theorem: terminal object theorem most explicit version} and Definition \ref{definition: construction of the terminal objects}.
	
	Let's construct the object ${\obj t}_n^\circ = (S^\circ,C^\circ,\ldots)$ of ${\cat V}_{0,1,n}$ which corresponds to the open stratum of the terminal object where the parametrized curves are integral:
	\begin{itemize}
		\item $S^\circ = G_{\tline}^{n-1}$, $C^\circ = G^{n-1}_{\tline} \times {\mathbb P}^1$, $\pi^\circ$ is the projection to the first factor;
		\item $w_1^\circ$ is the constant $[1:0]$ section;
		\item $x_1^\circ$ is the constant $[0:1]$ section, for $i \geq 2$, $x_i^\circ$ is the graph of the map 
		$$ G_\tline^{n-1} \xrightarrow{\text{projection to $(i-1)$-st factor}} G \to \tline \times {\mathbb P}^1 \to {\mathbb P}^1; $$
		\item $\displaystyle \phi^\circ = (1+tx) \frac{\partial}{\partial x}$, where $[X:Y]$ are the coordinates on ${\mathbb P}^1$ and $x = X/Y$.
	\end{itemize}
	The details of the verification that this is indeed the open stratum are skipped. There is a tautological $G_\tline^{n-1}$-action on $S^\circ$, but if we are to think of ${\obj t}_n^\circ$ more symmetrically, we should extend this to a $G_\tline^n$-action as follows: the composition 
	$$ G \xrightarrow{\operatorname{inverse}} G \xrightarrow{\operatorname{diagonal}} G_\tline^{n-1} $$ and the tautological action produce a $G$-action on $S^\circ$, and we may combine the latter action and the tautological action into a $G_\tline^n = G \times_\tline G^{n-1}_\tline$-action on $S^\circ$, called the natural action on $S^\circ$.
	
	\begin{proposition}\label{proposition: big group action on the big moduli space}
		With notation as above, there exists a $G^n_{\tline}$-action on $S$ over $\tline$ which extends the natural $G^n_{\tline}$-action on $S^\circ$. 
	\end{proposition}
	
	\begin{proof}
		Consider the object ${\obj y}_n$ of ${\cat V}_{0,1,n}$ which coincides with the pullback $p_2^*{\obj t}_n$, where $p_2:G^n_{\tline} \times_{\tline} S \to S$ is the projection to the second factor, with the sole exception that the $i$-th section (denoted by $x_i$ in Definition \ref{definition: categories of curves}) is the composition
		\begin{equation*}
			G^n_{\tline} \times_{\tline} S \xrightarrow{(p_1,x_i)} G^n_{\tline} \times_{\tline} C \xrightarrow{(q_1,\alpha_i)} G^n_{\tline} \times_{\tline} C,
		\end{equation*}
		where $p_1,q_1$ are the projections to the first factors respectively, and $\alpha_i$ uses the $i$-th factor of $G^n_{\tline}$ to act on $C$ via the action provided by Theorem \ref{theorem: group actions on universal curve}. Note that ${\obj y}_n$ is indeed an object in ${\cat V}_{0,1,n}$. By Theorem \ref{theorem: terminal object theorem most explicit version}, there exists a morphism
		$ \beta: G^n_{\tline} \times_{\tline} S \to S $ such that ${\obj y}_n = \beta^* {\obj t}_n$. 
		
		We only sketch the verification that $\beta$ is a $G^n_{\tline}$-action on $S$ over $\tline$. Let $\pr_i$ be the projection to the $i$th factor of $G_\tline^n \times_{\tline} G_\tline^n \times_{\tline} S$, and let
		$\mu:G \times_\tline G \to G$ and $\mu_\tline^n: G^n_{\tline} \times_{\tline} G^n_{\tline} \to G^n_{\tline}$
		be the group laws on $G$ and $G^n_{\tline}$. We need to check that
		\begin{equation}\label{equation: group action condition in proposition: big group action on the big moduli space}
			\beta \circ (\mu_\tline^n,\pr_3) = \beta \circ (\pr_1,\pr_1^*\beta)
		\end{equation}
		as morphisms $G_\tline^n \times_{\tline} G_\tline^n \times_{\tline} S \to S$. It is straightforward to check that $\beta$ restricts to the natural action on $S^\circ$. Then the restriction of \eqref{equation: group action condition in proposition: big group action on the big moduli space} to $G_\tline^n \times_{\tline} G_\tline^n \times_{\tline} S^\circ$ must hold, so \eqref{equation: group action condition in proposition: big group action on the big moduli space} must hold a fortiori since everything in sight is reduced and separated.
	\end{proof}
	
	Proposition \ref{proposition: big group action on the big moduli space} also shows that ${\overline P}_{n,{\mathbb K}}$ is a ${\mathbb G}_{a,{\mathbb K}}^{n-1}$-variety, as claimed in \S\ref{subsection: Configurations on a line modulo scaling or translation}.
	
	Theorem \ref{theorem: isotrivial degeneration theorem} follows from the examples in \S\ref{example: examples with initial data: Losev-Manin II}, \S\ref{example: examples with initial data: Pn first announcement}, and \S\ref{example: examples with initial data: isotrivial degeneration first announcement}, Theorem \ref{theorem: terminal object theorem most explicit version}, Theorem \ref{theorem: Pn bar as moduli space} (serving as the definition of $\overline{P}_n$), and Proposition \ref{proposition: big group action on the big moduli space}.
	
	\section{Concluding remarks and questions}\label{section: concluding remarks}
	
	\subsection{Further examples of iterative constructions}\label{subsection: positive genus examples} We briefly discuss two more applications of Construction \ref{construction: main inductive construction}, this time in strictly positive genus. Their details will be left to the interested reader: for clarity, \emph{anything not proved or not stated precisely in \S\ref{subsection: positive genus examples} is included in Figure \ref{figure: Yn} and Remarks \ref{remark: Yn} and \ref{remark: deforms to power of E}}.
	
	\subsubsection{An example in arbitrary genus}\label{example: examples with initial data: positive genus example} This example is very similar to \cite[Example 4.2.(e)]{[Wo15]}. Let $Y$ be a smooth, projective, connected curve over ${\mathbb K}$ of genus $g \geq 1$. Run Construction \ref{construction: main inductive construction} with the following initial data in ${\cat V}^+_{g,0,0}$: 
	\begin{equation*}
		S=\spec {\mathbb K}, \quad C=Y, \quad \phi = 0,
	\end{equation*}
	and denote the resulting $n$-th space by $Y\{n\}$.
	\begin{figure}[h]
		\begin{center}
			\begin{tikzpicture}[scale = 0.7]
				\draw [thick] (-2,0) -- (7,0) node[anchor=west] {$Y$}; 
				\fill[black] (3.6,1) circle (3pt);
				\draw[black] (3.6,1) circle (5pt);
				\draw[black] (3.6,1) circle (7pt);
				\draw [thick] (3.6,0.3) -- (3.6,2.3);
				\draw [thick] (4.3,0.3) -- (4.3,2.3);
				\fill[black] (4.3,0.6) circle (3pt);
				\fill[black] (4.3,1.4) circle (3pt);
				\draw[thick] (3.2,2) -- (6.7,2);
				\draw [thick] (5.3,-0.3) -- (5.3,2.3);
				\draw [thick] (5,1) -- (6.7,1);
				\fill[black] (6.2,1) circle (3pt);
				\draw [thick] (0.4,1) -- (2.8,1);
				\draw[thick] (0.4,2) -- (2.8,2);
				\draw [thick] (1.3,-0.3) -- (1.3,2.3);
				\fill[black] (2,2) circle (3pt);
				\fill[black] (2,1) circle (3pt);
				\draw[black] (2,1) circle (5pt);
				\draw [thick] (0,-0.3) -- (0,2.3);
				\fill[black] (0,1.4) circle (3pt);
				\draw [thick] (-1,-0.3) -- (-1,2.3);
				\fill[black] (-1,1) circle (3pt);
				\fill[black] (-1,2) circle (3pt);
				\draw[black] (-1,2) circle (5pt);
			\end{tikzpicture}
			\caption{An object parametrized by $Y\{13\}$. One can think of it as $4$ rational trees, in $\overline{P}_{3,{\mathbb K}}$, $\overline{P}_{1,{\mathbb K}}$, $\overline{P}_{3,{\mathbb K}}$, and $\overline{P}_{6,{\mathbb K}}$, attached to a copy of $Y$, on which the vector field is $0$.}
			\label{figure: Yn}
		\end{center}
	\end{figure}
	
	\begin{remark}\label{remark: Yn} It is possible to establish a modular interpretation of $Y\{n\}$, similar to that of $\overline{P}_n$ in Theorem \ref{theorem: Pn bar as moduli space}, using essentially the same methods, though we also need to pescribe a map to $Y$ satisfying suitable properties as part of the data, and suitably update the stability condition. (In the stability/ampleness condition in item \ref{item: item 3 in definition: more categories of curves} in Definition \ref{definition: more categories of curves}, we need to twist our line bundle by the pullback of an ample line bundle on $Y$.) As Figure \ref{figure: Yn} suggests, the relation between $\overline{P}_n$ and $Y\{n\}$ is reminiscent of the relation between $\overline{M}_{0,n+1}$ and the Fulton--MacPherson space $Y[n]$. However, $Y\{n\}$ only parametrizes reducible curves for $n \geq 1$, and $Y\{n\}$ is reducible of pure dimension $n$ for $n \geq 2$.
	\end{remark}
	
	\subsubsection{A feature of \ref{example: examples with initial data: positive genus example} in genus $1$}\label{example: examples with initial data: a genus 1 example}
	Let $E$ be a smooth genus $1$ curve over ${\mathbb K}$ with a nonzero vector field denoted rather abusively by $\frac{\partial}{\partial z}$. Run Construction \ref{construction: main inductive construction} with the following initial data in ${\cat V}^+_{1,0,0}$: 
	\begin{equation*}
		S=\spec {\mathbb K}[t] = {\mathbb A}^1, \quad C={\mathbb A}^1 \times E, \quad \phi = t\frac{\partial}{\partial z},
	\end{equation*}
	and denote the resulting $n$-th space by $W$. As in \ref{example: examples with initial data: isotrivial degeneration first announcement}, the bases retain projective flat morphisms to ${\mathbb A}^1$ throughout the application of Construction \ref{construction: main inductive construction}. Then:
	\begin{itemize}
		\item $W_0 \simeq E\{n\}$, cf. \S\ref{example: examples with initial data: positive genus example}.
		\item $W \backslash W_0 \simeq ({\mathbb A}^1 \backslash \{0\}) \times E^n$ over ${\mathbb A}^1 \backslash \{0\}$, parametrizing simply $n$-tuples of points on $E$ together with a nonzero vector field on $E$.
	\end{itemize}
	The reason for the second point is that, in the complement of the fiber over $0$, the applications of Constructions \ref{construction: first bubbling up construction} and \ref{construction: second bubbling up construction} inside Construction \ref{construction: main inductive construction} are always trivial -- the bubbling up will simply never get off the ground. 
	
	In conclusion, $E\{n\}$ is a degeneration of $E^n$.
	
	\begin{remark}\label{remark: deforms to power of E}
		Accepting the interpretation of $E\{n\}$ from Remark \ref{remark: Yn}, we may form the quotient $E\{n\}/S_n$ by permuting the markings, and we may even restrict to a fixed linear system $|{\sh L}|$ for the divisor on $E$ equal to the sum of the markings. Then, the resulting space $E\{ {\sh L} \}$ is a flat degeneration of ${\mathbb P}^{n-1}$.
	\end{remark}
	
	\subsection{Primitive linear systems on abelian (and K3) surfaces} The author's motivation to study the constructions in the current paper originally came from \cite[Question 4.14]{[Za21]}. Although the techniques developed here do not answer this question, they might hint at some progress, according to the speculations below.
	
	As in \cite{[Za21]}, consider a family $\{(V_t,{\sh L}_t):t \in \Delta\}$ of $(1,d)$-polarized abelian surfaces specializing over $0 \in \Delta$ to $V_0 = E \times F$, a product of two smooth genus $1$ curves, polarized by ${\sh L}_0 = {\sh J}_E \boxtimes {\sh J}_F$, with $\deg {\sh J}_E = d$ and $\deg {\sh J}_F = 1$. \cite[Question 4.14]{[Za21]} proposes to try to construct an alternate degeneration of the space of stable maps, in the style of the well-known work of J. Li \cite{[Li01]} or Kim, Kresh, and Oh \cite{[KKO14]}, by allowing $V_0$ to `expand' as in \cite[Figure A]{[Za21]}.
	
	Let us consider a similar task: to find an alternate degeneration ${\mathfrak L} \neq |{\sh L}_0|$ of the primitive linear systems $\{ |{\sh L}_t|: t \in \Delta, t \neq 0 \}$ based on this idea, by taking the limit divisors $D$ to live on expansions $W_0$ of $V_0$ (and imposing the `correct' conditions, which will not be discussed here). We expect an isomorphism 
	\begin{equation*} 
		{\mathfrak L} \simeq E \{ {\sh J}_E \}
	\end{equation*}
	given by $D \mapsto (\varpi^{-1}(p), D \cap \varpi^{-1}(p),v)$, where $\varpi: W_0 \to F$ is the projection map, $p \in F$ is such that ${\sh J}_F \not\cong {\sh O}_F(p)$ ($D \cap \varpi^{-1}(p)$ are the unordered markings on the curve $\varpi^{-1}(p)$), and the vector field $v$ is related to the ${\mathbb G}_a$-action on the ruled surface in \cite[\S3]{[Za21]}. Then ${\mathfrak L}$, and hence $E \{ {\sh J}_E \}$, needs to be a degeneration of ${\mathbb P}^{d-1}$, which is a very strong requirement. However, by Remark \ref{remark: deforms to power of E}, this very strong requirement is actually satisfied! Unfortunately, despite this, it is still not clear (to the author, at least) how to construct the desired degeneration, or if it even exists.
	
	\begin{question}\label{problem: degeneration of primitive linear system}
		Does an alternate degeneration of $|{\sh L}_t|$ ($t \neq 0$) with the features suggested above exist?
	\end{question}
	
	Similar ideas might also apply to K3 surfaces specializing to elliptically fibered K3s with $24$ nodal fibers, as in e.g. \cite{[BL00]}.
	
	\subsection{Other toric-to-${\mathbb G}_a^d$ degenerations} Recall the program in \cite{[HT99]} of classifying ${\mathbb G}_a^d$-varieties, and comparing the resulting picture with toric geometry. Although \cite{[HT99]} makes it clear that no direct analogy with toric geometry holds, Theorem \ref{theorem: isotrivial degeneration theorem} raises the following question.  
	
	\begin{question}\label{problem: question about other such degenerations}
		Which projective toric varieties of dimension $d$ degenerate isotrivially to ${\mathbb G}^d_a$-varieties in a manner compatible with the group actions? Classify all such degenerations.
	\end{question}
	
	For instance, the toric variety associated with the permutohedron (the Losev-Manin space) has the property in Problem \ref{problem: question about other such degenerations}, by Theorem \ref{theorem: isotrivial degeneration theorem}.

\end{document}